\documentclass[11pt,a4paper]{amsart}

\usepackage{kpfonts}
\usepackage{mathrsfs}
\usepackage{amsmath}
\usepackage{enumerate}
\usepackage{amssymb}
\usepackage{amscd}
\usepackage{amsthm}
\usepackage{amsfonts}
\usepackage{comment}
\usepackage{graphicx}
\usepackage{hyperref}
\usepackage[matrix, arrow, curve]{xy}
\usepackage{etoolbox}
\usepackage{mdframed}
\usepackage{setspace}
\usepackage{mathrsfs}  
\DeclareMathSymbol{\shortminus}{\mathbin}{AMSa}{"39}

\numberwithin{equation}{section}

\patchcmd{\subsection}{-.5em}{.5em}{}{}
\patchcmd{\subsubsection}{-.5em}{.5em}{}{}

\newcommand{\SL}{\operatorname{SL}}

\newcommand{\cB}{\mathcal{B}}
\newcommand{\cC}{\mathcal{C}}

\newcommand{\cF}{\mathcal{F}}

\newcommand{\cL}{\mathcal{L}}
\newcommand{\cM}{\mathcal{M}}

\newcommand{\cO}{\mathcal{O}}

\newcommand{\cS}{\mathcal{S}}

\newcommand{\bC}{\mathbb{C}}

\newcommand{\bN}{\mathbb{N}}

\newcommand{\bP}{\mathbb{P}}

\newcommand{\bR}{\mathbb{R}}

\newcommand{\bT}{\mathbb{T}}

\newcommand{\bZ}{\mathbb{Z}}

\newcommand{\R}{\mathbb{R}}
\newcommand{\C}{\mathbb{C}}

\newcommand{\ra}{\rightarrow}

\newcommand{\qand}{\quad \textrm{and} \quad}

\newcommand\subsetsim{\mathrel{%
\ooalign{\raise0.2ex\hbox{$\subset$}\cr\hidewidth\raise-0.8ex\hbox{\scalebox{0.9}{$\sim$}}\hidewidth\cr}}}
\newcommand{\eps}{\varepsilon}

\DeclareMathOperator{\pr}{pr}

\DeclareMathOperator{\supp}{supp}

\DeclareMathOperator{\Prob}{Prob}
\DeclareMathOperator{\Stab}{Stab}

\DeclareMathOperator{\Per}{Per}

\newcommand{\Q}{\mathbb Q}
\newcommand{\Z}{\mathbb Z}

\renewcommand{\epsilon}{\varepsilon}
\newcommand{\dd}{\;\mathrm{d}}

\DeclareMathOperator{\covol}{covol}
\DeclareMathOperator{\ind}{Ind}
\DeclareMathOperator{\Hom}{Hom}

\theoremstyle{theorem}
\newtheorem{theorem}{Theorem}[section]
\newtheorem{corollary}[theorem]{Corollary}
\newtheorem{proposition}[theorem]{Proposition}
\newtheorem{lemma}[theorem]{Lemma}
\newtheorem{mainthm}{Theorem}
\newtheorem{mainprop}[mainthm]{Proposition}

\theoremstyle{definition}
\newtheorem{definition}[theorem]{Definition}
\newtheorem{no}[theorem]{\S}

\newtheorem{construction}[theorem]{Construction}
\newtheorem{remark}[theorem]{Remark}
\newtheorem{example}[theorem]{Example}

\renewcommand{\phi}{\varphi}

\begin{document}

\title{Siegel-Radon transforms of transverse dynamical systems}

\author{Michael Bj\"orklund}
\address{Department of Mathematics, Chalmers, Gothenburg, Sweden}
\email{micbjo@chalmers.se}
\thanks{}

\author{Tobias Hartnick}
\address{Institut f\"ur Algebra und Geometrie, KIT, Karlsruhe, Germany}
\curraddr{}
\email{tobias.hartnick@kit.edu}
\thanks{}




\date{}

\dedicatory{}

\maketitle

\begin{abstract}
We extend Helgason’s classical definition of a generalized Radon transform, defined for a pair of
homogeneous spaces of an lcsc group $G$, to a broader setting in which one of the spaces is replaced by a possibly non-homogeneous dynamical system over $G$ together with a suitable cross section.
This general framework encompasses many examples studied in the literature, including Siegel 
(or $\Theta$-) transforms and Marklof–Strombergsson transforms in the geometry 
of numbers, Siegel–Veech transforms for translation surfaces, and Zak transforms in 
time-frequency analysis. 

Our main applications concern dynamical systems $(X, \mu)$ in which the cross section is induced from a separated cross section. We establish criteria for the boundedness, 
integrability, and square-integrability of the associated Siegel–Radon transforms, 
and show how these transforms can be used to embed induced \( G \)-representations into 
\( L^p(X, \mu) \) for appropriate values of \( p \). These results apply in particular to hulls of approximate lattices and certain ``thinnings'' thereof, including arbitrary positive density subsets in the amenable case.

In the special case of cut-and-project sets, we derive explicit formulas for the dual transforms, and in the special case of the Heisenberg group we provide isometric embedding of 
Schrödinger representations into the \( L^2 \)-space of the hulls of positive density subsets of approximate lattices in the Heisenberg group by means of aperiodic Zak transforms.
\end{abstract}


\section{Introduction}

The present work is motivated by spectral problems in the theory of aperiodic order and more specifically by the desire to perform spectral computations for approximate lattices in non-abelian groups (as in \cite{BHNilpotent, BHP1, BHP2, BHP3}). Its original goal was to generalize basic properties of Siegel-Radon transforms associated with lattices in locally compact second countable groups, such as the classical Siegel transform \cite{Siegel} and its generalizations (see e.g. \cite{KelmerYu1, KelmerYu2}) or the Zak transform from time frequency analysis (see e.g. \cite[Chapter 8]{Gr} and \cite[Chapter 5]{AT}) to ``strong approximate lattices'' in the sense of \cite{BHAL}.

However, it soon became apparent that the natural setting for the study of Siegel-Radon transforms is the wider framework of integrable transverse dynamical systems, parts of which were already outlined in \cite{BHK}. This framework, whose point-process perspective is close in spirit to \cite{Athreya}, is wide enough to also include Siegel-Radon transforms for certain ``thinnings'' of approximate lattices (including all positive density subsets of approximate lattices in amenable groups), as well as the more classical Siegel-Veech transform \cite{Veech} in the context of translation surfaces and the Siegel transform of Marklof-Str\"ombergsson \cite{MS,MS2} in the context of quasicrystals. In this article we will thus
\begin{itemize}
\item develop the theory of Siegel-Radon transforms in its natural habitat of integrable transverse dynamical systems;
\item explain how various classical examples of generalized Siegel transforms fit into this framework;
\item introduce a special class of integrable tranverse systems which are ``induced from separated cross sections'' and analyze their integrability properties;
\item use this to construct new examples of Siegel-Radon transforms and discuss some applications to aperiodic order.
\end{itemize}
Classical Siegel transforms have been used extensively for different counting problems in Diophantine approximation and geometry of numbers, see e.g. 
\cite{AM, BG, Kelmer, Schmidt, SS}, and our new framework allows to consider similar counting problems concerning orbits of approximate lattices in homogeneous spaces. We plan to return to this in future work.

\subsection{Integrable transverse systems and intersection spaces}
Throughout this article, $G$ denotes a unimodular lcsc group, $H<G$ denotes a closed unimodular subgroup and $\pi: G \to H \backslash G$ denotes the canonical projection. We fix Haar measures $m_G$ and $m_H$ for $H$ and $G$ and denote by $m_{H\backslash G}$ the corresponding $G$-invariant measure on $H \backslash G$.

\begin{no} A Borel subset $Y \subset X$ of a standard Borel $G$-space $X$ is called a  \emph{$(G, H)$-cross section} if $G.Y = X$ and $H.Y = Y$. Given such a cross section $Y$ and $x \in X$, we denote by
\[
Y_x := \{Hg \in H \backslash G \mid g.x \in Y\} \subset H \backslash G.
\]
the associated \emph{hitting time set} as and define the  \emph{intersection space} of $(X,Y)$ with $(H\backslash G, \{H\})$ as 
\[
\widetilde{X} := \{(Hg, x) \in H \backslash G \times X \mid Hg \in Y_x\}
\]
and its \emph{lifted $(G,H)$-cross section} by $\widetilde{Y} := \{H\} \times Y \subset \widetilde{X}$. This yields a double fibration
\begin{equation}\label{DoubleFibration1}
\begin{xy}\xymatrix{
&(\widetilde{X}, \widetilde{Y}) \ar[rd]^{\pi_{X}}\ar[ld]_{\pi_{H \backslash G}}&\\
(H\backslash G, \{H\})&&(X,Y),
}\end{xy}
\end{equation}
whose fibers are given by
\[
\pi_{H \backslash G}^{-1}(Hg)=\{Hg\} \times g^{-1}.Y \cong Y \qand \pi_X^{-1}(x) = Y_x \times \{x\} \cong Y_x.
\]
Under favourable circumstances such a double fibration gives rise to a generalized Radon transform. A particularly rich and systematic theory was developed in the case
where $X$ is homogeneous; see in particular the work of Helgason \cite{Helgason1, Helgason2}. Here we will be interested in situations in which $X$ is possibly non-homogeneous, the fibers of $\pi_{H \backslash G}$ carry invariant probability measures and the fibers of $\pi_X$ are countable. 
\end{no}
\begin{no} Let $p \geq 1$. If $Y\subset X$ is a $(G,H)$-cross section and $\mu \in \mathrm{Prob}(X)^G$ is a $G$-invariant probability measure on $X$, then  $(X,Y, \mu)$ is called a \emph{$p$-integrable\footnote{In the language of point process theory, this means that the underlying point process is \emph{locally} $p$-integrable.}  transverse $(G,H)$-system} if \[\int_X |Y_x \cap L|^p \dd \mu(x) < \infty \quad \text{ for every compact subset } L \subset G.\] If $p=1$ (respectively $p=2$) we say that $(X, Y, \mu)$ is \emph{integrable} (respectively \emph{square-integrable}). Note that integrability implies that almost every $Y_x$ is locally finite, hence countable, and the following proposition provides a canonical family of invariant probability measures on the $\pi_{H \backslash G}$-fibers of the intersection space. Given a non-negative Borel function $F: G \times Y \to [0, \infty)$ we definite its \emph{$Y$-periodization}
\[
TF: X \to [0, \infty], \quad TF(x) := \sum_{Hg \in Y_x}  \int_H F((hg)^{-1}, hg.x) \dd m_H(h).
\]
We also write $\cL^\infty_{c}(G \times Y)$ for the space of bounded Borel functions on $G \times Y$ which vanish outsides a set of the form $L \times Y$ for some compact $L \subset G$.
\end{no}
\begin{mainprop}\label{PropTAdjoint} Let $(X, Y, \mu)$ be an integrable $(G,H)$-transverse system.
\begin{enumerate}[(i)]
\item For every $F \in  \cL^\infty_{c}(G \times Y)$ the $Y$-periodization $TF$ is integrable.
\item There exists a unique finite $H$-invariant measure $\sigma$ on $Y$ such that for all $F \in  \cL^\infty_{c}(G \times Y)$ and $\phi \in \cL^\infty(X, \mu)$ we have
\begin{equation}\label{TAdjoint}
\int_X TF(x) \overline{\phi(x)} \dd \mu(x) = \int_G \int_Y F(g,y) \overline{\phi(g.y)}  \dd \sigma(y) \dd m_G(g).
\end{equation}
\end{enumerate}
\end{mainprop}
\begin{no} In the situation of Proposition \ref{PropTAdjoint} we refer to $\sigma$ as the \emph{$H$-transverse measure} of $\mu$, and we define the associated \emph{intersection measure} on the intersection space $\widetilde{X}$ by
\[\int_{\widetilde{X}} F \dd \widetilde{\mu} := \int_{H \backslash G} \Big( \int_Y F(Hg,g^{-1}.y) \, d\sigma(y) \Big) \dd m_{H \backslash G}(Hg).
\]
The measure class of $\widetilde{\mu}$ then projects to the measure classes of $m_{H \backslash G}$ and $\mu$ respectively, and under the above identifications of the fibers with $Y$ and $Y_x$ the fiber measures are given by $\sigma$ and counting measure respectively. 
\end{no}
\subsection{Siegel-Radon transforms and their duals}
Given an integrable $(G, H)$-transverse system $(X, Y, \mu)$ one can use the double fibration \eqref{DoubleFibration1} to define transforms (in both directions) between suitable function classes on $H\backslash G$ and $X$. As in the classical theory of Radon transforms \cite{Helgason1, Helgason2}, these transforms are defined by pulling back a given function to the intersection space and then performing a fiber integration.
\begin{mainthm}\label{RadonUntwisted} Let $(X, Y, \mu)$ be an integrable $(G,H)$-transverse system with transverse measure $\sigma$.
\begin{enumerate}[(i)]
\item There are well-defined linear operators
\[
S: C_c(H \backslash G) \to L^1(X, \mu) \qand S^*: L^\infty(X, \mu) \to L^\infty(H\backslash G),
\]
given by pullback followed by fiber-integration in \eqref{DoubleFibration1}. Explicitly,
\[
Sf(x) = \sum_{Hg \in Y_x} f(Hg) \qand S^*\phi(Hg) = \int_Y \phi(g^{-1}.y) \dd \sigma(y).
\]
\item The operators $S$ and $S^*$ are formally adjoint in the sense that
\[\int_{X)} S f \cdot \overline{\phi} \dd  \mu  = \int_{H \backslash G} f \cdot \overline{S^*\phi} \dd m_{H \backslash G} \quad (f \in C_c(H \backslash G), \phi \in L^\infty(X, \mu)).\]
\end{enumerate}
\end{mainthm}
\begin{no}\label{SiegelFormulaConstant}
We refer to the operator $S$ as the \emph{Siegel-Radon transform} of the integrable $(G,H)$-transverse system $(X, Y, \mu)$, and to the statement in (ii) as \emph{Siegel duality}. By choosing $\phi \equiv 1$ this specializes to the \emph{Siegel formula} (cf.\  \cite[Section 2]{Siegel})
\begin{equation}\label{SiegelFormulaIntro}
\int Sf \dd \mu = \sigma(Y) \cdot \int_{H \backslash G} f \dd m_{H \backslash G} \quad (f \in C_c(H \backslash G)).
\end{equation}
In view of this formula we refer to $\sigma(Y)$ as the \emph{Siegel constant} of $(X,Y, \mu)$. One can show that the Siegel transform takes values in $L^p(X, \mu)$ if and only if $(X, Y, \mu)$ is $p$-integrable.
\end{no}
\subsection{Twisted Siegel-Radon duality}
\begin{no} In applications in representation theory and harmonic analysis, versions of the Siegel transform which are twisted by a unitary $H$-representation often play an important role. For simplicity we only deal with the case of characters here, leaving the general case for future work. Thus let $(X, Y, \mu)$ be integrable and fix a unitary character $\xi: H \to \bT$. Inside the induced representation $\mathrm{Ind}_H^G(\xi)$ we consider the space of test functions
\[
{}^0\mathrm{Ind}_H^G(\xi) := \{f \in C(G) \mid f(hg) = \xi(h) f(g) \; \text{for all }h \in H, g \in G  \; \text{and}\; |f| \in C_c(H \backslash G)\},
\]
If $\psi_\xi: Y \to \bC$ is a non-zero Borel function such that 
 \begin{equation}\label{Eigenfunction}
\psi_\xi(h^{-1}.y) = \xi(h)\psi_\xi(y) \quad \text{for all $h \in H$ and $y \in Y$},
\end{equation} 
then $\psi_\xi$ is called a  \emph{(strict\footnote{Below we will work with a more general notion of almost everywhere defined $H$-eigenfunctions.})
$H$-eigenfunction} and $\xi$ is called an  
 \emph{eigencharacter} of $Y$. 
\end{no}
\begin{mainthm}\label{RadonTwisted} Let  $(X, Y, \mu)$ be an integrable $(G, H)$-transverse system and let $\psi_\xi$ be an $H$-eigenfunction of $Y$ with eigencharacter $\xi$.
\begin{enumerate}[(i)]
\item There is a well-defined $G$-equivariant \emph{$\psi_\xi$-twisted Siegel-Radon transform} 
\[
S_{\psi_\xi}: {}^0\mathrm{Ind}_H^G(\xi) \to L^1(X, \mu), \quad S_{\psi_\xi}f(x) := \sum_{Hg \in Y_x} f(g)\psi_\xi(g.x).
\]
\item If $\phi \in \cL^\infty(X)$, then $S_{\psi_\xi}^*\phi(g) := \int_Y \phi(g^{-1}.y) \overline{\psi_\xi(y)} d\sigma(y)$ is a bounded Borel function on $G$.
\item If $f \in {}^0\mathrm{Ind}_H^G(\xi)$ and $\phi \in \cL^\infty(X)$, then the product $f \cdot \overline{S^*_{\psi_\xi}\phi}$ descends to an integrable function on $H \backslash G$ and
\[\int_{X} S_{\xi} f \cdot \overline{\phi} \dd  \mu  = \int_{H \backslash G} f \cdot \overline{S^*_{\xi}\phi} \dd m_{H \backslash G}.\]
\end{enumerate}
\end{mainthm}
\begin{no}\label{SiegelIntegrability}
It turns out that the $\psi_\xi$-twisted Siegel-Radon transform $S_{\psi_\xi}$ depends only on the equivalence class of $\psi_\xi$ in $L^\infty(Y_{\Lambda})$; it reduces to the untwisted Siegel-Radon transform in the case $\psi_\xi \equiv 1$ and takes values in $L^p(X, \mu)$ if $(X, Y, \mu)$ is $p$-integrable. If $\xi$ is a \emph{non-trivial} eigencharacter of $Y$, then the Siegel formula for the $\psi_\xi$-twisted Siegel-Radon transform reads as 
\[
\int S_{\psi_\xi}f \dd \mu = 0 \quad (f \in  {}^0\mathrm{Ind}_H^G(\xi)).
\]
\end{no}

\subsection{Classical examples}\label{SubsecClassIntro}
Before we continue to develop the theory to the point where we can provide applications to aperiodic order, let us pause to discuss some classical examples which motivated the theory.
The most classical family of examples of Siegel-Radon transforms, which served as a blueprint for our theory, arises from double fibration of lattices:
\begin{no}\label{LatticeCase} Let $\Gamma < G$ be a lattice and $\mu$ be the unique element of  $\mathrm{Prob}(\Gamma \backslash G)^G$.
We say that $\Gamma$ is \emph{$H$-compatible} if $\Gamma_H := H \cap \Gamma$ is a lattice in $H$; in this case, $(\Gamma \backslash G, \Gamma_H \backslash H, \mu)$ is an integrable $(G,H)$-transverse system, and the transverse measure is given by $\covol(\Gamma)^{-1} \cdot \sigma_0$, where $\sigma_0$ denotes the unique $\Gamma$-invariant probability measure on $\Gamma_H \backslash H$. The associated Siegel transform 
is precisely the Helgason-Radon transform \cite{Helgason1, Helgason2} of the double fibration
\begin{equation*}\label{DFClassical}
\begin{xy}\xymatrix{
& \widetilde{X} = {\Gamma_H}\backslash G \ar[rd]^{\pi_{X}}\ar[ld]_{\pi_{H \backslash G}}&\\
H\backslash G&& X = \Gamma\backslash G
}\end{xy}
\end{equation*}
\end{no}
\begin{example}[Siegel transform]\label{SiegelIntro} Siegel's original example \cite{Siegel} pertains to the case $G = \mathrm{SL}_n(\R)$, $\Gamma= \mathrm{SL}_n(\Z)$ and $H = N$, the stabilizer of $e_1 = (1, 0, \dots, 0)^\top \in \R^n$ in $G$. In this case, $X$ can be identified with the space $\cL_n$ of unimodular lattices in $\R^n$ and $H\backslash G \cong \R^n \setminus\{0\}$. Under these identifications we obtain the classical (reduced) Siegel transform
\[
S: C_c(\R^n \setminus \{0\}) \to L^1(\cL_n, \mu_n), \quad Sf(\Lambda) = \sum_{x \in \Lambda_{\mathrm{vis}}} f(x),
\]
where $\mu_n$ denotes the invariant probability measure on $\cL_n$ and $\Lambda_{\mathrm{vis}} \subset \Lambda$ denotes the set of visible lattice points. The classical Siegel transform actually takes values in $L^2(\cL_n, \mu_n)$, as follows from  the celebrated work of Rogers \cite[Theorem 5]{Rogers}.
\end{example}
\begin{no}\label{LatticeCharacters} In the situation of \S \ref{LatticeCase}, the eigencharacters of $Y = \Gamma_H \backslash H$ are given by those characters $\xi$ of $H$ which vanish on $\Gamma_H$, and a canonical choice of $H$-eigenfunction for $\xi$ is $\psi_\xi(\Gamma_H h) = \xi(h)$. For each such character we thus obtain a $\psi_\xi$-twisted Siegel transform, and these transforms arise naturally in many areas of mathematics. For instance, the following example plays a major role in time frequency analysis. Other prominent (but more technical to state) examples of twisted Siegel transforms arise in the theory of automorphic forms.
\end{no}
\begin{example}[Zak transform]\label{IntroZak1} Let $U:=V:= \R^n$ and $Z:=\R$. We consider the model of the $(2n+1)$-dimensional Heisenberg group given by $G := U \times Z \times V$ with group multiplication
 \[
(u_1, t_1, v_1)(u_2, t_2, v_2) = (u_1+v_1, t_1+t_2 + \langle u_1, v_2 \rangle - \langle u_2, v_1\rangle, u_2+v_2),
\]
and its subgroup $H:= U \times \R$. Given $n \in \Z \setminus \{0\}$, define a character $\xi_n$ on $H$ by $\xi_n(u, t) = e^{2\pi i n t}$; then $\pi_n = \mathrm{Ind}_H^G(\xi_n)$ is the Schrödinger representation with central character $\xi_n(0, \cdot)$. 

\item Now let $\Gamma_U, \Gamma_V$ be lattice in $U$ and $V$ respectively and assume that $\langle \Gamma_U, \Gamma_V \rangle \subset \Z$ so that $\Gamma := \Gamma_U \times \Z \times \Gamma_V$ is a lattice in $G$ which intersects $H$ in a lattice $\Gamma_H$. For every $n \in \Z\setminus\{0\}$ the character $\xi_n$ is an eigencharacter of $\Gamma_H \backslash H$, and after identifying ${}^0\mathrm{Ind}_H^G(\xi_n) \cong C_c(V)$ in the usual way (see Subsection \ref{SecInducedRep} below), the associated twisted Siegel transform corresponds to the \emph{twisted Zak-transform} (cf.\  \cite[Chapter 8]{Gr} and \cite[Chapter 5]{AT})
\[
S_{\xi_n}: C_c(V) \to L^2(\Gamma \backslash G), \quad S_{\xi_n} f(\Gamma(u,t,v)) = e^{2\pi i n (t-\langle u, v \rangle)} \sum_{k \in \Gamma_V} f(v+k)e^{-4\pi i n \langle u, k \rangle}.
\]
If we rescale this transform by a factor of $\covol(\Gamma)^{-1/2}$, then it extends continuously to a unitary intertwiner $\mathrm{Ind}_H^G(\xi_n) \to L^2(\Gamma \backslash G)$. This provides a realization of all integral Schr\"odinger representations $(\pi_n)_{n \in \bZ\setminus\{0\}}$ inside $L^2(\Gamma \backslash G)$.
\end{example}
\begin{example}[Semi-classical Siegel-Radon transforms] 
Here are two less classical examples which fall into our framework. In fact, these examples have motivated us to look for a framework for abstract Siegel-Radon transforms beyond the lattice case:
\begin{enumerate}[(i)]
\item The Siegel-Veech transform \cite{Veech} of a stratum of translation surfaces can be seen as a Siegel-Radon transform, see Subsection \ref{SecSiegelVeech} below for details. By varying the cross section one can also obtain Siegel-Veech transforms for specific configurations (as in \cite{Lagace}).
\item Marklof-Str\"ombergsson's Siegel transform for quasicrystals from \cite{MS, MS2} can be seen as a Siegel-Radon transform, see Subsection \ref{SecMS} below for details. In this context, the Siegel formula was established in \cite[Theorem 6]{MS}
\end{enumerate}
\end{example}

\subsection{Induced systems and their integrability}

\begin{no} While the formalism or Siegel-Radon transforms works in the generality of integrable $(G,H)$-transverse systems, we obtain our strongest results in a more specialized setting of ``systems induced from separated systems''
that we now describe. All the examples from aperiodic order (as well as lattices) fall into this framework, but it is unclear to us whether this is the case for the Siegel-Veech transform and the Marklof-Str\"ombergsson transform.

We refer to a Borel subset $Z \subset X$ with $G.Z \subset X$ as a \emph{$G$-cross section}, and we say that $Z$ is \emph{separated} if the \emph{return time set} $\Lambda(Z) = \bigcup_{z \in Z} Z_z$ does not accumulate at $e_G$; it is called \emph{cocompact} if $X = D.Z$ for some compact $D \subset G$.
By a classical result of Conley, every free standard Borel $G$-space admits a cocompact separated $G$-cross section $Z$ (see e.g.\ \cite[Corollary 1.2]{Kechris})

If $Z \subset X$ is separated, then $Y := H.Z$ is a $(G,H)$-cross section, called the \emph{induced $(G,H)$-cross section} of $Z$, and these will be the cross sections we are most interested in. If $\mu \in \mathrm{Prob}(X)^G$, then we refer to $(X, Z, \mu)$ as a \emph{separated $G$-transverse system} and to $(X,Y, \mu)$ as the \emph{induced $(G,H)$-transverse system}. The following notions appear naturally in the study of the integrability of such systems.
\end{no}
\begin{definition} Let $\Lambda \subset G$ be a symmetric subset containing the identity.
\begin{enumerate}[(i)]
\item $\Lambda$ has \emph{finite covolume} (respectively, is \emph{relatively dense}) if $G = \Lambda \cdot \cF$ for some $\cF \subset G$ of finite Haar measure (respectively, some compact $\cF \subset G$).
\item $\Lambda$ is called an \emph{approximate subgroup} if  $\Lambda \cdot \Lambda \subset \Lambda \cdot F$ for some finite subset $F \subset G$. 
\item A discrete approximate subgroup of finite covolume is called\footnote{This definition is more restrictive, but easier to state than the original definition from \cite{BHAL}, see \cite{MachadoDefinitions}.} an \emph{approximate lattice}. It is called a \emph{uniform approximate lattice} if it is moreover relatively dense.
\item If $\Lambda \subset G$ and $\Lambda \cap H \subset H$ are both (uniform) approximate lattices, then we refer to $(H, \Lambda)$ as a \emph{(uniform) compatible pair}.
\end{enumerate}
\end{definition}
\begin{mainthm}[Integrability criterion]\label{IntegrabilityIntro}Assume that $(X, Z, \mu)$ is a separated transverse $G$-system with return time set $\Lambda(Z)$ and induced system $(X, Y, \mu)$.
\begin{enumerate}[(i)]
\item If there is a compatible pair $(H, \Lambda)$ with $\Lambda(Z) \subset \Lambda$, then $(X, Y, \mu)$ is integrable.
\item If $(H, \Lambda)$ can moreover be chosen to be uniform, then $(X, Y, \mu)$ is $p$-integrable for all $p \in [1, \infty]$.
\end{enumerate}
\end{mainthm}
While the uniform case of Theorem \ref{IntegrabilityIntro} can be proved by elementary combinatorics, the non-uniform case requires some rather involved transverse measure theory. In the non-uniform case, square-integrability may fail even for lattices \cite[Appendix~2]{DRS}. In some specific cases of interest we will be nevertheless be able to establish square-integrability in the non-uniform setting, cf.\ Remark \ref{CUPInt} below. In general, the question of square-integraibility is related to generalizations of Rogers' formula.

\subsection{Non-classical examples from aperiodic order}
\begin{no} The space $\cC(G)$ of closed subsets of $G$ carries a natural compact metrizable topology known as \emph{Chabauty-Fell topology}, and $G$ acts on $\cC(G)$ by homeomorphisms via $g.P := Pg^{-1}$. Given $P_o \in G$ we denote by  $X^+(P_o)$ its orbit closure in $\cC(G)$. The lcsc space \[X(P_o) := X^+(P_o) \setminus\{\emptyset\}\] is known as the \emph{hull}\footnote{In the literature this term is sometimes used for $X^+(P_o)$ instead, but the current terminology is more convenient in our context.} of $P_o$. It has a canonical separated $G$-cross section given by \[Z(P_o) = \{P \in X(P_o) \mid e \in P\},\] and the induced $(G,H)$-cross section is given by
\[
Y(P_o) = \{P \in X(P_o) \mid P_o \cap H \neq \emptyset\}.
\]
\end{no}
\begin{definition}
We say that $P_o \subset G$ is \emph{measured} if $\mathrm{Prob}(X(P_o))^G \neq \emptyset$.
\end{definition}
\begin{mainthm}\label{HullIntIntro} Let $(H, \Lambda)$ be a compatible pair and $P_o \subset \Lambda$ be a measured subset of $G$. Then for every $\mu_{P_o} \in \mathrm{Prob}(X({P_o}))^G$ the $(G,H)$-transverse system $(X(P_o), Y(P_o), \mu_{P_o})$ is integrable.
\end{mainthm}
In the situation of Theorem \ref{HullIntIntro}, the associated Siegel-Radon transform can be written as
\[
S: C_c(H\backslash G) \to L^1(X(P_o), \mu_{P_o}), \quad Sf(P) = \sum_{Hg \in \pi(P)} f(Hg),
\]
and similarly for the twisted version. Since the theorem applies to measured subsets of approximate lattices, the question arises how such sets can be constructed. We will provide several constructions, leading to a wealth of examples. The most basic class of examples is given by (regular) cut-and-project sets \cite{BHP1}:
\begin{example}[Cut-and-project sets]\label{CuPIntro} If $G = G_1$ and $G_2$ are lcsc groups and $\Gamma<G_1 \times G_2$ is a lattice, then $\cS := (G_1, G_2, \Gamma)$ is called a \emph{cut-and-project scheme} if $\mathrm{pr}_{G_1}|_{\Gamma}$ is injective and $\mathrm{pr}_{G_2}(\Gamma) \subset G_2$ is dense. If $W \subset G_2$ is a Jordan-measurable bounded subset with non-empty interior (a so-called \emph{window}), then the \emph{cut-and-project sets} associated with $(\cS, W)$ are defined as 
\[
P_o(\cS, W, g_1, g_2) := \mathrm{pr}_{G_1}(\Gamma(g_1, g_2) \cap (G_1 \times W)) \subset G_1 \quad (g_1 \in G_1, g_2\in G_2).
\]
Cut-and-project sets were introduced by Meyer \cite[Chapter 2, Section 5]{Meyer1} (for $G_1$, $G_2$ abelian) and \cite{BHP1} (general case); by \cite{BHAL} they are always approximate lattices, and for generic choices of $g_1, g_2$ they are measured by \cite[Sec.\ 3.2]{BHK}. In fact, there is a large class of cut-and-project sets, known as \emph{regular cut-and-project sets} for which the hull is uniquely ergodic \cite{BHP1}. 

The standard way to produce invariant measures on hulls of cut-and-project sets is by viewing them as hitting times (as in \cite{HKW,SW}): Consider the space $X(\cS) := \Gamma \backslash (G_1 \times G_2)$ with its unique $(G_1 \times G_2)$-invariant measure $\mu_{\cS}$. Every window $W$ define an integrable $(G, H)$-cross section $Y(W) := \Gamma \backslash (H \times W)$, and there is a measurable isomorphism (up to nullsets)
\[
(X(\cS), Y(W), \mu_{\cS}) \to (X(P_o), Y(P_o), \mu_{P_o}), \quad x \mapsto Y(W)_x
\]
where $P_o$ is a generic cut-and-project set constructed from $\cS$ and $W$ and $\mu_{P_o}$ is the push-forwards measure of $\mu_{\cS}$.
\end{example}
\begin{remark}\label{CUPInt}
The explicit form of the hull system of a regular cut-and-project set has several applications. Firstly, in many case we can compute the $H$-transverse measure and the Siegel constant explicitly in terms of the underlying lattice and window, see Theorem \ref{TransverseExplicit} below. Also, in certain arithmetic situations (see e.g. Example \ref{ArithmeticExamples}) the aperiodic system $(X(\cS), Y(W), \mu_{\cS})$ admits a monotone joining with a lattice system (cf.\ Subsection \ref{HigherIntegrability} below), and thus is square-integrable provided the corresponding lattice system is, which is known in certain cases (see e.g. \cite[Theorem 1.3]{Kim}). Note that square-integrability comes for free if $\Gamma$ happens to be cocompact.
\end{remark}
\begin{example}[Thinned cut-and-project sets] There are several (integrability preserving) methods of ``thinning'' a $(G,H)$-transverse system (see in particular Subsection \ref{SecCouplings} and Subsection \ref{SecCupThinning} below) which allow us to define Siegel transforms for the corresponding subsets of (regular) cut-and-project sets.
\end{example}
\begin{example}[Positive density sets]\label{PosDen} We show in Theorem \ref{Densitymeasure} below that every subset of ``positive density''  in an amenable group is measured. Thus, Theorem \ref{HullIntIntro} applies to positive density subsets of approximate lattices in amenable groups.
\end{example}

\subsection{Aperiodic twisted Zak transforms}
\begin{no} We conclude this introduction with a sample application of Siegel-Radon transforms to a spectral problem in aperiodic order, generalizing Example \ref{IntroZak1}. Our techniques apply to a variety of nilpotent groups, but for concreteness' sake
we consider again the Heisenberg group $G = U \times Z \times V$ and its subgroup $H = U \times Z$ as in Example \ref{IntroZak1}.

\item Now let $\Lambda_U$,  $\Lambda_V$ and $\Lambda_Z$ be approximate lattices in $U$, $V$ and $Z$ respectively such that $\langle \Lambda_U, \Lambda_V \rangle \subset \Lambda_Z$; there are plenty of (e.g.\ arithmetic) examples of this form, see e.g.\ \cite{BHNilpotent}.
\end{no}
\begin{mainthm}[Aperiodic Zak transform]\label{ZakIntro} 
Every subset $P_o \subset  \Lambda_U \times \Lambda_Z \times \Lambda_V$ of positive density is measured. Moreover, the following hold for every choice of $\mu_{P_o} \in \mathrm{Prob}(X(P_o))^G$ and every $\eps \in (0,1)$:
\begin{enumerate}[(i)]
\item $(X(P_o), Y(P_o), \mu_{P_o})$ is a square-integrable $(G,H)$-transverse system, hence gives rise to an $H$-transverse measure $\sigma_{P_o}$.
\item If $\xi \in \Lambda_Z^\eps \setminus \{\mathbf{1}\}$ for some $\eps \in (0,1)$, then $1 \otimes \xi$ is an eigencharacter of $(Y(P_o), \sigma_{P_o})$.
\item  If $\xi \in \Lambda_Z^\eps \setminus \{\mathbf{1}\}$ for some $\eps \in (0,1)$, then for every eigenfunction $\psi_\xi$ of $1 \otimes \xi$ the rescaled twisted Siegel-Radon transform $\smash{\sigma_{P_o}(Y_{P_o})^{-1/2} \cdot S_{\psi_\xi}}$ extends continously to a unitary intertwiner
 \[\mathrm{Zak}_{\psi_\xi} : \mathrm{Ind}_H^G(\xi) \to L^2(X(P_o), \mu_{P_o}).\]
\item The subset $\Lambda_Z^\eps \subset \widehat{Z}$ is relatively dense, hence a relatively dense set of Schr\"odinger representations embeds discretely into $L^2(X(P_o), \mu_{P_o})$.
\end{enumerate}
\end{mainthm}
\begin{remark}
Part (iv) of the theorem was established in \cite{BHNilpotent} in the case where $P_o$ itself was an approximate lattices. Theorem \ref{ZakIntro} extends this result to positive density subsets, but more importantly it provides \emph{explicit intertwiners} from the corresponding Schr\"odinger representations into the $L^2$-space of the hull.
\end{remark}
\subsection{Organization of the article} 

This article is organized in three parts: 
\begin{itemize}
\item The first part (Sections \ref{SecITS} and \ref{SecTSR}) concerns the abstract theory of Radon-Siegel transforms in the most general setting. Section \ref{SecITS} contains background on integrable transverse systems. In particular, we establish Proposition \ref{PropTAdjoint}. Section \ref{SecTSR} introduces Radon-Siegel transforms and establishes Theorem \ref{RadonUntwisted} and Theorem \ref{RadonTwisted}.

\item The second part (Section \ref{SecInducedSystems} and Section \ref{InducedIntegrability}) contains those parts of the theory which are specific to systems induced from separated systems. Section \ref{SecInducedSystems} contains general background on such systems, whereas Section \ref{InducedIntegrability} is devoted to the proof of Theorem \ref{IntegrabilityIntro}.

\item The final part (Sections \ref{SecClassical}-\ref{SecZak}) contains concrete examples and applications. Section \ref{SecClassical} explains how the various classical examples mentioned in Subsection \ref{SubsecClassIntro} fit into our framework. Section \ref{SecAO} is devoted to the new examples from aperiodic order and establishes Theorem \ref{HullIntIntro}. Finally, Section \ref{SecZak} is concerned with aperiodic Zak transforms and their generalizations and establishes Theorem \ref{ZakIntro}.
\end{itemize}

\noindent \textbf{Acknowledgements.} We are grateful to Marc Burger, Manfred Einsiedler, Alex Gorodnik, Anja Randecker and the participants of the Workshop ``Approximate lattices and aperiodic order, II'' at KIT for discussions related to the subject of the present article. We are grateful to KIT for funding this workshop and to KIT and Chalmers University for their hospitality during out mutual visits.

\section{Integrable transverse dynamical systems}\label{SecITS} 
\begin{no} We recall that $G$ is always a unimodular lcsc group and $H<G$ is always a closed unimodular subgroup with canonical projection $\pi: G \to H \backslash G$. The group $G$ will always act on itself on the \emph{right} hence for $g,x \in G$ and $A \subset G$ we write $g.x := xg^{-1}$ and $g.A := Ag^{-1}$.

Throughout this article (or $X$ for short), $(X, \cB_X, a_X)$ will always denote a standard Borel $G$-space, i.e. $(X, \cB_X)$ is a standard Borel space and $a_X: G \times X \to X$ is a Borel $G$-action, and $\mu \in \mathrm{Prob}(X)^G$ will be a $G$-invariant Borel probability measure on $X$.  Moreover, $Y$ will always denote a $(G,H)$-cross section in $X$ so that $(X, Y, \mu)$ is a $(G,H)$-transverse system. We recall from the introduction that this means that $Y \subset X$ is an $H$-invariant Borel set such that $G.Y = X$, and that the hitting time set of $Y$ from $x \in X$ is defined as
\[
Y_x := \{Hg \in H \backslash G \mid g.x \in Y\} \subset H \backslash G.
\]
If we denote the $G$-action on $H\backslash G$ by $g.Ha := Hag^{-1}$, then we have
\begin{equation}\label{YxEquiv} Y_{g.x}=g.Y_x \quad \text{for all } g \in G,\, x \in X.
\end{equation} 
(In general, we will always act by $G$ on the right, 
\end{no}

\subsection{Existence of $H$-transverse measures}
In this subsection we establish Proposition \ref{PropTAdjoint} from the introduction and some minor variants thereof. The following lemma establishes Part (i) of the proposition:
\begin{lemma}\label{TOperator} For every integrable $(G,H)$-transverse system $(X,Y, \mu)$  there is a well-defined linear map $T: \cL^\infty_{c}(G \times Y) \to L^1(X, \mu)$ given by
\[
 TF(x) =  \sum_{Hg \in Y_x}  \int_H F((hg)^{-1}, hg.x) \dd m_H(h).
\]
\end{lemma}
\begin{proof} Let $K \subset G$ be compact so that  $F \in \cL^\infty_K(G \times Y)$. If $F((hg)^{-1}, y) \neq 0$, then $hg \in K^{-1}$, hence $h \in H \cap g K$ and $Hg \in \pi(K^{-1})$. 
We note that if $h \in H \cap gK$, then
\[
m_H(H \cap gK) = m_H(h^{-1} (H \cap gK)) = m_H(H \cap h^{-1}g K) \leq m_H(H \cap K^{-1}K) =: C_K.
\] 
Since $K$ is compact we have $C_K < \infty$. Now,
\[
|TF(x)| = \left| \sum_{Hg \in Y_x \cap \pi(K^{-1})} \int_{H \cap gK}F((hg)^{-1}, hg.x)\dd m_H(h)\right| \leq C_K \|F\|_\infty |Y_x \cap \pi(K^{-1})|.
\]
Since $Y$ is integrable we deduce that $TF \in L^1(X, \mu)$.
\end{proof}
Towards the proof of Part (ii) of the Proposition we first establish the special case $\phi \equiv 1$:
\begin{proposition}\label{PropsigmaIntro} 
For every integrable $(G,H)$-cross section $Y \subset X$ there
exists a unique finite $H$-invariant Borel measure $\sigma$ on $Y$ such that for all $F \in \cL^\infty_c(G \times Y)$ we have
\begin{equation}\label{SigmaDefEq}
\int_X TF(x) \dd \mu(x) = \int_G \int_Y F(g,y) \dd \sigma(y) \dd m_G(g).
\end{equation}\end{proposition}

\begin{proof}
By Lemma \ref{TOperator} the integral
\[
\Lambda(u, \varphi) :=  \int_X T(u \otimes \phi)\dd \mu 
\]
converges for every $u \in C_c(G)$ and $\phi \in \cL^\infty(Y)$ and we have $\Lambda(u, \varphi) \geq 0$ if $u \geq 0$ and $\varphi \geq 0$.\\

\item If $\sigma$ satisfies \eqref{SigmaDefEq}, then necessarily $\Lambda(u, \varphi) = m_G(u) \sigma(\varphi)$ for all $\phi \in \cL^\infty(Y)$ and $\phi \in \cL^\infty(Y)$. If we choose $u_o$ with $m_G(u_o) = 1$, then we obtain $\sigma(\varphi) = \Lambda(u_o, \varphi)$, hence $\sigma$ is uniquely determined by \eqref{SigmaDefEq}.\\

\item Now fix a non-negative function $\varphi \in \cL^\infty(Y)$ and define a positive linear functional on $C_c(G)$ by $\Lambda_{\varphi}(u) := \Lambda(u, \varphi)$. Let $g, g_0 \in G$ and $x \in X$. If we set $g' := gg_0$ and $x' := g_o^{-1}.x$, then 
\[
Hg \in Y_x \iff Hgg_o \in Y_xg_o = Y_{g_o^{-1}.x} \iff Hg' \in Y_{x'}
\]
We deduce that if $u \in C_c(G)$ and $u_{g_o}(g) := u(g_o^{-1}g)$, then 
\begin{align*}
\Lambda_{\varphi}(u_{g_o}) &= \int_X \left(\sum_{Hgg_o \in Y_{g_o^{-1}.x}} \int_H \check u(hgg_0) \phi(hgg_o.(g_o^{-1}.x)) \dd m_H(h) \right) \dd \mu(x)\\
&=  \int_X \left(\sum_{Hg'\in Y_{x'}} \int_H \check u(hg') \phi(hg'.x') \dd m_H(h) \right) \dd \mu(x')\; =\; \Lambda_{\phi}(u),
\end{align*}
i.e. $\Lambda_{\varphi}$ is a left-Haar measure. We thus find $C(\phi) \in (0, \infty)$ such that
\begin{equation}\label{LambdaCmG}
\Lambda(u, \varphi) = C(\phi) \cdot m_G(u).
\end{equation}
We now fix  $u_o$ with $m_G(u_o) = 1$. We then define $\sigma(B) := C(1_B) = \Lambda(u_o, 1_B) \geq 0$ for every Borel subset $B \subset X$. Since $\Lambda$ is bilinear this is finitely additive, and we have $\sigma(Y) = C(1) < \infty$. Given a countable family $(B_n)_{n \in \bN}$ of disjoint Borel sets in $X$ we define
$B := \bigsqcup B_n$, $f := u_o \otimes 1_B$ and $f_N := \sum_{n=1}^N u_o \otimes 1_{B_n}$. Then $0 \leq Tf_1 \leq Tf_2 \leq \dots \leq Tf$ and $Tf = \sup Tf_n$, hence by monotone convergence we have
\[
\sigma(B) = \int_X Tf \dd\mu = \int_X \sup T f_N \dd \mu = \sup \int_X f_N \dd \mu = \sum_{n=1}^\infty \sigma(B_n),
\]
which shows that $\sigma$ is $\sigma$-additive, hence a finite Borel measure on $Y$. 

\item We deduce from \eqref{LambdaCmG} that \eqref{SigmaDefEq} holds for all $F$ of the form $F = u \otimes 1_B$ with $u \in C_c(G)$ and $B \subset Y$ Borel, and hence for arbitrary $F \in \cL^\infty_c(G \times Y)$.
\end{proof}
\begin{proof}[Proof of Proposition \ref{PropTAdjoint}] Part (i) was established in Lemma \ref{TOperator}. Concerning (ii), let $F \in \cL^\infty_c(G \times Y)$ and $\phi \in \cL^\infty(X)$. The function $F_\phi(g,y) := F(g,y)\overline{\phi(g.y)}$ satisfies 
\[F_\phi((hg)^{-1}, hg.x) = F((hg)^{-1}, hg.x) \overline{\phi(x)} \implies
TF_\phi(x) = TF(x)\overline{\phi(x)}.\]
Applying \eqref{SigmaDefEq} to $F_\phi$ thus yields \eqref{TAdjoint}. Note that we can recover \eqref{SigmaDefEq} from \eqref{TAdjoint} by setting $\varphi \equiv 1$, hence $\eqref{TAdjoint}$ determines $\sigma$.
\end{proof}
\begin{remark}\label{SigmaDefEqConv} In the sequel we need some small variants of Proposition \ref{PropsigmaIntro}. Thus let
$Y \subset X$ be an integrable $(G,H)$-cross section.  
\begin{enumerate}[(1)]
\item Given $F \in \cL^\infty_c(G \times Y)$ we denote $\check F(g,y):= F(g^{-1},y)$. By unimodularity of $G$ we then have $(m_G \otimes \sigma)(\check F) = (m_G \otimes \sigma)(F)$, hence \eqref{SigmaDefEq} is equivalent to
\begin{equation}\label{SigmaDefEq2}
\int_X \left( \sum_{Hg \in Y_x}  \int_H F(hg, hg.x) \dd m_H(h) \right) \dd \mu(x) = \int_G \int_Y F(g,y) \dd \sigma(y) \dd m_G(g).
\end{equation}
\item If $F \in \cL^\infty(G \times Y)$ is non-negative (but does not have bounded support) then the integrals in \eqref{SigmaDefEq} and \eqref{SigmaDefEq2} may be infinite but still coincide, as can be seen by approximation by compactly supported functions.
In particular, if one of them is finite, then so is the other one.
\item In the proof of Siegel duality we will apply Proposition \ref{PropsigmaIntro} in the following form. We recall that our Haar measures have been normalized so that 
\begin{equation}\label{Weil}
\int_G f(g) \dd m_G(g) = \int_{H \backslash G} \int_H f(hg) \dd m_H(h) \dd m_{H \backslash G}(Hg) \quad (f \in C_c(G)).
\end{equation}

\end{enumerate}
\end{remark}
\begin{corollary}\label{GHIntegralSigma}
Let $F : G \times Y \ra \bC$ be a bounded Borel function such that $\supp(F) \subset HK \times Y$ for some compact subset $K \subset G$ and
$F(hg,h.y) = F(g,y)$ for all $h \in H$ and $(g,y) \in G \times Y$. Then
\begin{equation}\label{GHIntegralSigma1}
\int_{H \backslash G} \Big( \int_Y F(g,y) \dd \sigma(y) \Big) \dd m_{H \backslash G}(Hg) = \int_X \sum_{Hg \in Y_x} F(g,g.x) \dd \mu(x)
\end{equation}
and both of these integrals converge. 
\end{corollary}
\begin{proof} We may assume that $F$ is non-negative so that both integrals exist (but are possibly infinite). For $g \in G$ and $x \in X$ we then have $F(g,g.x) = 0$ unless $Hg \in \pi(K)$. We thus obtain
\[
 \int_X \sum_{Hg \in Y_x} F(g,g.x) \dd \mu(x) = \int_X \sum_{Hg \in Y_x \cap \pi(K)} F(g, g.x)\dd\mu(x) \leq \|F\|_\infty \cdot \int_X |Y_x \cap \pi(K)| \dd \mu(x) < \infty,
\]
hence the right integral is actually finite. By \cite{Raghunathan} there exists a non-negative function $\rho \in C_c(G)$ such that
\begin{equation}\label{MagicRho1}
\int_H \rho(hg) \dd m_H(h) = 1 \quad \textrm{for all $g \in G$ with $Hg \in \pi(K)$}.
\end{equation}
We then define  $F_\rho(g,y) := \rho(g) F(g,y)$. Applying \eqref{MagicRho1}, $H$-invariance of $F$ and \eqref{SigmaDefEq2} (which applies in view of $\S \ref{SigmaDefEqConv}.(3)$) we get
\begin{align*}
 \int_X \sum_{Hg \in Y_x} F(g,g.x) \dd \mu(x) &= \int_X \sum_{Hg \in Y_x \cap \pi(K)} \Big(\int_H \rho(hg)  \dd m_H(h) \Big) \, F(g,g.x)\dd \mu(x)\\
 &= \int_X \sum_{Hg \in Y_x} \Big(\int_H \rho(hg) F(hg,hg.x)  \dd m_H(h) \Big) \dd \mu(x)\\
 &=  \int_X \sum_{Hg \in Y_x} \Big(\int_H F_\rho(hg,hg.x) \dd m_H(h) \Big) \dd \mu(x)\\
 &=\int_G \int_Y F_\rho(g,y) \dd \sigma(y) \dd m_G(g),
\end{align*}
Using \eqref{Weil} we conclude that
\[
 \int_X \sum_{Hg \in Y_x} F(g,g.x) \dd \mu(x) = \int_{H \backslash G}\left(\int_H \int_Y \rho(hg) F(hg,y) \dd \sigma(y)\dd m_H(h)\right) \dd m_{H \backslash G}(Hg),
\]
and by $H$-invariance of $F$ and $\sigma$ and \eqref{MagicRho1} the inner integral equals
\[
\int_H \int_Y \rho(hg) F(hg,hy) \dd \sigma(y)\dd m_H(h) = \int_Y F(g,y) \left( \int_H \rho(hg) \dd m_H(h)\right) \dd \sigma(y) =  \int_Y F(g,y)\dd \sigma(y),
\]
where we have used that $F(g,y) \neq 0$ implies that $Hg \in \pi(K)$ so that \eqref{MagicRho1} applies.
\end{proof}
\begin{no} Let us spell out the results of this subsection in the special case $H = \{e\}$. In this case, a $(G,H)$-cross section $Z \subset X$ is simply a $G$-cross section, and if $\mu \in \mathrm{Prob}(X)^G$ then the $\{e\}$-transverse measure $\nu$ of $(X,Z, \mu)$ is simply called the \emph{transverse measure} of $(X,Z, \mu)$. As before we abbreviate
\[
TF(x) :=\sum_{g \in Z_x} F(g^{-1}, g.x) \quad (F \in \cL_c^\infty(G \times Z)).
\]
Then $\nu$ is a finite Borel measure which satisfies
\begin{equation}\label{TransverseMeasure1}
\int_X TF(x)  \overline{\phi(x)} \dd \mu(x) = \int_G \int_Z F(g, z) \overline{\phi(g.z)} \dd \nu(z) \dd m_G(g),
\end{equation}
for all $\phi \in \cL^\infty(X, \mu)$ and is uniquely determined by satisfying \eqref{TransverseMeasure1} for $\phi = 1$.
\end{no}

\subsection{Intersection measures} 
Based on the results of the previous subsection we can now introduce intersection spaces and intersection measures; with a view towards later applications we will do this in slightly greater generality than in the introduction.
\begin{no} Let $Y \subset X$ be a $(G,H)$-cross section such that for every $x \in X$ the hitting time set $Y_x$ is countable. As in the introduction we then define the \emph{intersection space} of $H \backslash G$ with $(X, Y)$ is the pair $(\widetilde{X}, \widetilde{Y})$ given by
\[
\widetilde{X} := G.\widetilde{Y} \subset H\backslash G \times X, \quad \text{where} \quad \widetilde{Y} := \{H\} \times Y.
\]
Since $Y_x \subset H\backslash G$ is countable and $H$ is $\sigma$-compact, the fibers of the action map $G \times \widetilde{Y} \to \widetilde{X}$ are $\sigma$-compact, hence $\widetilde{X} \subset H\backslash G \times X$ is Borel by  \cite[Corollary 18.10]{K}. We thus obtain a $G$-equivariant double fibration
\[\begin{xy}\xymatrix{
&(\widetilde{X}, \widetilde{Y}) \ar[rd]^{\pi_{X}}\ar[ld]_{\pi_{H \backslash G}}&\\
(H\backslash G, \{H\})&&(X,Y),
}\end{xy}
\]
in which the fibers of $\pi_{H\backslash G}$ can be identified with $Y$ via the family of Borel isomorphisms $\iota_g: Y \to \pi^{-1}(Hg)$ given by $\iota_g(y) := (Hg, g^{-1}.y)$. 
\end{no}
\begin{no} Let now $Z \subset Y$ be a $G$-cross section, where neither $H$-invariance nor any form of integrability is assumed. (We are mostly interested in the case $Z=Y$, but other cases will occur later.) We then denote $\widetilde{Z} := \{H\} \times Z \subset \widetilde{Y} \subset \widetilde{X}$.

We say that a Borel measure $\eta$ on $X$ is \emph{$(G,Z)$-finite} if $\eta(K.Z) < \infty$ for every compact subset $K \subset G$. Equivalently, $\eta(G_n.Z) < \infty$ for some (hence any) exhaustion $(G_n)_{n \in \bN}$ of $G$ by compact subsets, hence every $(G,Z)$-finite measure is $\sigma$-finite. $(H,Z)$-finite measures on $Y$ are defined accordingly.
\end{no}
\begin{construction}[Intersection measures] Assume that we are given an $(H,Z)$-finite and $H$-invariant Borel measure $\tau$ on $Y$. We then define a  $(G, \widetilde{Z})$-finite measure $\widetilde{\mu}_\tau$ on $\widetilde{X}$ as follows. Let $F : \widetilde{X} \to \bC$  be a bounded Borel function which vanishes outside $K.\widetilde{Z}$ for some compact subset $K \subset G$. Since $\tau$ is $H$-invariant, we obtain a well-defined function
\[
\phi_F: H \backslash G \to \bC, \quad \varphi_F(Hg) = \int_Y F(Hg,g^{-1}.y) \dd\tau(y).
\]
Then $\phi_F$ is bounded with bounded support, hence we may define $\widetilde{\mu}_\tau$ by
\begin{equation}\label{Defmutau}
\int_{\widetilde{X}} F \dd \widetilde{\mu}_\tau := \int_{H \backslash G} \Big( \int_Y F(Hg,g^{-1}.y) \, d\tau(y) \Big) \dd m_{H \backslash G}(Hg).
\end{equation}
Here, $G$-invariance of $\widetilde{\mu}_\tau$ follows from $G$-invariance of $m_{H \backslash G}$. We refer to $\widetilde{\mu}_\tau$ as the \emph{intersection measure} defined by $\tau$.
\end{construction}
\begin{remark} By construction the projection $(\widetilde{X}, \widetilde{\mu}_{\tau}) \to (H\backslash G, m_{H \backslash G})$ is measure class preserving and the fiber measure over $Hg \in H \backslash G$ is precisely $(\iota_g)_*\tau$. In particular, $\tau$ is uniquely determined by $\widetilde{\mu}_{\tau}$. 
\end{remark}
\begin{example}\label{IntersectionInducedY} Let $(X,Y, \mu)$ be an integrable $(G,H)$-transverse system with $H$-transverse measure $\sigma$, and let $Z := Y$. Then $\sigma$ is finite (hence $(H, Z)$-finite) and $H$-invariant, hence gives rise to an intersection measure $\widetilde{\mu} = \widetilde{\mu}_\sigma$, called the \emph{intersection measure} of $(X,Y, \mu)$. This is the intersection measure we considered in the introduction and our most important example.
\end{example}

\section{Twisted Siegel-Radon transforms}\label{SecTSR}

In this section we introduce twisted Siegel-Radon transforms in the generality of integrable transverse $(G,H)$-systems and establish Theorem \ref{RadonUntwisted} and Theorem \ref{RadonTwisted}.

\subsection{Eigencharacters}

\begin{no} We denote by $\widehat{H}$ the set of all (unitary) characters $\xi: H \to \bT$. Let $\xi \in \widehat{H}$ and let $Y$ be a standard Borel $H$-space and $\sigma$ be a finite $H$-invariant measure on $\sigma$. In our applications, $Y$ will be an integrable $(G, H)$-cross section and $\sigma$ will be an $H$-transverse measure.

We say that $\xi$ is an \emph{eigencharacter} of $(Y, \sigma)$ if there exists an $H$-invariant $\sigma$-conull set $Y_0 \subset Y$ and a nonzero bounded Borel function $\psi_\xi: Y_0 \to \bC$ such that
 \begin{equation}\label{Eigenfunction}
\psi_\xi(h^{-1}.y) = \xi(h)\psi_\xi(y) \quad \text{for all }h \in H, \; y \in Y_0.
\end{equation} 
The function $\psi_\xi$ is then called an \emph{eigenfunction} with domain $Y_0$ for $\xi$. If $Y_0 = Y$, then we say that $\psi_\xi$ is a \emph{strict eigenfunction}. Note that every eigenfunction can be extended by $0$ to a strict eigenfunction. 
\end{no}
\subsection{Induced representations}\label{SecInducedRep}
\begin{no} We will consider two different models of the induced representation of a character $\xi \in \widehat{H}$. Concerning the first model, we observe that if a Borel function $f: G \to \bC$ satisfies
\begin{equation}\label{Indxi1}
f(hg) = \xi(h) f(g) \quad \text{for all }h \in H, g \in G
\end{equation}
for some $\xi \in \widehat{H}$, then $|f|$ descends to a function on $H \backslash G$. We may thus define 
\[
{}^0\mathrm{Ind}_H^G(\xi) = \{f \in C(G) \mid f(hg) = \xi(h) f(g) \; \text{for all }h \in H, g \in G  \; \text{and}\; |f| \in C_c(H \backslash G)\},
\]
and obtain an action of $G$ on this space by $\pi_\xi(g)f(x) = f(xg)$. Given $f_1, f_2 \in {}^0\mathrm{Ind}_H^G(\xi)$ the function $f_1(g)\overline{f_2(g)}$ descends to a function on $H \backslash G$, and we define 
\[
\langle f_1, f_2 \rangle_{\mathrm{Ind}_H^G(\xi)} :=  \int_{H \backslash G} f_1(g)\overline{f_2(g)}\dd m_{H\backslash G}(Hg).
\]
The corresponding completion $\mathrm{Ind}_H^G(\xi)$ is then a Hilbert space and $\pi_\xi$ extends to a unitary $G$-representation on $\mathrm{Ind}_H^G(\xi)$, which is our first model for the induced representation. For the trivial character we have ${}^0\mathrm{Ind}_H^G(\mathbf{1}) \cong C_c(H \backslash G)$ and $\mathrm{Ind}_H^G(\mathbf{1}) \cong L^2(H \backslash G)$.
 \end{no}
 \begin{no}\label{GHConventions}\label{InducedModel2} 
 Our second model of the induced representation depends on the choice of a locally bounded Borel section $s: H \backslash G \to G$ of $\pi$ with $s(H) = e$. We fix such a choice once and for all and employ the following notations. Given $g \in G$ we write 
  \begin{equation}\label{ell}
  g = h(g)\ell(g),\quad \text{where} \quad \ell(g) := s(\pi(g)) \qand h(g) := g\ell(g)^{-1} \in H.
 \end{equation} 
This provides a bijection $G \to H \times \ell(G)$, $g \mapsto (h(g), \ell(g))$ with inverse given by multiplication. We now define a cocycle $\alpha = \alpha_s$ by
 \[
 \alpha: G \times H \backslash G \to G, \quad \alpha(g, Hx) = s(Hx)gs(Hxg)^{-1} = \ell(x)g\ell(xg)^{-1}.
 \]
 This satisfies the cocycle and normalization identities 
 \begin{equation}\label{alphacocycle}
 \alpha(g_1g_2, Hx) = \alpha(g_1, Hx) \alpha(g_2, Hxg_1), \quad \alpha(e, Hx) = e \qand \alpha(g, H) = h(g),
 \end{equation}
hence we may define a $G$- action $\pi_{s,\xi}$ on $C_c(H \backslash G)$ by 
\begin{equation}\label{CocycleAction}
\pi_{s,\xi}(g)\phi(Hx) := \xi(\alpha(g, Hx)) \phi(Hxg),
\end{equation}
which in turn extends to a unitary representation on $L^2(H \backslash G)$. Moreover, the map
\begin{equation}\label{Interi}
i: {}^0\mathrm{Ind}_H^G(\xi) \to C_c(H\backslash G), \quad i(\phi_o)(Hx) := \phi_o(s(Hx)) = \phi_o(\ell(x))
\end{equation}
is a bijective intertwiner between $\pi_\xi$ and $\pi_{s,\xi}$ with inverse given by
\[
i^{-1}(\phi)(g) = \xi(h(g))\phi(Hg) = \xi(\alpha(g,H))\phi(Hg) = \pi_{s,\xi}(g)\phi(H).
\]
which extends to a unitary equivalence $i: (\mathrm{Ind}_H^G(\xi), \pi_\xi) \to (L^2(H\backslash G), \pi_{s, \xi})$. Note that 
\begin{equation}\label{SiegelAlt1}
g \in \ell(G) = s(H \backslash G) \implies h(g) = e \implies  i^{-1}(\phi)(g) = \phi(Hg).
\end{equation}
\end{no}

\subsection{Twisted Siegel-Radon transforms} 
In this subsection we establish Part (i) of Theorem \ref{RadonTwisted} and discuss the integrability properties of Siegel-Radon transforms mentioned in \S \ref{SiegelIntegrability}.
\begin{no} From now on, $(X,Y, \mu)$ denotes a $p$-integrable $(G,H)$-transverse system for some $p \geq 1$ with $H$-transverse measure $\sigma$. Integrability implies that $Y_x$ is locally finite for $\mu$-almost every $x \in X$, hence for all $f \in C_c(H \backslash G)$ and $\mu$-almost every $x \in X$, the sum
\[
Sf(x) := \sum_{Hg \in Y_x} f(Hg) 
\]
is finite. Moreover, 
\[
\int_X \left| Sf(x) \right|^p \dd \mu(x) \leq \|f\|_\infty \int |Y_x \cap \mathrm{supp}(f)|^p \dd \mu(x) < \infty,
\]
hence we obtain a well-defined \emph{Siegel-Radon transform}
\begin{equation}\label{SiegelDefiniteVersion}
S: C_c(H \backslash G) \to L^p(X, \mu), \quad Sf(x) := \sum_{Hg \in Y_x} f(Hg).
\end{equation}
\end{no}
\begin{no} If $\xi$ is an eigencharacter of $(Y, \sigma)$, then we can define a twisted version of \eqref{SiegelDefiniteVersion}. For this we fix an eigenfunction $\psi_\xi$ with domain $Y_0$. Denote by $X_0$ the set of all $x \in X$ such that $Y_x$ is locally finite and coincides with $(Y_0)_x$. If $x \in X_0$ and $Hg \in Y_x$, then $g.x \in Y_0$, $\psi_\xi(g.x)$ is defined.  Moreover, if $h \in H$, then
\[
f(hg)\psi_\xi(hg.x) = \xi(h)f(g)\xi(h^{-1})\psi_x(g.x) = f(g)\psi_\xi(g.x), 
\]
hence the function $Hg \mapsto f(g)\psi_\xi(g.x)$ is well-defined. We may thus define
\[
S_{\psi_{\xi}}f: X_0 \to \bC, \quad x \mapsto \sum_{Hg \in Y_x}  f(g)\psi_\xi(g.x).
\]
\end{no}
\begin{lemma}\label{TwistedWellDef} The subset $X_0 \subset X$ is a $\mu$-conull set and $S_{\psi_\xi} f \in L^p(X, \mu)$.
\end{lemma}
\begin{proof} Since $Y$ is integrable, the condition that $Y_x$ be locally finite is a $\mu$-conull condition. To see that $X_0$ is conull it remains to show that for $\mu$-almost every $x \in X$ we have $Y_x = (Y_0)_x$. For this let $N := Y \setminus Y_0$; this is an $H$-invariant $\sigma$-nullset and we have to show that $N_x = Y_x \setminus (Y_0)_x = \emptyset$ for $\mu$-almost all $x \in X$. Since $H\backslash G$ is $\sigma$-compact, it suffices to show that for every compact subset $K \subset H \backslash G$ we have $N_x \cap K = \emptyset$ for $\mu$-almost every $x \in X$. We thus fix a compact subset $K \subset H \backslash G$; by \cite{Raghunathan} there then exists a non-negative function $\rho \in C_c(G)$ such that 
\[
\widetilde{\rho}(Hg) = \int_H \rho(hg)\dd m_H(h) = 1 \quad \text{for all }Hg \in K.
\]
By \eqref{SigmaDefEq} we then have (using that $N$ is an $H$-invariant null set)
\begin{align*}
0 &= \sigma(N) \cdot m_G(\rho) = \int_G \int_Y \rho\otimes \chi_N(g,y) \dd \sigma(y) \dd m_G(g) \\
& = \int_X\left( \sum_{Hg \in Yx} \int_H \rho(hg)\chi_N(hg.x)\dd m_H(h) \right) \dd \mu(x) =  \int_X\left(\sum_{Hg \in Y_x} \widetilde{\rho}(g)\chi_N(g.x) \right) \dd \mu(x).
\end{align*}
Restricting the sum to $Y_x \cap K$ then yields
\[
0 =  \int_X\left(\sum_{Hg \cap K \in Y_x} \widetilde{\rho}(g)\chi_N(g.x) \right) \dd \mu(x) = \int_X |N_x \cap K| \dd \mu(x).
\]
This proves that $X_0$ is conull, and we deduce that
\[
\int_X \left| S_{\psi_{\xi}}f(x) \right|^p \dd \mu(x) = \int_{X_0}\left |S_{\psi_{\xi}}f(x) \right|^p \dd \mu(x) \leq \|f\|_\infty \|\psi_\xi\|_\infty \int_{X_0} |Y_x \cap \mathrm{supp}(f)|^p \dd \mu(x) < \infty.\qedhere
\]
\end{proof}
\begin{definition} The map $S_{\psi_\xi}:  {}^0\mathrm{Ind}_H^G(\xi) \to L^p(X, \mu)$ is called the \emph{$\psi_\xi$-twisted Siegel-Radon transform} of the $p$-integrable $(G,H)$-transverse system $(X,Y, \mu)$.
\end{definition}
Since $S_{\psi_\xi}$ is $G$-equivariant by \eqref{YxEquiv}, we have at this point established Part (i) of Theorem \ref{RadonTwisted} and the integrability claims made in \S \ref{SiegelIntegrability}.
\begin{remark}\label{WlogStrict} It follows from Lemma \ref{TwistedWellDef} that $S_{\psi_\xi}$ does not change if we modify $\psi_\xi$ on a conull set. We may thus assume that $\psi_\xi$ is a strict eigenfunction to simplify notation.
\end{remark}

\subsection{Twisted orbital integrals and twisted Siegel duality} 
We now conclude the proof of Theorem \ref{RadonTwisted} by establishing Part (ii) and (iii) of the theorem.
\begin{no}
Let $(X,Y, \mu)$ be an integrable $(G,H)$-transverse system with $H$-transverse measure $\sigma$. We denote by $\widetilde{X}$ the associated intersection space and by $\widetilde{\mu}$ the associated intersection measure (see Example \ref{IntersectionInducedY}).
 We fix an $H$-eigencharacter $\xi$ of $Y$ with eigenfunction $\psi_\xi$. 
 \end{no}
\begin{definition} If  $\phi \in \cL^\infty(X)$ the \emph{$\psi_\xi$-twisted orbital integral} of $\phi$ along $Y$ is defined as
\[
S_{\psi_\xi}^*\phi(g) := \int_Y \phi(g^{-1}.y) \overline{\psi_\xi(y)} \dd\sigma(y) \in \cL^\infty(G).
\]
\end{definition}
\begin{no} If $\phi \in \cL^\infty(X)$, then for all $g,g' \in G$ and $h \in H$ we have
\[
S^*_{\psi_\xi}(\phi)(gg') = S^*_{\psi_\xi}(g'.\phi)(g) \qand  S^*_{\psi_\xi}\phi(hg) = \xi(h) S^*_{\psi_\xi}\phi(g). 
\]
In particular, $|S^*_{\psi_\xi}\phi|$ descends to a bounded function on $H \backslash G$, and for every $f \in {}^0\mathrm{Ind}_H^G(\xi)$ the product $f \cdot \overline{S^*_{\psi_\xi}\phi}$ descends to an integrable function on $H \backslash G$, hence we obtain a well-defined pairing
\[
\langle f,S_{\psi_\xi}^*\varphi \rangle_{H \backslash G} :=  \int_{H \backslash G}  f(g)\overline{S^*\psi_\xi(g)} \dd m_G(Hg).
\]
We also define
\[
\langle S_{\psi_\xi} f, \varphi \rangle_X  := \int_X  S_{\psi_\xi} f \cdot \overline{\varphi} \dd \mu.
\]
\end{no}
\begin{theorem}[Twisted Siegel duality]\label{SiegelDuality} For all $f \in {}^0\mathrm{Ind}_H^G(\xi)$ and $\phi \in \cL^\infty(X)$ we have
\[
\langle S_{\psi_\xi} f, \varphi \rangle_X  = \langle f,S_{\psi_\xi}^*\varphi \rangle_{H \backslash G}.
\]
\end{theorem}
\begin{proof} By Remark \ref{WlogStrict} we assume that $\psi_\xi$ is strict. We then abbreviate \[F_{f,\varphi}(g,y) = f(g) \overline{\varphi(g^{-1}.y)} \psi_\xi(y),\] so that by definition
\[
\langle f,S_{\psi_\xi}^*\varphi \rangle_{H \backslash G} =  \int_{H \backslash G} \int_Y F_{f,\varphi}(g,y) \dd \sigma(y) \dd m_{H \backslash G}(Hg).
\]
Since $F_{f,\varphi}(hg,h.y) = F_{f,\varphi}(g,y)$ for all $h \in H$ and $(g,y) \in G \times Y$, Corollary \ref{GHIntegralSigma} yields
\[
\langle f,S_{\psi_\xi}^*\varphi \rangle_{H \backslash G}  = \int_X \sum_{Hg \in Y_x} F_{f,\varphi}(g,g.x) \dd \mu(x)
=   \int_X \sum_{Hg \in Y_x} f(g) \overline{\varphi(x)} \psi_\xi(g.x) \dd \mu(x)  = \langle S_{\psi_\xi} f, \varphi \rangle_X.\qedhere
\]
\end{proof}
For the following reformulation, recall that $s:H \backslash G \to G$ denotes a locally bounded Borel section of $\pi: G \to H\backslash G$.
\begin{corollary}[Intersection formula] For every $\phi \in \cL^\infty(X)$ and $f \in {}^0\mathrm{Ind}_H^G(\xi)$ we have
\begin{equation}\label{IntForm}
 \langle S_{\psi_\xi} f, \varphi \rangle_X  = \langle f,S_{\psi_\xi}^*\varphi \rangle_{H \backslash G} =  \int_{\widetilde{X}} f(s(Hg))\psi_\xi(s(Hg).x)\overline{\phi(x)}\dd \widetilde{\mu}(Hg,x).
\end{equation}
Moreover, if $\xi$ is a non-trivial character, then
\begin{equation}\label{TwistedSiegelFormula}
\int_X S_{\psi_\xi} f(x) \dd \mu(x) = \langle S_{\psi_\xi} f, 1 \rangle_X  =  0.
\end{equation}
\end{corollary}
\begin{proof} The first equality in \eqref{IntForm} is just Siegel duality. For the second equality we note that by definition of $\widetilde{\mu}$ the integral on the right hand side is given by
\[
 \int_{H \backslash G}\left( \int_Y f(s(Hg))\psi_\xi(s(Hg).g^{-1}y)\overline{\phi(g^{-1}y)}\dd \sigma(y) \right) \dd m_{H \backslash G}(Hg).
\]
Recall that $\ell(g) = s(Hg)$. By \eqref{ell} we have
\[
f(\ell(g)) \psi_\xi(\ell(g).g^{-1}y)  = \xi(h(g))^{-1}f(g)  \psi_\xi(\ell(g).g^{-1}y) = f(g) \psi_\xi(h(g)\ell(g)g^{-1}.y) = f(g)\psi_\xi(y),
\]
ans hence the inner integral equals
 \[
\int_Y f(g) \psi_\xi(y) \overline{\phi(g^{-1}y)}\dd \sigma(y)  = f(g) \overline{S^*\phi(g)},
 \]
which proves \eqref{IntForm}. If $\phi \equiv 1$, then for all $h \in H$ the inner integral satisfies
\[
I_\xi(f) := \int_Y f(g) \psi_\xi(y) \dd \sigma(y) =\int_Y f(g) \psi_\xi(h^{-1}.y) \dd \sigma(y) = \xi(h) I_\xi(f),
\]
which for nontrivial $\xi$ forces  $I_\xi(f) = 0$ and proves \eqref{TwistedSiegelFormula}.
\end{proof}
At this point we have established Theorem \ref{RadonTwisted} and all the claims made in \S \ref{SiegelIntegrability}.

\subsection{Untwisted Radon-Siegel transform and Siegel formula}
\begin{no}
We now consider the untwisted Siegel-Radon transform \eqref{SiegelDefiniteVersion}, which corresponds to the trivial character and the constant eigenfunction $1$. We observe that this is precisely the Siegel-Radon transform from the introduction, hence Theorem \ref{RadonUntwisted} is just a special case of Theorem \ref{RadonTwisted}. Concerning the Siegel formula we observe first that the (untwisted) \emph{orbital integral} is given by
\[
S^*\phi(g) =  \int_Y \phi(g^{-1}.y) \dd\sigma(y), \quad (\phi \in \cL^\infty(X)).
\]
In particular, $S^*1 = \sigma(Y)$, and hence  Theorem \ref{RadonUntwisted} yields
\[
\int_X Sf \dd \mu = \int_X Sf \cdot 1 \dd \mu = \int_{H \backslash G} f \cdot S^*1 \dd m_{H \backslash G} = \sigma(Y) \cdot \int_{H \backslash G} f \dd m_{H\backslash G}.
\]
We have thus established Theorem \ref{RadonUntwisted} and the claims in \S \ref{SiegelFormulaConstant}.
\end{no}

\subsection{Alternative model}
\begin{no} For concrete computations it is convenient to realize the induced representation in the model from \S \ref{InducedModel2}. If we identify $ ({}^0\mathrm{Ind}_H^G(\xi), \pi_\xi)$ with $(C_c(H\backslash G), \pi_{s,\xi})$ via the intertwiner $i$ from \eqref{Interi}, then in this model of the induced representation the Siegel transform is the $\pi_{s,\xi}$-intertwiner
\begin{equation}\label{SecondComing}
S_{\psi_\xi}: C_c(H \backslash G) \to L^1(X, \mu), \quad S_{\psi_\xi}f(x) := \sum_{g \in s(Y_x)} f(Hg) \psi_\xi(g.x),
\end{equation}
where we have used \eqref{SiegelAlt1} to simplify the expression for $i^{-1}(f)$. Then \eqref{SiegelFormulaIntro} and \eqref{TwistedSiegelFormula} read as
\begin{equation}
\int_X S_{\psi_\xi} f \dd \mu = \left\{\begin{array}{rl}\sigma(Y) \cdot \int_{H \backslash G} f \dd m_{H\backslash G},& \text{if } \xi=\mathbf{1}, \psi_\xi \equiv 1,\\ 0, &\text{if } \xi \neq \mathbf{1}.\end{array}\right.
\end{equation}
We also define a version of the orbital integral by
\[
S^*_{\psi_\xi}: \cL^\infty(X, \mu) \to \cL^\infty(H \backslash G), \quad S^*_{\psi_\xi}\phi(Hg) := \int_Y \phi(s(Hg)^{-1}.y) \overline{\psi_\xi(y)} \dd\sigma(y).
\]
in order to obtain the following reformulation of Siegel duality.
\end{no}
\begin{corollary}
For all $f \in C_c(H \backslash G)$ and $\phi \in \cL^\infty(X, \mu)$ we have
\[
\int_X  (S_{\psi_\xi} f) \cdot \overline{\phi} \dd \mu = \int_{H \backslash G} f \cdot (\overline{S_{\psi_\xi^*}\phi}) \dd m_{H \backslash G}.
\]
\end{corollary}
\begin{proof} By Theorem \ref{SiegelDuality} we have 
\[
\int_X  (S_{\psi_\xi} f) \overline{\phi} \dd \mu = \int_{H \backslash G} i^{-1}(f)(g) \left( \int_Y \overline{\phi(g^{-1}.y)} {\psi_\xi(y)} \dd\sigma(y)  \right) \dd m_{H \backslash G}(Hg).
\]
Now,
\[
 i^{-1}(f)(g)\overline{\phi(g^{-1}.y)} \psi_\xi(y) =   \xi(h(g))f(Hg)\overline{\phi(g^{-1}.y)}{\psi_\xi(y)} = f(Hg) \overline{\phi(g^{-1}.y)} {\psi_\xi(h(g)^{-1}.y)},
\]
and using $H$-invariance of $\sigma$ we deduce that 
\begin{align*}
\int_X  (S_{\psi_\xi} f) \overline{\phi} \dd \mu &=  \int_{H \backslash G} f(Hg) \left(\int_Y \overline{\phi(\ell(g)^{-1}h(g)^{-1}.y)} {\psi_\xi(h(g)^{-1}.y)} \dd \sigma(y) \right) \dd m_{H \backslash G}(Hg)\\
&= \int_{H \backslash G} f(Hg) \left(\int_Y \overline{\phi(\ell(g)^{-1}.y)} {\psi_\xi(y)} \dd \sigma(y) \right) \dd m_{H \backslash G}(Hg),
\end{align*}
where the inner integral equals $\overline{S_{\psi_\xi^*}\phi(Hg)}$.
\end{proof}

\section{Induced transverse systems} \label{SecInducedSystems}

\subsection{Separated cross sections and induced transverse systems}

\begin{no}\label{HCountableIntersectionInduced} Let $Z$ be a $G$-cross section with return time set $\Lambda(Z) = \bigcup_{z \in Z} Z_z \subset G$. We then refer to $(X, Z, \mu)$ as a \emph{$G$-transverse system}.
We say that $Z$ is  \emph{$H$-countable} if the restriction of the action map $a_X$ to $H \times Z$ has countable fibers. By \cite[Theorem 18.10]{K} this implies that
\[
Y := H.Z \subset X
\]
is a Borel subset, hence a hence a $(G,H)$-cross section, called the \emph{induced $(G,H)$-cross section} of $Z$. By definition we have
\begin{equation}\label{YxZx}Y_x = \pi(Z_x) \quad \text{for all $x \in X$},
\end{equation}
from which we deduce that
\begin{equation}\label{IntReturn1}
\widetilde{Z}_{(Ha,x)} = Ha \cap Z_x \qand \widetilde{Y}_{(Ha, x)} = \pi(\widetilde{Z}_{Ha, x}) \quad (a \in G, x \in X),
\end{equation}
where $(\widetilde{X}, \widetilde{Y})$ denotes the intersection space of $H\backslash G$ with $(X,Y)$ and $\widetilde{Z} := \{H\} \times Z$.
\end{no}
\begin{no} We recall that a $G$-cross section $Z \subset X$ is called \emph{separated} if $e_G$ is an isolated point of its return time set $\Lambda(Z)$. This implies in particular that $Z$ is $H$-countable, hence we can consider the induced $(G,H)$-cross section $Y = H.Z$. We then have $Z \subset Y \subset X$, and $Z$ is a separated $H$-cross section for $Y$. If $\mu \in \mathrm{Prob}(X)^G$, then we refer to $(X,Z, \mu)$ as a \emph{separated $G$-transverse system} and to $(X,Y, \mu)$ as the  \emph{induced $(G,H)$-transverse system}.
\end{no}
\begin{remark}\label{TildeZSep} If $Z \subset X$ is separated and $Y := H.Z \subset X$ with intersection space $\widetilde{X}$, then $\widetilde{Z} \subset \widetilde{X}$ is separated by \eqref{IntReturn1}, and $H.\widetilde{Z}= \widetilde{Y}$.
\end{remark}
\begin{no}\label{InjSets}
Following \cite{BHK}, we say that a subset $C \subset G \times X$ is \emph{injective} if the action map $a_X: G \times X \to X$ is injective on $C$. If $Z \subset X$ is a separated $G$-cross section, then by \cite{BHK} there exist an identity neighbourhood $U$ in $G$ such that $U \times Z$ is injective and a countable collection $\mathcal C = (C_k)_{k \in \bN}$ of pairwise disjoint injective Borel subsets of $G \times Z$ such that $X = \bigsqcup_{k=1}^\infty a_X(C_k)$.
\end{no}
\subsection{Transverse measure theory of separated transverse systems}

\begin{no}\label{InducedMeasure} The transverse measure theory of \emph{separated} $G$-transverse systems was developed in \cite{BHK}, from which we recall the following basic facts: 

\item Firstly, every separated transverse system is integrable, and the transverse measure $\nu$ is the unique finite Borel measure on $Z$ such that
\begin{equation}\label{PalmLocal}
\mu(a_X(C)) = (m_G \otimes \nu)(C) \quad \text{for every  injective Borel set $C \subset G \times Z$.}
\end{equation}
More generally, if $\mu$ is a $(G,Z)$-finite $G$-invariant measure on $X$, then there exists a unique finite Borel measure on $Z$ which satisfies \eqref{PalmLocal} and is still called the transverse measure of $\mu$. It follows from \eqref{PalmLocal} that assigns which every  $(G,Z)$-finite $G$-invariant measure on $X$ its transverse measure is injective.

\item Secondly, if $\cO_G$ denotes the $G$-orbit equivalence relation on $X$, then $\nu$ is $\mathcal O_{G}|_Z$-invariant. Conversely, for every $\mathcal O_{G}|_{Z}$-invariant finite Borel measure $\nu$ on $Z$ there exists a unique Borel measure $\mu$ on $X$ which is $G$-invariant, $(G,Z)$-finite (but not necessarily finite) and satisfies
\eqref{PalmLocal}; this is called the \emph{induced measure} of $\nu$.
\end{no}
\begin{lemma}\label{DoubleRestriction} 
Let $Z \subset \widetilde{Z}$ be separated cross sections for $X$ and $\mu \in \mathrm{Prob}(X)^G$. If $\nu_Z$ and $\nu_{\widetilde{Z}}$ denote the corresponding transverse measures on $Z$ and $\widetilde{Z}$ respectively, then
\[
\mu_Z= \mu_{\widetilde{Z}}|_Z.
\] 
\end{lemma}
\begin{proof} Since $\widetilde{Z}$ is separated there exists an identity neighbourhood $U \subset G$ such that $U \times \widetilde{Z}$ is injective, and hence also $U \times Z$ is injective. By \eqref{PalmLocal} we then have 
\[
\mu_Z(B) = \frac{\mu(U.B)}{m_G(U)} = \mu_{\widetilde{Z}}(B) \quad \text{for any Borel subset $B \subset Z$.}\qedhere
\]
\end{proof}
\subsection{Couplings and thinnings of separated transverse systems}\label{SecCouplings}

\begin{no}
Let $(X_1,Z_1,\mu_1)$ and $(X_2,Z_2,\mu_2)$ be two $G$-transverse systems. A probability measure $\mu \in \mathrm{Prob}(X_1 \times X_2)$ is called a \emph{coupling} if it pushes forward to $\mu_1$ and $\mu_2$ respectively under the canonical projections. A coupling is called \emph{monotone} if it gives full measure to the Borel subset
\[
\cM = \{ (x_1,x_2) \in X_1 \times X_2 \, : \, (Z_1)_{x_1} \subseteq (Z_2)_{x_2} \} \subset X_1 \times X_2
\]
is a Borel subset, and a \emph{monotone joining} if it is moreover $G$-invariant. We say that $(X_1,Z_1,\mu_1)$ is a \textit{thinning} of $(X_2,Z_2,\mu_2)$ if such a monotone joining exists.

\item It follows from \eqref{YxZx} that every monotone coupling of two separated $G$-transverse systems $(X_1,Z_1,\mu_1)$ and $(X_2,Z_2,\mu_2)$ is also a monotone coupling of the induced systems $(X_1, Y_1, \mu_1)$ and $(X_2, Y_2, \mu_2)$. In particular, if $(X_1, Z_1, \mu_1)$ is a thinning of $(X_2, Z_2, \mu_2)$, then $(X_1, Y_1, \mu_1)$ is a thinning of $(X_2, Y_2, \mu_2)$. The following lemma provides a useful way to construct thinnings.
\end{no}
\begin{lemma}\label{GeneralThinning}
Suppose there exists a Borel map $p : X_1 \ra X_2$ such that 
\[
p_*\mu_1 = \mu_2 \qand Z_1 \subseteq p^{-1}(Z_2).
\]
Then $(X_1,Z_1,\mu_1)$ is a thinning of $(X_2,Z_2,\mu_2)$, and if the latter is separated, then so is the former.
\end{lemma}

\begin{proof}
Define $\eta \in \Prob(X_1 \times  X_2)^G$ by
\[
\eta(f) = \int_{X_1} f(x_1,p(x_1)) \, d\mu_1(x_2), \quad f \in \mathscr{L}^\infty(X_1 \times X_2).
\]
Since $p_*\mu_1 = \mu_2$, we deduce that $\eta$ is a joining of $(X_1,\mu_1)$ and $(X_2,\mu_2)$. 
Note that
\[
(Z_1)_{x_1} \subseteq (p^{-1}(Z_2))_{x_1} = (Z_2)_{p(x_1)} \quad \textrm{for all $x_1 \in X_1$},
\]
and thus $(x_1,p(x_1)) \in \cM$ for all $x_1 \in X_1$. In particular, the joining $\eta$ is monotone, and $\Lambda(Z_1) \subset \Lambda(Z_2)$, hence $Z_1$ is separated, hence $Z_1$ is separated if $Z_2$ is.

\end{proof}
Lemma \ref{GeneralThinning} will be particularly useful for systems arisings from cut-and-project schemes, see Subsection \ref{SecCupThinning} below, but it also provides a method to construct thinnings for general systems which we now describe. 
\begin{construction}\label{ProductThinning} We can construct a thinning of any $G$-ergodic separated $G$-transverse system $(X_2,Z_2,\mu_2)$ using an arbitrary auxiliary weakly mixing separated $G$-transverse system $(X_3,Z_3,\mu_3)$. For this we choose a  symmetric identity neighbourhood $V \subset G$ and define a new system $(X_1, Z_1, \mu_1)$ as follows: We define subsets of $X_2 \times X_3$ by
\[
Z_1 = Z_2 \times V^2.Z_3 \qand X_1 := G.Z_1.
\]
To define the measure $\mu_1$ we use the following observation.
\end{construction}
\begin{lemma}\label{ProductThinningM1} In the situation of Construction \ref{ProductThinning} we have $\mu_2 \otimes \mu_3(X_1) = 1$.
\end{lemma}
\begin{proof} Since $(X_2,\mu_2)$ is ergodic and $(X_3,\mu_3)$ is weakly mixing, the product system $(X_2 \times X_3,\mu_2 \otimes \mu_3)$ is ergodic, hence it suffices to show that $(\mu_2\otimes \mu_3)(X_1) > 0$. This in turn readily follows 
from the inclusion
\[
G.(Z_2 \times V^2.Z_3) \supseteq V.Z_2 \times V.Z_3,
\]
and the fact that $\mu_i(V.Z_i) > 0$ for $i=2,3$.
\end{proof}
In view of the lemma we can now define $\mu_1 := (\mu_2 \otimes \mu_3)|_{X_1}$. 
\begin{corollary} $(X_1, Z_1, \mu_1)$ is a separated $G$-transverse system and a thinning of $(X_2, Z_2, \mu_2)$.
\end{corollary}
\begin{proof} By construction, $Z_1$ is a $G$-cross section of $X_1$, and by Lemma \ref{ProductThinningM1} the measure $\mu_1$ is a probability measure. We now consider the map $p: X_2 \to X_1$ given by $p(x_2, x_2) := x_2$. By construction we have $p_*\mu_1 = \mu_2$ and $Z_1 \subseteq p^{-1}(Z_2)$ so that Lemma \ref{GeneralThinning} applies. 
\end{proof}
We call $(X_1, Z_1, \mu)$ the \emph{product thinning} of $(X_2, Z_2, \mu)$ constructed from $(X_3, Z_3, \mu)$ and $V$.
\subsection{Partitions of unity}
Throughout this subsection let $(X,Z)$ be an integrable transverse $G$-system.
\begin{definition} A Borel function $\rho: G \times Z \to [0,1]$ is called a \emph{partition of unity} for $(X, Z)$ if $T\rho(x) = 1$ holds for all $x \in X$.
\end{definition}
As a special case of \eqref{TransverseMeasure1} we observe:
\begin{corollary} If $(X,Z)$ admits a partition of unity $\rho$, then for all $\phi \in \mathcal L^\infty(X)$ we have
\begin{equation}\label{PartUn}
\pushQED{\qed}\int_X \phi \dd \mu = \int_G \int_Z \phi(g.z) \rho(g,z) \dd \nu(z) \dd m_G(g).\qedhere\popQED
\end{equation}
\end{corollary}

\begin{proposition}\label{ExPU} If $Z \subset X$ is separated, then  $(X,Z)$ admits a partition of unity.
\end{proposition}
\begin{proof} Let $\cC=(C_k)_{k \in \bN}$ be as in \S \ref{InjSets} and define 
\[
\rho_{\mathcal C}: G \times Z \to \{0,1\}, \quad (g,z) \mapsto \sum_{k=1}^\infty \chi_{C_k}(g,z).
\]
We claim that this defines a partition of unity, i.e.\ that 
\begin{equation}\label{PartitionOfUnityToShow}
 \sum_{g \in Z_x} \sum_{k=1}^\infty \chi_{C_k}(g^{-1}, g.x) = 1 \quad \text{for all } x \in X.
\end{equation}
For the proof of \eqref{PartitionOfUnityToShow} we fix some $x \in X$. By assumption there exist a unique $k(x)$ such that $x \in a_X(C_{k(x)})$. By injectivity there then exists a unique pair $(h(x),z(x)) \in C_{k(x)}$ such that $x = h(x).z(x)$. If we set $g(x) := h(x)^{-1}$ , then $g(x).x = z(x)$ and hence $g(x) \in Z_x$. Moreover,
\[
\chi_{C_{k(x)}}(g(x)^{-1}, g(x).x) = \chi_{C_{k(x)}}(h(x), z(x)) = 1,
\]
i.e.\ one of the summmands in \eqref{PartitionOfUnityToShow} equals $1$. We have to show that all other summands are $0$. For this assume that $\chi_{C_k}(g^{-1}, g.x) = 1$ for some $g \in Z_x$ and $k \in \bN$. If we set $h := g^{-1}$ and $z := g.x$, then $(h,z) \in C_k$ and $x = h.z \in a_X(C_k)$, which implies $k = k(x)$. Moreover, since $h.z = x = h(x).z(x)$, injectivity of $a_X|_{C_{k(x)}}$ implies that $h = h(x)$ and hence $g = g(x)$. This shows that $\chi_{C_{k}}(g^{-1}, g.x)=0$ unless $k = k(x)$ and $g = g(x)$.
\end{proof}

\section{Integrability of induced cross sections} \label{InducedIntegrability} 
\begin{no} Throughout this section let $(X,Z, \mu)$ be a separated $G$-transverse system with transverse measure $\nu$. We will be concerned with conditions which guarantee that the induced $(G,H)$-transverse system $(X,Y, \mu)$ is integrable. We first provide an abstract characterization of integrability, which is however hard to apply in praxis. We then establish Theorem \ref{IntegrabilityIntro}, which provides a more tractable sufficient conditions for integrability which is motivated from our applications to aperiodic order. 
\end{no}
\subsection{An abstract integrability criterion}

\begin{no}
By \S \ref{InducedMeasure}, $Z$ is integrable, and if $\cO_G$ and $\cO_H$  denote the orbit equivalence relation of $G$ and $H$ on $X$, then the transverse measure $\nu$ is $\cO_{G}|_Z$-invariant, hence in particular $\cO_H|_Z$-invariant. It thus induces an $H$-invariant $(H,Z)$-finite Borel measure $\tau$ on $Y$, which will in general not be finite.
\end{no}
\begin{theorem}\label{IntegrableInduced} The induced cross section $Y$ is integrable if and only if $\tau(Y) < \infty$, and in this case its $H$-transverse measure $\sigma$ is given by $\sigma = \tau$.
\end{theorem}
The remainder of this subsection is devoted to the proof of Theorem \ref{IntegrableInduced}.
\begin{no}
Denote by $\widetilde{X} \subset H \backslash G \times X$ the intersection space of $H \backslash G$ and $(X,Y)$ and let $\widetilde{Z}\subset \widetilde{Y} \subset \widetilde{X}$ as in \S \ref{HCountableIntersectionInduced}. By Remark \ref{TildeZSep}, $\widetilde{Z} \subset \widetilde{X}$ is separated, which has two consequences: By \S \ref{InducedMeasure} it is integrable and by Proposition \ref{ExPU} it admits a partition of unity $\rho: G \times \widetilde{Z} \to [0,1]$. We now fix such a partition once and for all.

Since the measure $\tau$ is $(H,Z)$-finite and $H$-invariant, the associated intersection measure $\widetilde{\mu}_\tau$ given by \eqref{Defmutau} is $G$-invariant and $(\widetilde{X}, \widetilde{Z})$-finite. It thus admits a transverse measure $\widetilde{\nu}$, which is a finite measure on $\widetilde{Z}$, which we will identify explicitly in Lemma \ref{Transversemutau} below. We need:
\end{no}
\begin{lemma}\label{ZxYxDec} For all $x \in X$ we have
$
Z_x =  \bigsqcup_{Hg \in Y_x} \widetilde{Z}_{(Hg,x)}$.
\end{lemma}
\begin{proof} By \eqref{YxZx} we have
$Z_x = \bigsqcup_{Hg \in Y_x} Z_x \cap Hg$. If  $Hg \in Y_x$, then $(Hg, x)= g^{-1}.(H,g.x) \in G.(\{H\} \times Y) = \widetilde{X}$, so that $ \widetilde{Z}_{(Hg,x)} $ is well-defined. 
For all $a \in G$ we then have
\[
a \in \widetilde{Z}_{(Hg,x)} 
\iff ag^{-1} \in H \enskip \textrm{and} \enskip a \in Z_x \iff
a \in Z_x \cap Hg.\qedhere
\]
\end{proof}
\begin{corollary} For every bounded measurable function $f: \widetilde{X} \to [0, \infty)$ we have
\begin{equation}\label{SumOverPartition}
\sum_{Hg \in Y_x} f(Hg, x) = \sum_{a \in Z_x} f(Ha, x) \rho(a^{-1}, (H, a.x)).
\end{equation}
\end{corollary}
\begin{proof} Note that if $a \in \widetilde{Z}_{(Hg,x)}$, then $Ha = Hg$, hence by Lemma \ref{ZxYxDec} the sum on the right-hand side of \eqref{SumOverPartition} equals
\[
\sum_{Hg \in Y_x} \sum_{a \in \widetilde{Z}_{(Hg, x)}} f(Ha,x) \rho(a^{-1}, (H, a.x)) =  \sum_{Hg \in Y_x} f(Hg,x) \sum_{a \in \widetilde{Z}_{(Hg, x)}} \rho(a^{-1}, (H, a.x)).
\]
On the other hand, for every fixed $Hg \in Y_x$ we have
\[
 \sum_{a \in \widetilde{Z}_{(Hg, x)}} \rho(a^{-1}, (H,a.x)) =  \sum_{a \in \widetilde{Z}_{(Hg, x)}} \rho(a^{-1}, a.(Hg,x)) =1. \qedhere
\]
\end{proof}
\begin{lemma}\label{Transversemutau} The transverse measure of the $G$-transverse system $(\widetilde{X}, \widetilde{Z}, \widetilde{\mu}_\tau)$ is given by $\widetilde{\nu}=\delta_H \otimes \nu$.
\end{lemma}
\begin{proof} We have to show that for every $f \in \cL_c^\infty(G \times \widetilde{Z})$,
\[
I_f := \int_{\widetilde{X}} \Big( \sum_{a \in \widetilde{Z}_{(Hg,x)}} f(a^{-1},a.(Hg,x)) \Big) \dd\widetilde{\mu}_\tau(Hg,x) = \int_G \int_Z f(g, (H,z)) \dd\nu(z) \dd m_G(g).
\]
Now let $a,g \in G$ and $y \in Y$. If we set $h := ag^{-1}$, then we have $a = hg$ and
\[
a \in  \widetilde{Z}_{(Hg,g^{-1}.y)} \iff h \in Z_y \cap H \implies (a^{-1}, a.(Hg, g^{-1}y)) = (g^{-1}h^{-1},(H, h.y)).
\]
Using \eqref{Defmutau} we thus obtain
\begin{align*} I_f
&= \int_{H \backslash G} \int_Y \sum_{a \in \widetilde{Z}_{(Hg, g^{-1}.y)}} f(a^{-1}, a.(H, g^{-1}.y)) \dd \tau(y) \dd m_{H \backslash G}(Hg)\\
&=  \int_{H \backslash G} \int_Y \sum_{h \in Z_y \cap H} f(g^{-1}h^{-1}, (H,h.y)) \dd \tau(y) \dd m_{H \backslash G}(Hg).
\end{align*}
If we abbreviate $f_g(h,z) := f(g^{-1}h, (H,z))$ and apply \eqref{TransverseMeasure1} (with $\varphi \equiv 1$) to the $H$-cross section $Z$ of $Y$, then we obtain
\begin{align*}
I_f &=  \int_{H \backslash G} \int_Y \sum_{h \in Z_y \cap H} f_g(h^{-1}, h.y) \dd \tau(y) \dd m_{H \backslash G}(Hg)\\
&= \int_{H \backslash G} \int_H \int_Z \ f_g(h, z)\dd \nu(z) \dd m_H(h)  \dd m_{H \backslash G}(Hg)\\
&= \int_{H \backslash G} \int_H \int_Z \ f(g^{-1}h, (H,z))\dd \nu(z) \dd m_H(h)  \dd m_{H \backslash G}(Hg)
\end{align*}
Using unimodularity of $H$ and $G$ and \eqref{Weil} we obtain
\begin{align*}
I_f  &= \int_{H \backslash G} \int_H \int_Z \ f(g^{-1}h^{-1}, (H,z))\dd \nu(z) \dd m_H(h)  \dd m_{H \backslash G}(Hg)\\
&= \int_G \int_Z f(g^{-1}, (H,z)) \dd \nu(z)\dd m_G(g) = \int_G \int_Z f(g, (H,z)) \dd\nu(z) \dd m_G(g).\qedhere
\end{align*}
\end{proof}
\begin{corollary} For every bounded measurable function $f: \widetilde{X} \to [0, \infty)$ we have
\begin{equation}\label{IntMainFormula}
\int_{\widetilde{X}} f \dd \widetilde{\mu}_\tau = 
\int_X \Big( \sum_{Hg \in Y_x} f(Hg,x) \Big) \dd\mu(x) \; \in [0, \infty].
\end{equation}
\end{corollary}
\begin{proof} Given $f$ as in the corollary we then set \[F(g,z) := f(Hg^{-1}, g.z) \rho(g, (H,z)), \quad \text{so that} \quad F(a^{-1},a.x) = f(Ha, x) \rho(a^{-1},(H, a.x)).\]
Using \eqref{SumOverPartition} and  \eqref{TransverseMeasure1} (with $\phi \equiv 1$) we then have
\begin{align*}
I_f &:= \int_X \Big( \sum_{Hg \in Y_x} f(Hg,x) \Big) \dd\mu(x) \;= \; \int_X \sum_{a \in Z_x} f(Ha,x)\rho(a^{-1}, (H,a.x)) \dd \mu(x) =  \int_X \sum_{a \in Z_x} F(a^{-1}, a.x) \dd \mu(x)\\
&= \int_G \int_Z F(g,z) \dd \nu(z) \dd m_G(g) = \int_G \int_Zf(Hg^{-1}, g.z) \rho(g, (H,z))  \dd \nu(z) \dd m_G(g)\\ &=  \int_G \int_{\widetilde{Z}} f(g.(H, z)) \rho(g, (Hx,z))  \dd (\delta_H \otimes \nu)(Hx, z) \dd m_G(g).
\end{align*}
Since $\delta_H \otimes \nu$ is the transverse measure of $\widetilde{\mu}_\tau$, the corollary follows from \eqref{PartUn}.
\end{proof}
\begin{corollary} The cross section $Y$ is integrable if and only if $\tau(Y) < \infty$. If this is the case, then
for every bounded Borel subset $B \subset H \backslash G$ we have 
\begin{equation}\label{IntegralTau}
\int_X |Y_x \cap B| \dd \mu(x) = \tau(Y) \cdot m_{H \backslash G}(B).
\end{equation}
\end{corollary}
\begin{proof}
Given a bounded set $B$ we apply \eqref{IntMainFormula} to the function $f_B(Hg, x):= \chi_{B}(Hg)$. In view of \eqref{Defmutau} the left-hand side of \eqref{IntMainFormula} then becomes
\[
\int_{\widetilde{X}} f_B \dd \widetilde{\mu}_\tau = \int_{H \backslash G} \Big( \int_Y  \chi_B({Hg}) \, d\tau(y) \Big) \dd m_{H \backslash G}(Hg) = \tau(Y) \cdot m_{H \backslash G}(B),
\] 
and the right hand side becomes
\[
\int_X \Big( \sum_{Hg \in Y_x} \chi_B(Hg) \Big) \dd \mu(x) = \int_X |Y_x \cap Hg| \dd \mu(x).
\]
This establishes \eqref{IntegralTau}, and hence $Y$ is integrable if and only if $\tau(Y) < \infty$. 
\end{proof}
\begin{proposition} If $Y$ is integrable, then the $H$-transverse measure $\sigma$ is given by $\sigma = \tau$.
\end{proposition}
\begin{proof} Let $\varphi : H \times Z \ra \bC$ be a Borel function such that $\{ \varphi \neq 0\} \subset L \times Y$ for some compact set $L \subset H$. We need to show
that \[
I := \int_Y \Big( \sum_{h \in Z_y \cap H} \varphi(h^{-1},h.y) \Big) \dd\sigma(y) 
\overset{!}{=} \int_Y \Big( \sum_{h \in Z_y \cap H} \varphi(h^{-1},h.y) \Big) \dd\tau(y) =  \int_H \int_Z \varphi(h,z) \dd\nu(z) \dd m_H(h).
\]
For this we fix $u \in C_c(G)^{+}$ with $m_G(u) = 1$ and $\varphi : H \times Z \ra \bC$ as above, and define
\[
f(g,y) = u(g) \sum_{h \in Z_y \cap H} \varphi(h^{-1},h.y).
\]
Then, by \eqref{SigmaDefEq} we have
\begin{align*}
I &= m_G(u) \cdot I = \int_G \int_Y f(g,y) \dd \sigma(y) \dd m_G(g) = \int_X \Big( \sum_{Hg \in Y_x} \int_H f(g^{-1}h^{-1},hg.x) \dd m_H(h) \Big) \dd\mu(x)\\
&=\int_X \sum_{Hg \in Y_x} \int_H \Big(  \sum_{h' \in Z_{hg.x} \cap H} \check{u}(hg)  \varphi((h')^{-1},h'hg.x) \Big) \, dm_H(h) \, d\mu(x) 
\end{align*}
For $h,h' \in H$, $g \in G$ and $x \in X$ we then have
\begin{align*}
h' \in Z_{hg.x} \cap H \iff h'hg \in Z_x \cap Hg = \widetilde{Z}_{(Hg,x)},
\end{align*}
and thus using the substitution $a := h'hg$ (i.e.\ $h' = a(hg)^{-1}$) we can rewrite the inner sum as
\begin{align*}
\sum_{h' \in Z_{hg.x} \cap H} \check{u}(hg)  \varphi((h')^{-1},h'hg.x)
&= 
\sum_{a \in \widetilde{Z}_{(Hg,x)}} \check{u}(hg) \varphi(a(hg)^{-1},a.x) 
=
\sum_{a \in \widetilde{Z}_{(Hg,x)}} \check{u}(hg) \varphi(ag^{-1}h^{-1},a.x).
\end{align*}
If $a \in \widetilde{Z}_{(Hg,x)}$, then $a.x \in Z$ and $ga^{-1} \in H$, which implies $\mathrm{d}m_H(h) = \mathrm{d}m_H(hga^{-1})$. If we define
\[
\psi : G \times \widetilde{Z} \ra \bC, \quad \psi(a,(H,z)) = \int_H \check{u}(ha^{-1}) \varphi(h^{-1},z) \, dm_H(h),
\]
and recall from Lemma \ref{ZxYxDec} that $Z_x =  \bigsqcup_{Hg \in Y_x} \widetilde{Z}_{(Hg,x)}$ we thus obtain
\begin{align*}
I &=  \int_X \sum_{Hg \in Y_x}  \Big( \sum_{a \in \widetilde{Z}_{(Hg,x)}} \int_H \check{u}((hga^{-1})a) \varphi((hga^{-1})^{-1},a.x)  \dd m_H(h) \Big) \dd\mu(x)\\
&=   \int_X \sum_{Hg \in Y_x}  \sum_{a \in \widetilde{Z}_{(Hg,x)}} \int_H \check{u}(ha) \varphi(h^{-1},a.x) \, dm_H(h) \dd\mu(x)\\
&= \int_X \sum_{Hg \in Y_x}  \sum_{a \in \widetilde{Z}_{(Hg,x)}} \psi(a^{-1}, (H, a.x))\dd\mu(x) =  \int_X \sum_{Hg \in Y_x}  \sum_{a \in \widetilde{Z}_{(Hg,x)}} \psi(a^{-1}, a.(Hg, x))\dd\mu(x)\\
&= \int_X \sum_{a\in Z_x} \psi(a^{-1}, a.(Ha, x))\dd\mu(x) = \int_X \sum_{a\in Z_x} \psi(a^{-1}, (H, a.x))\dd\mu(x).
\end{align*}
Applying \eqref{TransverseMeasure1}  (with $\phi \equiv 1$) and using $m_G(u) = 1$ and unimodularity of $G$ and $H$ we obtain
\begin{align*}
I &=  \int_G \int_Z \psi(g,(H,z)) \dd \nu(z) \dd m_G(g) = \int_G  \int_Z \int_H u(g^{-1}h^{-1}) \varphi(h^{-1},z)  \dd m_H(h) \dd \nu(z)  dm_G(g) \\[0.2cm]
&= \int_G  \int_Z \int_H u(g) \varphi(h,z)  \dd m_H(h) \dd \nu(z)  dm_G(g) 
=  \int_H \int_Z \varphi(h,z) \dd\nu(z) \, dm_H(h).\qedhere
\end{align*}
\end{proof}
This completes the proof of Theorem \ref{IntegrableInduced}.
\subsection{The cocompact case}\label{IntCoc}
\begin{no} We recall that a $G$-cross section $Z \subset X$ is called \emph{cocompact} if there exists a compact subset $D \subset G$ such that $D.Z = X$. We are going to provide an integrability criterion for systems induced from cocompact separated cross sections for which the return time set $\Lambda(Z)$ is contained in an approximate lattice. 

\item If $(X,Z, \mu)$ is a separated $G$-transverse system, then it follows from \eqref{YxZx} that the induced system $(X, Y, \mu)$ is $p$-integrable if and only if the function $f_K(x) := |\pi(Z_x) \cap \pi(K)|$ is in $L^p(X, \mu)$ for every compact subset $K \subset G$. In favourable cases, these functions can be estimated using the following combinatorial lemma.
\end{no}
\begin{lemma}
\label{Lemma_ABC}
Let $A, B \subset G$. Suppose there is a set $C \subset H$ such that $(A \cap H)C = H$.
Then,
\[
|\pi(A) \cap \pi(B)| \leq |A^{-1}A \cap BC|.
\]
\end{lemma}

\begin{proof} If $Hg \in \pi(A) \cap \pi(B)$, then there are $h_A, h_B \in H$, $a \in A$ and $b \in B$ such that
\[
g = h_A a = h_B b.
\]
In particular, $b = h_B^{-1} h_A a$. Since $(A \cap H) C = H$, and thus $C^{-1}(A \cap H)^{-1} = H$, we can find $a_H \in A \cap H$ and $c \in C$ such that $h_B^{-1} h_A = c^{-1} a_H^{-1}$. Hence, $b =c^{-1} a_H^{-1}a$, or equivalently, 
\[
a_H^{-1}a = cb \in (A \cap H)^{-1}A \cap CB \subset A^{-1}A \cap CB.
\]
Note that $Hg = Ha \in \pi(A^{-1}A \cap CB)$. Since $Hg \in \pi(A) \cap \pi(B)$ is arbitrary, we now see that
\[
|\pi(A) \cap \pi(B)| \leq |\pi(A^{-1}A \cap BC)| \leq |A^{-1}A \cap BC|,
\]
which finishes the proof.
\end{proof}
\begin{no} We recall that $(H, \Lambda)$ is called a \emph{uniform compatible pair} if $\Lambda$ is a uniform approximate lattice (i.e. a discrete relatively dense approximate subgroup of $G$) and $\Lambda \cap H$ is a uniform approximate lattice in $H$.
\end{no}
\begin{proposition}\label{CocompactCase}
Let $(H, \Lambda)$ be a uniform compatible pair and let $Z \subset X$ be a cocompact separated $G$-cross section with $\Lambda(Z) \subset \Lambda$. If $Y:= H.Z$ denotes the induced $(G,H)$-cross section, then
for every compact $K \subset G$ there exists $C_K > 0$ such that $|Y_x \cap \pi(K)|< C_K$ for all $x \in X$. In particular, $Y$ is $p$-integrable for every $p \in [1, \infty]$ with respect to every $\mu \in \mathrm{Prob}(X)^G$.
\end{proposition}
\begin{remark} The following proof of Proposition \ref{CocompactCase} actually works under slightly weaker assumptions on $\Lambda$. It suffices to assume that $\Lambda \supset \Lambda(Z)$ is symmetric, that  $\Lambda^2$ is locally finite and that $\Lambda \cap H$ is relatively dense in $H$.
\end{remark}
\begin{proof}[Proof of Proposition \ref{CocompactCase}] We have to show that the function  $f_K(x) := |\pi(Z_x) \cap \pi(K)|$ is bounded. We fix compact sets $C \subset H$ and $D \subset G$ such that $(\Lambda' \cap H)C = H$ and $D.Z = X$. Given $x \in X$ we write $x = dz$ with $d \in D$ and $z \in Z$ and obtain
\[
f_K(x) = |\pi(Z_z)d^{-1} \cap \pi(K)| = |\pi(Z_z) \cap \pi(Kd)| \leq |\pi(\Lambda(Z)) \cap \pi(KD)| \leq  |\pi(\Lambda) \cap \pi(KD)|.
\]
Applying Lemma \ref{Lemma_ABC} with $A:=\Lambda$ and $B:= KD$ (and using that $\Lambda$ is symmetric) thus yields
\[
f_K(x) \leq |\Lambda^2 \cap KDC| < \infty.\qedhere
\]
\end{proof}
At this point we have proved Theorem \ref{IntegrabilityIntro} in the case of a cocompact cross section.

\subsection{The non-cocompact case}\label{IntNoncoc}
We now discuss the proof of Theorem \ref{IntegrabilityIntro} in the non-cocompact case. In fact, we prove the result under the following slightly weaker assumptions.
\begin{proposition}\label{FiniteTau} Let $Z \subset X$ be a separated $G$-cross section. Assume that $\Lambda(Z)$ is countable and contained in a subset $\Lambda' \subset G$ with the following properties: 
\begin{enumerate}[(1)]
\item $\Lambda'$ is countable and $(\Lambda')^3$ does not accumulate at $e \in G$;
\item there exists a Borel subset $D \subset H$ with $m_H(D) < \infty$ such that $H = D(\Lambda' \cap H)$.
\end{enumerate}
Then the induced $(G,H)$-cross section $Y := H.Z$ is integrable with respect to every $\mu \in \mathrm{Prob}(X)^G$.
\end{proposition}
\begin{proof} By Theorem \ref{IntegrableInduced} we have to show that $\tau(Y)< \infty$, where $\tau$ is the unique $(H,Z)$-finite $H$-invariant measure on $Y$ with transverse measure $\nu$. For this we abbreviate $\Lambda := \Lambda(Z)$ and choose $D$ as in (2). Furthermore we set $\Lambda'_H := \Lambda' \cap H$ and $Z' := \Lambda'_H.Z \subset X$. We observe that $H.Z' = H.Z = Y$, i.e. $Z' \subset Y$ is a cross section.\\

\item \textsc{Claim 1:} The cross section $Z' \subset X$ is separated.

\item If $z' \in Z'$ and $g \in Z'_{z'}$, then there exists $\lambda_1, \lambda_2 \in \Lambda' \cap H$ and $z_1,z_2 \in Z$ such that 
$z' = \lambda_1.z_1$ and $g.z' = g\lambda_1.z = \lambda_2.z_2$, and thus
$\lambda_2^{-1}g\lambda_1 \in \Lambda$, whence
\[
g \in \lambda_2\Lambda \lambda_1^{-1} \subset (\Lambda')^3. 
\]
Since $g \in \Lambda(Z')$ is arbitrary, we conclude that $\Lambda(Z') \subset (\Lambda')^3$. By assumption, we can find an open identity neighbourhood $V$ in $G$ such that $(\Lambda')^3 \cap V = \{e\}$, and thus $\Lambda(Z') \cap V = \{e\}$, proving that $Z'$ is separated.\\

\item We denote by $\nu'$ the transverse measure of $(X, Z', \mu)$. Since $\mu$ is finite, so is $\nu'$. Since $\nu'$ is invariant under $\cO_G|_{Z'}$, it is invariant under $\cO_H|_{Z'}$ and hence induces a measure $\tau'$ on $Y$.\\

\item \textsc{Claim 2:} The measure $\tau'$ is finite.

\item We consider the function $\rho: Y \to \bN_0$ given by
\[
\rho(y) :=  \sum_{h \in Z'_{y}} \chi_{D^{-1}}(h).
\]
Since $\nu'$ is the transverse measure of $\tau'$ and $H$ is unimodular we have\
\[
\int_Y \rho \dd \tau'(y) \;=\; \int_{Y} \Big( \sum_{h \in Z'_{y}} \chi_{D}(h^{-1}) \Big) \, d\tau'(y) \;=\; \int_{H}\int_{Z'} \chi_D(h) \dd \nu'(Z') \dd m_H(h) \;=\; m_H(D) \nu'(Z'),
\]
which is finite since $\nu'$ is finite. It thus suffices to show that $\rho(y) \geq 1$ for all $y \in Y$. Since $\rho$ takes values in $\bN_0$ it in fact suffices to show that $\rho(y) \neq 0$ for all $y \in Y$. Otherwise we can find some $y \in Y$ with $Z'_y \cap D^{-1} = \emptyset$. 
Since $Z'_y = \Lambda_H Z_y$ this would imply that 
\[
\emptyset = \Lambda'_H Z_y \cap D^{-1} \iff \emptyset = 
Z_y \cap (D\Lambda'_H)^{-1} = Z_y \cap H.
\]
On the other hand, since $Y = H.Z$ we have $y = h.z$ for some $h \in H$, $z \in Z$ and hence $h^{-1} \in Z_y \cap H$, contradiction.\\

\item \textsc{Claim 3:}  $\tau = \tau'$, hence $\tau$ is finite.

\item As explained in \S \ref{InducedMeasure}, it suffices to show that the corresponding transverse measure $\tau_Z$  and $\tau'_Z$ on $Z$ coincide. By construction we have $\tau_Z = \nu$ and $\tau'_{Z'} = \nu'$, and by 
Lemma \ref{DoubleRestriction} the latter implies that $\tau'_Z = \nu'|_Z$. We are thus left with showing that $\nu'|_Z = \nu$.

\item For this we fix an open symmetric identity neighbourhood $V_o \subset V$ such that $V_o^2 \subset V$. Then for any Borel set $B \subset Z \subset Z'$ the set $V_o \times B$ is injective, hence by \eqref{PalmLocal} we have
\[
m_G(V_o) \nu'(B) = \mu(V_o.B) =  m_G(V_o) \nu(B).
\]
Since $m_G(V_o) > 0$, the claim follows.
\end{proof}  
This finishes the proof of Theorem \ref{IntegrabilityIntro}.
\subsection{Higher integrability via monotone couplings}
\begin{no} It is often of great interest to establish that the image of the Siegel transform consists of functions which are not just integrable, but even square-integrable. Beyond the cocompact case, this is difficult problem, since square-integrability can fail even in the lattice case. For some of our examples below, square-integrability problem can be reduced to the lattice case using the following trivial observation.
\end{no}
\begin{proposition}\label{JoinInt} If $(X_1, Y_1,\mu_1)$ admits a monotone coupling to $(X_2, \mu_2, Y_2)$ and $(X_2, \mu_2, Y_2)$ is $p$-integrable for some $p \in [1, \infty]$, then so is $(X_1, Y_1,\mu_1)$.
\end{proposition}
\begin{proof} If $X := X_1 \times X_2$ and $\mu \in \mathrm{Prob}(X)$ is a monotone coupling, then for every compact subset $L \subset H \backslash G$ we have
\begin{align*}& \int_{X_1} |L \cap (Y_1)_{x_1}|^p \dd \mu_1(x_1) = \int_{X} |L \cap(Y_1)_{x_1}|^p \dd {\mu}(x_1, x_2)\\ \leq&  \int_{X} |L \cap (Y_2)_{x_2}|^p \dd {\mu}(x_1, x_2) =  \int_{X_2} |L \cap (Y_2)_{x_2}|^p \dd \mu_2(x_2) < \infty.\qedhere
\end{align*}
\end{proof}

\section{Classical examples}\label{SecClassical}

\subsection{Lattices}

In this subsection we briefly explain how the case of lattices, as outlined in \S \ref{LatticeCase} and \S \ref{LatticeCharacters} in the introduction, fits into our general framework.

\begin{no} If $\Gamma< G$ is a lattice, then $X(\Gamma) := \Gamma \backslash G$ admits a unique $G$-invariant probability measure $\mu$ and $Z(\Gamma) := \{\Gamma e\} \subset X(\Gamma)$ is a $G$-cross section with return time set $\Gamma$, hence separated. The induced $(G,H)$-cross section is $Y(\Gamma) := \Gamma_H \backslash Y$, where $\Gamma_H = \Gamma \cap H$. If $\Gamma$ is $H$-compatible, then $Y(\Gamma)$ is integrable by Theorem \ref{IntegrabilityIntro} (applied with $\Lambda := \Gamma$), and if $\sigma_0$ denotes the unique $H$-invariant probability measure on $Y(\Gamma)$, then the $H$-transverse measure is given by $\sigma = \mathrm{covol}(\Gamma)^{-1} \cdot \sigma_0$. In particular, the Siegel constant is given by $\sigma(Y) = \mathrm{covol}(\Gamma)^{-1}$.
\end{no}

\begin{no} We observe that $Z(\Gamma)_{\Gamma a} = \{g \in G \mid \Gamma a g^{-1} = \Gamma\} = \Gamma a$. It thus follows from \eqref{YxZx} that $Y(\Gamma)_{\Gamma a} = \pi(\Gamma a)$. This in turn implies that the Siegel-Radon transform associated with the system $(X(\Gamma), Y(\Gamma), \mu)$ is given by 
\[
 S f(\Gamma a) = \sum_{Hg \in \pi(\Gamma a)} f(Hg).
\]
As mentioned in the introduction, this is precisely the Helgason-Radon transform associated with the homogeneous double fibration $H \backslash G \leftarrow  {\Gamma_H}\backslash G \rightarrow \Gamma \backslash G$, which identifies with the intersection space.
\end{no}

\begin{no} To describe eigencharacters of $Y(\Gamma)$ observe that if $\xi \in \widehat{H}$ vanishes on $\Gamma_H$, then it descends to a function $\psi_\xi$ on $Y(\Gamma)$. For all $y = \Gamma_H a \in Y(\Gamma)$ and $h \in H$ we then have
\[
\psi_\xi(h^{-1}.y) = \xi(ah) = \xi(h)\xi(a) = \xi(h)\psi_\xi(y),
\]
whence $\xi$ is an eigencharacter of $Y(\Gamma)$ with eigenfunction $\psi_\xi$.
\end{no}

\begin{remark} For uniform lattices $\Gamma$ which intersect $H$ cocompactly, Theorem \ref{IntegrabilityIntro} implies that the associated Siegel transform takes values in $L^\infty(X(\Gamma), \mu)$, whereas in the non-uniform case we can only guarantee that it takes values in $L^1(X, \mu)$. While it does indeed take values in $L^2(X(\Gamma), \mu)$ in many classical examples, there are examples of pairs $(H, \Gamma)$ (with $\Gamma_H < H$ a non-uniform lattice) for which the (untwisted) Siegel-Radon transform does not take values in $L^2(X, \mu)$ (cf. \cite[Appendix 2]{DRS}). 
\end{remark}

\subsection{Siegel-Veech transform}\label{SecSiegelVeech}

\begin{no} If we consider Example \ref{SiegelIntro} for $n=2$ and identify $\R^2 \cong \C$, then every $\Lambda \in \cL_2$ corresponds to a pair $(\C/\Lambda, dz)$ of a Riemann surface (in fact, a torus) with a nowhere vanishing holomorphic $1$-form, normalized to volume $1$. More generally, given $\alpha = (\alpha_1, \dots, \alpha_n)^\top \in \bN_0^n$ we can consider pairs $(\Sigma, \omega)$, where $\Sigma$ is a Riemann surface (necessarily of genus $g = \frac 1 2 (\alpha_1 + \dots + \alpha_n+2)$) and $\omega$ is a holomorphic $1$-form on $\Sigma$ with $n$ zeros of orders $\alpha_1, \dots, \alpha_n$ respectively, subject to the volume normalization $\frac i 2 \int_{\Sigma} \omega \wedge \overline{\omega} = 1$. Such pairs are called \emph{(normalized) translation surfaces} of type $\alpha$; see \cite{AMasur} for references.

Every translation surface $(\Sigma, \omega)$ carries a canonical metric, which is flat outside the marked points and has a cone point singularity of angle $2\pi(\alpha_i + 1)$ at the $i$th marked point; a geodesic for this metric which connects two marked points is called a \emph{saddle connection}, and the collection of all such geodesics is denoted $\mathcal{SC}(\Sigma, \omega)$. Every saddle connection $\gamma$ defines a \emph{holonomy vector} $\mathrm{hol}(\gamma) = \int_\gamma \omega \in \C$, and we denote the collection of these holonomy vectors by $\Lambda(\Sigma, \omega)$. 
\end{no}
\begin{no}
The moduli space of (normalized) translation surfaces of a fixed genus $g$ decomposes into strata $X(\alpha)$ consisting of (marked) translation surfaces of type $\alpha$, and these strata admit a (non-transitive) action of $G = \mathrm{SL}_2(\R)$ which preserves an invariant probability measure $\mu_{\alpha}$ called the \emph{Masur-Veech measure}. Moreover, the map 
\[
\Lambda: X(\alpha) \to \bC \cong \R^2, \quad (\Sigma, \omega)\mapsto\Lambda(\Sigma, \omega)
\]
is $G$-equivariant, and its image is contained in $\R^2 \setminus \{0\} \cong H \backslash G$ with $H$ as in Example \ref{SiegelIntro}. Since $G$ acts transitively on $\R^2 \setminus \{0\}$, the closed subset
\[
Y(\alpha)  := \{(\Sigma, \omega) \in X(\alpha) \mid (1,0)^\top \in \Lambda(\Sigma, \omega)\} \subset X(\alpha)
\]
is a $(G, H)$-cross section. Moreover, by equivariance of $\Lambda$, for every $(\Sigma, \omega) \in X(\alpha)$ we have
\[
Y(\alpha)_{(\Sigma, \omega)} = \{g^{-1}.e_1 \in G \mid e_1 \in \Lambda(g.(\Sigma, \omega))\} =  \{g^{-1}.e_1 \in G \mid g^{-1}.e_1 \in \Lambda(\Sigma, \omega)\} = \Lambda(\Sigma, \omega).
\]
Masur \cite{Masur} has proved that there exist constant $0 \leq c_1 < c_2 < \infty$ (depending on $\alpha$) such that for $\mu_{\alpha}$-almost all $(\Sigma, \omega) \in X(\alpha)$ we have
\begin{equation}\label{masur}
c_1 R^2 \leq |\Lambda(\Sigma, \omega) \cap B_R(0)| \leq c_2 R^2;
\end{equation}
this implies that the systemm $(X(\alpha), Y(\alpha), \mu_{\alpha})$ is integrable, and we denote by $\sigma_\alpha$ the $H$-transverse measure on $Y(\alpha)$. The associated Siegel-Radon transform 
\[
S:  C_c( \R^2 \setminus \{0\}) \to  L^1(X(\alpha), \mu_\alpha), \quad (\Sigma, \omega) \mapsto \sum_{x \in \Lambda(X, \omega)} f(x)
\]
is known as the \emph{Siegel-Veech transform}; in the case of tori without marked points we recover the classical Siegel transform, and it was shown recently \cite[Theorem 1.1]{ACM} that, as in the classical case, the Siegel-Veech transform takes values in $L^2(X(\alpha), \mu_\alpha)$.

\item The Siegel-Veech transform can be used to count holonomy vectors, i.e. to refine \eqref{masur}. For instance, this leads to a proof of the Eskin-Masur formula \cite[Theorem 2.1]{EM}
\[
 \lim_{R \to \infty} \frac{|\Lambda(\Sigma, \omega) \cap B_R(0)|}{\pi R^2} = \sigma_\alpha(Y(\alpha)) \text{ for $\mu_\alpha$-almost all }(\Sigma, \omega).
\]
where the Siegel constant $c(\alpha) := \nu_\alpha(Y(\alpha))$ is known as the \emph{Siegel-Veech constant} of the stratum. Studying the asymptotics of these constants for large genus is a major problem in the theory of translation surfaces. 
\end{no}
\begin{no} Already in the theory of Siegel-Veech transforms the flexibility in choosing a $(G,H)$-cross section$Y$ becomes useful. For instance, assume that $\alpha = (1,2,3)^\top$ and that we want to estimate the size of the set $\Lambda_{(1 \to 3)}(\Sigma, \omega)$ of holonomy vectors of saddle connections which connect a marked point of order $1$ to a marked point of order $3$. This amounts to studying the $\mathrm{SL}_2(\R)$-equivariant map $\Lambda_{(1 \to 3)}: X(\alpha) \to \R^2 \setminus \{0\}$. 

\item In general, an  $\mathrm{SL}_2(\R)$-equivariant map $\Lambda_*: X(\alpha) \to \R^2 \setminus\{0\}$ with $\Lambda_*(\Sigma, \omega) \subset \Lambda(\Sigma, \omega)$ is known as a \emph{configuration}. Every configuration gives rise to a corresponding $(G, H)$-cross section\[Y_*(\alpha) = \{(\Sigma, \omega) \in X(\alpha) \mid e_1 \in \Lambda_*(\Sigma, \omega)\} \subset Y(\alpha),\] and one can show (similarly to the proof of Lemma \ref{DoubleRestriction}) that the corresponding transverse measure is just $\nu_{\alpha}|_{Y_*(\alpha)}$. The associated Siegel transform is then given by \[S_*f(\Sigma, \omega) = \sum_{x \in \Lambda_*(\Sigma, \omega)} f(x),\] and
\[
 \lim_{R \to \infty} \frac{|\Lambda_*(\Sigma, \omega) \cap B_R(0)|}{\pi R^2} = \nu_\alpha(Y_*(\alpha)) \text{ for $\mu_\alpha$-almost all }(\Sigma, \omega).
\]
These generalized Siegel-Veech transforms have recently attracted a lot of attention in the study of the spectral theory of translation surfaces \cite{Lagace}.
\end{no}

\subsection{Marklof-Str\"ombergsson transform}\label{SecMS}
\begin{no} Let $m, n$ be positive integers and denote by $\R_m$ the space of \emph{row} vectors of length $m$. We consider the action of $G := \mathrm{SL}_m(\R)$ by \emph{right}-multiplication on $\R_m$ and denote by $H$ the stabilizer of $v_o := (1,0, \dots, 0)$. By smoothness of the action we obtain a homeomorphism 
\begin{equation}\label{Rmno0}
H\backslash G \to \R_m \setminus\{0\}, \quad Hg \mapsto v_o g.
\end{equation}
The group $G$ also acts on $\R_{m+n} \cong \R_m \times \R_n$ via $g.(x,y) = (xg^{-1},y)$ and hence on the space $\mathscr{L}_{m+n}$ of all unimodular lattices in $\bR_{m+n}$. We define a $G$-invariant Borel subset $\mathscr{L}'_{m+n} \subset \mathscr{L}_{m+n}$ by
\[
\mathscr{L}'_{m+n} = \{ \Delta \in \mathscr{L}_{m+n} \, : \, \textrm{
$\pr_{\bR_n}(\Delta) < \bR_n$ is dense} \}.
\]
We now fix  $\Delta_o \in \mathscr{L}'_{m+n}$ and denote by $X = X(\Delta_o)$ the orbit closure of $\Delta_o$ in $\cL'_{m+n}$. By \cite[Subsection 1.4]{MS}, there is a closed subgroup $L_{\Delta_o} < \SL_{m+n}(\bR)$ such that $X = L_{\Delta_o}.\Delta_o$ and the stabilizer of $\Delta_o$ in $L_{\Delta_o}$ is a lattice, hence $X$ carries a unique $L_{\Delta_o}$ invariant probability measure $\mu_{\Delta_o}$. We now choose a window $W \subset \R_n$, i.e. a bounded Borel subset with non-empty interior. We then define a $(G,H)$-cross section $Y =Y(W) \subset X(\Delta_o)$ by
\[
Y = \{ \Delta \in X(\Delta_o) \mid \Delta \cap (\{v_o\} \times W) \neq \emptyset \}.
\]
\end{no}
\begin{theorem}[Marklof-Str\"ombergsson, {\cite[Theorem 1.5]{MS}}] For every $\Delta_o \in \mathscr{L}'_{m+n}$ and every window $W \subset \R_n$
the $(G,H)$-transverse system $(X(\Delta_o), Y(W), \mu_{\Delta_o})$ is integrable.
\end{theorem}
One can show that the system is even square-integrable, see \cite[Theorem 1.2]{RSW}.

\begin{no} We now proceed to identify the Siegel-Radon transform associated with the system $(X(\Delta_o), Y(W), \mu_{\Delta_o})$. Given $\Delta \in X(\Delta_o)$ we define the \emph{associated cut-and-project set} $P_W(\Delta) \subset \R_m$ as
\[
P_W(\Delta) := \mathrm{pr}_{\R_m}(\Delta \cap (\R_m \times W))
\]
Since $\Delta$ projects densely to $\bR_n$ and $W$ is bounded and has non-empty interior in $\bR_n$, this is a Delone set (i.e.\ a relatively dense and uniformly discrete subset) in $\bR_m$. On the other hand,
by definition, for $\Delta \in X$ we have
\[
Y_\Delta = \{ v_o g \, : \, (\{v_og\} \times W) \cap \Delta \neq \emptyset \}.
\]
Under the identification \eqref{Rmno0} this corresponds precisely to the set $P_W(\Delta) \setminus\{0\}$, hence the Siegel transform of $(X(\Delta_o), Y(W), \mu_{\Delta_o})$ is given by
\[
S: C_c(\R_m \setminus\{0\}) \to L^1(X(\Delta_o),\mu_{\Delta_o}), \quad Sf(\Delta) := \sum_{v \in  P_W(\Delta)} f(v).
\]
This is precisely the Siegel transform for quasicrystals studied by Marklof and Str\"ombergsson in \cite[Theorem 1.5]{MS}.
\end{no}

\section{Induced Hull systems} \label{SecAO} 

\subsection{Measured sets and induced Hull systems}

\begin{no} The space $\cC(G)$ of closed subsets of $G$ carries a compact metrizable $G$-invariant topology known as the \emph{Chabauty-Fell topology} (see e.g.\ \cite{BHP1}). Given $P_o \in \cC(G)$ we denote by 
\[{X}^+(P_o) := \overline{\{P_og^{-1} \mid g \in G\}} \qand X(P_o) := X^+(P_o) \setminus \{\emptyset\}\] the orbit closure, respectively the \emph{hull} of $P_o$. We will be mostly interested in the case in which $P_o \subset G$ is (right-)uniformly discrete in the sense that $e_G$ is not an accumulation point of $P_oP_o^{-1}$. 
The following proposition summarizes basic properties of the hull; here we equip $C_c(G)$ with the inductive limit topology and its dual $C_c(G)^*$ with the corresponding weak-$*$-topology (cf.\ \cite{BHP1, BHP2}).
\end{no}
\begin{proposition} Let $P_o \subset G$ be a closed subset.
\begin{enumerate}[(i)]
\item $X(P_o)$ is an lcsc space and $G$ acts on $X(P_o)$ by homeomorphisms.
\item $X(P_o)$ is a homogeneous $G$-space if and only if $P_o = \Gamma g$ for some subgroup $\Gamma \subset G$ and some $g \in G$, in which case $X(P_o) = \Gamma \backslash G$.
\item $X(P_o)$ is compact if and only if $G = L P_o$ for some compact $L \subset G$. (i.e.\ $P_o$ is \emph{relatively dense}).
\item For all $P \in X(P_o)$ we have $PP^{-1} \subset {\overline{P_oP_o^{-1}}}$. 
\item If $P_o$ is uniformly discrete, then every $P \in X(P_o)$ is uniformly discrete and the embedding 
\[
\delta: X(P_o) \hookrightarrow C_c(G)^*, \quad P \mapsto \delta_P := \sum_{x \in P} \delta_x
\]
is a homeomorphism onto its image.
\end{enumerate}
\end{proposition}
\begin{no} Assume that $P_o \in \cC(G)$ is non-empty and let $Z(P_o) := \{Q \in X(P_o) \mid e \in Q\}$. By definition of the Chabauty topology, $Z(P_o)$ is compact, hence Borel and $p.P = Pp^{-1} \in Z$ for all $p \in P \in X(P_o)$. This shows that $Z(P_o)$ is a cross section for $X(P_o)$, called the \emph{canonical cross section}. Here, the word canonical refers to the fact that
\begin{equation}\label{CanonicalCross}
Z(P_o)_P = P \quad \text{for all } P \in X(P_o).
\end{equation}
\end{no}
\begin{proposition}\label{HullBasics} The canonical cross section $Z(P_o)$ has the following properties:
\begin{enumerate}[(i)]
\item $P_oP_o^{-1} \subset \Lambda(Z(P_o)) \subset \overline{P_oP_o^{-1}}$. In particular, if $P_oP_o^{-1}$ is discrete, then $ \Lambda(Z(P_o)) = P_oP_o^{-1}$.
\item $Z(P_o)$ is separated if $P_o$ is uniformly discrete. 
\item $Z(P_o)$ is cocompact if $P_o$ is there exists a compact subset $L \subset G$ such that $G = LP_o$.
\end{enumerate}
\end{proposition}
\begin{proof} (i) We may assume that $e \in P_o$, hence $P_oP_o^{-1} \subset \Lambda(Z(P_o))$. If $Q \in Z(P_o)$, then 
\[
Z(P_o)_Q = Q = Qe^{-1} \subset QQ^{-1} \subset \overline{P_oP_o^{-1}} \implies P_oP_o^{-1} \subset \Lambda(Z(P_o)) \subset \overline{P_oP_o^{-1}}.
\]
\item (ii) follows from (i).
\item (iii) If $L \subset G$ is compact with $G = LP_o$, then $P_0 \in \{P \in X(P_o) \mid LP=G\}$, and the latter is closed and $G$-invariant, hence coincides with $X(P_o)$. For every $P \in X(P_o)$ we thus find $p \in P$ with $e \in Lp$ and hence $p \in P \cap L^{-1}$. This implies $p.P \in Z(P_o)$ and hence $P \in p^{-1}.Z(P_o)$, which shows that $X(P_o) = L.Z(P_o)$.
\end{proof}
\begin{definition} We say that a uniformly discrete subset $P_o \subset G$ is \emph{measured} if $\mathrm{Prob}(X({P_o}))^G \neq \emptyset$.
\end{definition}
\begin{corollary}\label{HullSystem} If $P_o \subset G$ is a measured uniformly discrete subset and $\mu_{P_o} \in \mathrm{Prob}(X({P_o}))^G$, then $(X({P_o}), Z({P_o}), \mu_{P_o})$ is a separated $G$-transverse system. \qed
\end{corollary}
\begin{no}  In the situation of Corollary \ref{HullSystem} we define
\[
Y(P_o) := \{P \in X(P_o) \mid P \cap H \neq \emptyset\} = H.Z(P_o)
\]
so that $(X(P_o), Y(P_o), \mu_{P_o})$ is the induced $(G,H)$-transverse system of $(X({P_o}), Z({P_o}), \mu_{P_o})$. We refer to such a system as an \emph{induced hull system}.
\end{no}
\begin{remark}
Every relatively dense and uniformly discrete subset $P_o$ of an amenable group is measured since in this case $X(P_o)$ is compact and hence admits a $G$-invariant probability measure. We will see in Theorem \ref{Densitymeasure} below that it actually suffices to assume positive density of $P_o$, but this requires an argument. 

\item If $G$ is non-amenable (or $P_o$ is not relatively dense) then deciding whether a uniformly discrete subset $P_o \subset G$ is measured is more difficult. For example, unlike lattices, approximate lattices themselves need not be measured. See \cite{Hru, BH3} for examples of uniform approximate lattices in $\mathrm{GL}_2(\R)$ which are not measured.
\end{remark}
\begin{no}\label{Universality} Let $(X,Z, \mu)$ be a separated $G$-transverse system with induced $(G,H)$-system $(X,Y, \mu)$. As explained in \cite[Sec.\ 3.2]{BHK}, the map $\mathrm{can}: X \to \cC(G)$, $x \mapsto Z_x$ is Borel and its image consists of uniformly discrete subset of $G$. Now assume that $\mu$ ist ergodic and set $\mu' := \mathrm{can}_*\mu$;  then $\mu'(\emptyset) = 0$ for almost every $x$ the set $P_o = Z_x$ has a dense orbit in $\mathrm{supp}(\mu)$, hence we can consider as a probability measure on $X(P_o)$. The map $\mathrm{can}$ thus induces (up to nullsets) factor maps
\[
(X,Z, \mu) \to (X(P_o), Z(P_o), \mu') \qand (X,Y, \mu) \to (X(P_o), Y(P_o), \mu')
\]
which preserves $p$-integrability for all $p \in [1, \infty]$. To summarize, almost every hitting time set of a separated transverse $G$-system is measured, every $(G,H)$-transverse system which is induced from a separated system factors over an induced hull system.
\end{no}

\subsection{Integrability of induced hull systems}
From Theorem \ref{IntegrabilityIntro} we can now deduce the following (cf.\ Theorem \ref{HullIntIntro} from the introduction).
\begin{theorem}\label{HullInt} Let $(H, \Lambda)$ be a compatible pair and $P_o \subset \Lambda$ be a measured subset of $G$. Then for every $\mu_{P_o} \in \mathrm{Prob}(X({P_o}))^G$ the induced system $(X(P_o), Y(P_o), \mu_{P_o})$ is integrable with Siegel-Radon transform
\[
S: C_c(H\backslash G) \to L^1(X(P_o), \mu_{P_o}), \quad Sf(P) = \sum_{Hg \in \pi(P)} f(Hg).
\]
If $(H, \Lambda)$ is uniform, then $S$ takes values in $L^\infty(X(P_o), \mu_{P_o})$.
\end{theorem}
\begin{proof} Since $P_o \subset \Lambda$ we have $P_oP_o^{-1} \subset \Lambda\Lambda^{-1} = \Lambda^2$, and the latter is again an approximate lattice, in particular uniformly discrete. We deduce with Proposition \ref{HullBasics}.(i) that $\Lambda(Z(P_o)) \subset \Lambda^2$. Since $(H, \Lambda)$ is compatible pair, so is $(H, \Lambda^2)$, and hence Theorem \ref{IntegrabilityIntro} yields the desired integrability. It then follows  from \eqref{CanonicalCross} and \eqref{YxZx} that $Y(P_o)_P = \pi(P)$ for all $P \in X(P_o)$, resulting in the explicit formula for $S$.
\end{proof}
Note that Theorem \ref{HullInt} applies in particular if $P_0 = \Lambda$ is itself a measured approximate lattice which intersects $H$ in an approximate lattice.

\subsection{Cut-and-project sets as measured approximate lattices}
We now explain how cut-and-project systems give rise to measured approximate lattices; see \cite{BHP1, BHAL} and in particular \cite[Sec.\ 3.2]{BHK} for background. 
\begin{construction}
Given a cut-and-project system $\cS = (G_1, G_2, \Gamma)$ in the sense of Example \ref{CuPIntro} we denote by $\mu_{\cS}$ the unique $(G_1 \times G_2)$-invariant probability measure on the space
 \[X(\cS) := \Gamma \backslash (G_1 \times G_2).\] 
 Every window $W \subset G_2$ then defines a separated $G_1$ cross section $Z(W) := \Gamma \backslash (\{e_{G_1}\} \times W) \subset X(\cS)$, and we denote by
 \[
Y(W) := H_1.X(W) = \Gamma \backslash (H_1 \otimes W)
 \]
 the induced $(G,H)$-cross section. The key observation is now that the return time sets of $Z(W)$ are precisely the cut-and-project sets arising from $(\cS, W)$ (cf. \cite{HKW,SW}). We thus deduce from \S \ref{Universality} that for almost all choices of $g_1$ and $g_2$ the cut-and-project set $P_o = P_o(\cS, W, g_1, g_2)$ is measured, and that we can choose $\mu_{P_o} \in \mathrm{Prob}(X(P_o))^G$ to obtain almost everywhere defined canonical factor maps
 \begin{equation}\label{Canonicalmodel}
  (X(\cS), Z(W), \mu_{\cS}) \to (X(P_o), Z(P_o), \mu_{P_o})   \qand (X(\cS), Y(W), \mu_{\cS}) \to  (X(P_o), Y(P_o), \mu_{P_o}).
 \end{equation}
 It will usually be more convenient to work with $(X(\cS), Y(W), \mu_{\cS})$ than with $ (X(P_o), Y(P_o), \mu_{P_o})$ directly. In general, there might be other invariant probability measure on $X(P_o)$, but if $P_o$ happens to be regular in the sense of \cite{BHP1}, then $\mu_{\cS}$ is actually the \emph{unique} $G$-invariant probability measure on $X(P_o)$ and the canonical factor map is an isomorphism up to nullsets (see \cite[Thm.\ 1.1]{BHP1} and \cite{Sch} for the original proof in the abelian case).
 
\end{construction}

\begin{no}\label{UniqueStationarity} For later use, let us point out that the measure $\mu_{\cS}$ satisfies a very strong uniqueness property. To define this, let  $p \in \mathrm{Prob}(G_1)$ be a regular probability measure, i.e.\ $p$ is symmetric, absolutely continuous with respect to Haar measure and its support generates $G_1$. A probability measure $\mu$ on a $G_1$-space is then called \emph{$p$-stationary} if $p \ast \mu = \mu$. By \cite[Lemma 3.7]{BHP1} the measure $\mu_{\cS}$ is the unique $p$-stationary probability measure for any regular probability measure $p$ on $G_1$.
\end{no}

\begin{remark} 
Following \cite{Hru} we say that an approximate lattice $\Lambda$ is \emph{laminar} if it is a relatively dense subset in a cut-and-project set. All approximate lattice in abelian \cite{Meyer1}, amenable \cite{MachadoAmenable} and semisimple \cite{Hru} groups are laminar, hence finite index subsets of \emph{measured} approximate lattices. However, there exist examples of approximate lattices in $\mathrm{GL}_2(\R)$ which are not commensurable to any measured approximate lattices \cite{Hru}.
\end{remark}
\subsection{$H$-transverse measures for convenient cut-and-project sets}
Given a (generic) cut-and-project set $P_o$ we would like to determine the $H$-transverse measure $\sigma$ of $(X(P_o), Y(P_o), \mu_{P_o})$, or equivalently, of the system $(X(\cS), Y(W), \mu_{\cS})$ (cf. \eqref{Canonicalmodel}). In order to obtain a nice formula, we will make some simplifying assumptions which are satisfied in many examples of interest.
\begin{no}\label{ConvenientCuP} Throughout this subsection, $G=G_1$ and $G_2$ are unimodular lcsc group, $H=H_1$ and $H_2$ are closed unimodular subgroups of $G_1$ and $G_2$ respectively. For $j \in \{1,2\}$ we pick compatible Haar measures $m_{G_j}$, $m_{H_j}$ and $m_{H_j \backslash G_j}$ and Borel sections $s_j$ of the canonical projections $\pi_j: G_j \to G_j/H_j$. We assume that $\Gamma < G_1 \times G_2$ is a lattice with the following properties:
\begin{enumerate}[(L1)]
\item $\Delta := \Gamma \cap (H_1 \times H_2) < H_1 \times H_2$ is a co-compact lattice. 
\item $\pr_{G_2}(\Gamma)$ and $\pr_{H_2}(\Delta)$ are dense subgroups of $G_2$ and $H_2$ respectively.
\item If $(\gamma_1,\gamma_2) \in \Gamma$ and $\gamma_1 \in H_1$, then $\gamma_2 \in H_2$. Furthermore, $\Gamma \cap (\{e\} \times G_2) = \{(e,e)\}$.
\end{enumerate}
 Then $\cS := (G_1, G_2, \Gamma)$ is a cut-and-project scheme and we choose a bounded Jordan-measurable window $W \subset G_2$ with non-empty interior. We are going to abbreviate
\[
X := X(\cS), \quad Z := Z(W) \qand Y := Y(W) = \Gamma \backslash (H_1 \times W),  
\]
and denote by $\mu = \mu_\cS$ the unique $(G_1 \times G_2)$-invariant probability measure on $X$. We are going to describe the $H_1$-transverse measure $\sigma$ of $(X, Y, \mu)$.
\end{no}
\begin{example}[Arithmetic cut-and-project sets]\label{ArithmeticExamples} To underscore that our present assumptions are fairly mild, let us provide a generic example satisfying these assumptions: For this consider the Galois-involution on $\cO = \Z[\sqrt 2]$ given by $(a+ b\sqrt 2)^* := a-b\sqrt 2$ for $a,b \in \Z$; then the discrete embedding $\Z[\sqrt 2] \hookrightarrow \R \times \R$, $\alpha \mapsto (\alpha, \alpha^*)$ induces a discrete embedding $\mathrm{SL}_2(\Z[\sqrt 2]) \hookrightarrow \mathrm{SL}_2(\R) \times \mathrm{SL}_2(\R)$. We now choose $G_1 := G_2 := \mathrm{SL}_2(\R)$ and $\Gamma := \mathrm{SL}_2(\Z[\sqrt 2])$. Then our assumptions are satisfied if we choose e.g. 
\[
H_1 := H_2 := \left\{\begin{pmatrix}  1& \ast\\ 0&1 \end{pmatrix}\right\} \subset \mathrm{GL}_2(\R).
\]
Replacing $\Q[\sqrt 2]$ by other number fields, $\Z[\sqrt 2]$ by other orders and $\mathrm{SL}_2$ and the group of upper unipotent $2 \times 2$-matrices by other suitable algebraic groups one can produce lots of examples by the same scheme.
\end{example}
\begin{no} Denote by $Y_o \subset X$ the closed $(H_1 \times H_2)$-orbit of $\Gamma$, which we may identify with the quotient $\Delta \backslash (H_1 \times H_2)$. Then $Y_o$ carries a unique $(H_1 \times H_2)$-invariant probability measure $\sigma_o$. We then denote by $c= c(H_1, H_2, \Delta)$ the covolume of $\Delta$ in $H_1 \times H_2$ so that for all $F \in \mathscr{L}^\infty_c(H_1 \times H_2)$ we have
\begin{equation}\label{DefCovol}
(m_{H_1} \otimes m_{H_2})(F) = c(H_1,H_2,\Delta)^{-1} \cdot \int_{Y_o} \Big( \sum_{\delta \in \Delta} F(\delta_1 h_1,\delta_2 h_2) \Big) \dd\sigma_o(\Delta(h_1,h_2)).
\end{equation}
We now fix once and for all a compactly supported symmetric continuous function $\rho_o: H_1 \to [0,1]$ with normalization $m_{H_1}(\rho_o) = 1$. By Remark \ref{UniqueStationarity}, the measure $\sigma_o$ is the unique $\rho_o$-stationary probability measure on $Y_o$, and since $Y_o$ is compact we have weak-$*$-convergence
\begin{equation}\label{ConvolutionLimit}
\frac{1}{n} \sum_{k=1}^n (\rho_o^{*k} \otimes \delta_e) * \delta_\Gamma \xrightarrow{n \to \infty} \sigma_o.
\end{equation}
 \end{no}
\begin{construction}\label{ConEta} Since $\sigma_o$ is $H_1 \times H_2$-invariant, we have a well-defined map \[
H_2 \backslash G_2 \ra M_{\textrm{fin}}(X)^{H_1}, \enskip 
H_2g_2 \mapsto (e,g_2)^{-1}_*\sigma_o,
\]
and hence we can define $\eta \in M_{\textrm{fin}}(X)^{H_1}$ by
\[
\eta(f) = \frac{1}{c(H_1,H_2,\Delta)} \int_{\pi_2(W)} (e,g_2)^{-1}_*\sigma_o(f)  \dd m_{H_2 \backslash G_2}(H_2g_2), \quad \textrm{for $f \in C_c(X)$}.
\]
\end{construction}
\begin{theorem}\label{TransverseExplicit} Let $\sigma \in M_{\mathrm{fin}}(Y)^{H_1}$ denote the $H_1$-transverse measure of $\mu$. 
 \begin{enumerate}[(i)]
\item If $\eta \in M_{\textrm{fin}}(X)^{H_1}$ is as in Construction \ref{ConEta}, then $\eta(X \setminus Y) = 0$.
\item $\sigma = \eta|_Y$, hence in particular the Siegel constant is given by $\sigma(Y) = \frac{m_{H_2 \backslash G_2}(\pi_2(W))}{c(H_1,H_2,\Delta)}$.
\end{enumerate}
\end{theorem}
The proof of Theorem \ref{TransverseExplicit} is surprisingly involved and will occupy the remainder of this subsection. We start with some preliminary observations.
\begin{no} Let $\nu$ denote the transverse measure of $(X, Z, \mu)$; one can check that
\begin{equation}\label{DefNu}
\nu(f) = \int_W f(\Gamma(e,w)) \dd m_{G_2}(w) \quad (f \in C_c(X))).
\end{equation}
We now consider the $H$-induced measure $\tau$ of $\nu$ on $Y$. We have seen in Theorem \ref{IntegrableInduced} that $\tau$ is actually finite and that $\sigma = \tau$. In the proof of Theorem \ref{TransverseExplicit} we will be almost exclusively working with $\tau$ and show that $\eta|_Y = \tau$.
\end{no}
\begin{no}
To show that $\eta|_{Y}=\tau$ we first need to construct Borel functions with some peculiar properties. 
Define $\widetilde{u}_W : H_2 \times H_2 \backslash G_2 \ra [0,\infty)$ by
\[
\widetilde{u}_W(h_2,H_2g_2)
= 
\left\{ 
\begin{array}{cl}
\frac{\chi_{H_2 \cap Ws_2(H_2g_2)^{-1}}(h_2)}{m_{H_2}(H_2 \cap Wg_2^{-1})} & \textrm{if $m_{H_2}(H_2 \cap Wg_2^{-1}) > 0$} \\[0.2cm]
0 & \textrm{otherwise}
\end{array}
\right..
\]
From this we define auxiliary functions
$u_W: G_2 \to [0, \infty)$ and $ \psi_W: Z \to [0, \infty)$ by
\[
 \psi_W(\Gamma(e, g_2)) := u_w(g_2) :=  \widetilde{u}_W(g_2s_2(H_2g_2)^{-1}, H_2g_2),
\]
where we have used the second assumption in (L3) for $\psi_W$ to be well-defined. We also set 
\[
G_2' := \{ g_2 \in G_2 \, : \, m_{H_2}(H_2 \cap Wg_2^{-1}) = m_{H_2}(H_2\cap W^o g_2^{-1}) \} \qand W' := W \cap G_2'.
\]
\end{no}
\begin{lemma}\label{WFunctions}
The set $W'$ and the functions $\widetilde{u}_W$ and $u_W$ have the following properties:
\begin{enumerate}[(i)]
\item $m_{G_2}(W \setminus W') = 0$, $m_{H_2 \backslash G_2}(\pi_2(W) \setminus \pi_2(W')) = 0$ and $\tau(\Gamma(H_1 \times W'))= \tau(\Gamma(H_1 \times W))$.
\item For every $g_ 2 \in W'$, the map $H_2 \ni h_2 \mapsto \widetilde{u}_W(h_2,H_2g_2)$ is a bounded $m_{H_2}$-Jordan measurable function, with support contained in the bounded  set $H_2 \cap WW^{-1}$, and 
\[
 \int_{H_2} \widetilde{u}_W(h_2,H_2 g_2) \, dm_{H_2}(h_2) = 1.
\]
\item $\{g_2 \in G_2 \mid u_W(g_2) \neq 0\} \subset W$.
\item For  $m_{H_2 \backslash G_2}$-almost every $H_2g_2 \in H_2 \backslash G_2$ we have
\[\int_{H_2} u_W(h_2 g_2) \, dm_{H_2}(h_2) = \chi_{\pi_2(W)}(H_2 g_2), \quad \text{and hence} \quad \int_{G_2} u_W(g_2) \, dm_{G_2}(g_2) = m_{H_2 \backslash G_2}(\pi_2(W)).\]
\end{enumerate}
\end{lemma}
Postponing the proof of the lemma for the moment, we now proceed as follows:
\begin{proposition}\label{ExplicitSigmaMain} For all $f \in C_b(X)$ we have
\[
c(H_1, H_2, \Delta) \cdot \eta(f) = \lim_{n \to \infty} \int_Y \beta_n(y) f(y) \dd \tau(y),
\]
where
\[
\beta_n(y) := \frac 1 n \sum_{k=1}^n \left( \sum_{h_1 \in Z_y \cap H_1} \rho_o^{*k}(h_1^{-1}) \psi_W(h_1.y) \right).
\]
\end{proposition}
\begin{proof} Using $H_2$-invariance of $\sigma_o$ and Part (iv) of Lemma \ref{WFunctions} we obtain
\begin{align*}
c(H_1, H_2, \Delta) \cdot \eta(f) & =  \int_{H_2g_2}\chi_{\pi_2(W)}(H_2g_2) (e,g_2)^{-1}_*\sigma_o(f)  \dd m_{H_2 \backslash G_2}(H_2g_2)\\
&= \int_{H_2g_2} \int_{H_2} u_W(h_2g_2)  ) (e,h_2g_2)^{-1}_*\sigma_o(f) \dd m_{H_2}(h_2)  \dd m_{H_2 \backslash G_2}(H_2g_2)\\
&= \int_{G_2} u_W(g_2) (e,g_2)^{-1}_*\sigma_o(f) \dd m_{G_2}(g_2)\\
&= \int_W \int_{Y_o} f(\Gamma(h_1,h_2g_2)) \, u_W(g_2) \dd\sigma_o(\Delta(h_1,h_2)) \dd m_{G_2}(g_2).
\end{align*}
Using \eqref{ConvolutionLimit}, the definition of $\psi_W$ and \eqref{DefNu} we can rewrite this as
\begin{align*}
c(H_1, H_2, \Delta) \cdot \eta(f) & = \lim_{n \to \infty} \sum_{k=1}^n \int_W \int_{Y_o} f(\Gamma(h_1, g_2)) \rho^{\ast k}(h_1)u_W(g_2) \dd m_{H_1}(h_1)\dd m_{G_2}(g_2)\\
&=  \lim_{n \to \infty} \sum_{k=1}^n \int_{H_1}\left( \int_W  f(\Gamma(h_1, g_2)) \psi_W(\Gamma(e, g_2)) \dd m_{G_2}(g_2)\right) \rho^{\ast k}(h_1) \dd m_{H_1}(h_1)\\
&=   \lim_{n \to \infty} \sum_{k=1}^n \int_{H_1} \int_Z f(h_1^{-1}.z) \psi_W(z)  \rho^{\ast k}(h_1) \dd \nu(z) \dd m_{H_1}(h_1).
\end{align*}
Since $\tau$ is the induced measure of $\nu$ we can rewrite the double integral inside the sum as
\[
 \int_{H_1} \int_Z f(h_1^{-1}.z) \psi_W(z)  \rho^{\ast k}(h_1) \dd \nu(z) \dd m_{H_1}(h_1) = \int_Y \left( \sum_{h_1 \in Z_y \cap H_1} \rho_o^{*k}(h_1^{-1}) \psi_W(h_1.y) \right) f(y) \dd \tau(y).\qedhere
 \]
\end{proof}
The key observation is now:
\begin{lemma}\label{LimitsExample} We have $\lim_{n \to \infty} \beta_n = c(H_1, H_2, \Delta)$ both $\tau$-almost everywhere and in $L^1(Y, \tau)$.
\end{lemma}
Once this is proved, the theorem follows immediately:
\begin{proof}[Proof of Theorem \ref{TransverseExplicit}] Since $\sigma = \tau$ and since $\sigma$ is supported on $Y$, both (i) and (ii) follows if we can show that $\eta(f) = \tau(f)$ for all $f \in C_b(X)$. For every such function $f$ we have
\[
\Big| \int_Y (\beta_n(y) - c(H_1,H_2,\Delta)) \cdot f(y)  \dd\tau(y) \Big|
\leq \|\beta_n - c(H_1,H_2,\Delta)\|_{L^1(Y,\tau)} \cdot \|f\|_\infty \ra 0, \quad \textrm{as $n \ra \infty$}.
\]
We deduce that
\[
\int_Y \beta_n(y) f(y) \dd \tau(y) = \int_Y (\beta_n(y) - c(H_1,H_2,\Delta)) \cdot f(y) \dd\tau(y) + c(H_1,H_2,\Delta) \cdot \tau(f) \rightarrow  c(H_1,H_2,\Delta) \cdot \tau(f),
\]
and hence $\eta(f) = \tau(f)$ by Proposition \ref{ExplicitSigmaMain}.
\end{proof}
We have thus reduced the proof of Theorem \ref{TransverseExplicit} to those of Lemma \ref{WFunctions} and Lemma \ref{LimitsExample}.
\begin{proof}[Proof of Lemma \ref{WFunctions}] Towards the proofs of (i) and (ii) we first claim that $G_2' \subset G_2$ is a left-$H_2$-invariant conull set. Indeed,  
since $W$ is $m_{G_2}$-Jordan measurable we have
\begin{align*}
m_{G_2}(W) &= \int_{H_2 \backslash G_2} m_{H_2}(H_2 \cap Wg_2^{-1}) \dd m_{H_2 \backslash G_2}(Hg_2)\\ &\geq \int_{H_2 \backslash G_2} m_{H_2}(H_2 \cap W^o g_2^{-1}) ) \dd m_{H_2 \backslash G_2}(Hg_2)
= m_{G_2}(W^o) = m_{G_2}(W), 
\end{align*}
hence the inequality must be an equality, and the claim follows. 

\item (i) The first two equalities are immediate from the claim, and they imply that $Z' = \Gamma(\{e\} \times W') \subset Z$ is conull and hence
\[
\nu((H_1.Z \setminus H_1.Z') \cap Z) \leq \nu(Z \cap (Z')^c) = 0.
\]
Since $\tau$ is induced from $\nu$, \cite[Lemma 4.6]{BHK} then yields $\tau(H_1.Z \setminus H_1.Z') = 0$. Since  $\Gamma(H_1 \times W') = H_1.Z'$ and $\Gamma(H_1 \times W) = H_1.Z$ this proves (i).

\item (ii) By definition of $W'$, for every $g_2 \in W'$, the set $H_2 \cap Wg_2^{-1}$ is $m_{H_2}$-Jordan measurable. Furthermore, we see that if $g_2 \in W' \subset W^o$, then $H_2 \cap W_2s_2(H_2g_2)^{-1}$ contains a non-empty open set in $H_2$, which forces $m_{H_2}(H_2 \cap W_2g_2^{-1}) > 0$, so in particular the map $H_2 \ni h_2 \mapsto \widetilde{u}_W(h_2,H_2g_2)$ is non-trivial and $m_{H_2}$-Jordan measurable for every $g_2 \in W'$, and as a function on $H_2$ it is bounded by the finite constant $m_{H_2}(H_2 \cap Ws(H_2g_2)^{-1})^{-1}$ on $H_2$. The claims concerning the support and the integral are immediate from the construction.

\item (iii) If $u_W(g_2) \neq 0$, then $g_2 s_2(H_2 g_2)^{-1} \in H \cap Ws_2(H_2 g_2)^{-1}$, and thus $g_2 \in H_2g_2 \cap W \subset W$. 

\item (iv)  If $g_2 \in W'$, then by (ii) we have
\[
\int_{H_2} u_W(h_2 g_2) \dd m_{H_2} = \int_{H_2} \widetilde{u}_W(h_2 g_2 s_2(H_2g_2)^{-1},H_2g_2) \dd m_{H_2}(h_2) = 
\int_{H_2} \widetilde{u}_W(h_2 ,H_2g_2) \dd m_{H_2}(h_2) = 1.
\]
If $H_2 g_2 \notin \pi(W)$, then $m_{H_2}(H_2 \cap Wg_2^{-1}) = 0$,
so $\widetilde{u}_W(h_2,H_2g_2) = 0$ for all $h_2 \in H_2$, and thus 
\[
\int_{H_2} u_W(h_2 g_2) \dd m_{H_2} = 0.
\]
Since $\pi(W) \setminus \pi(W')$ is a $m_{H_2 \backslash G_2}$-null set by (i), we are done.
\end{proof}
\begin{no}
We now turn to the proof of Lemma \ref{LimitsExample}. We will need two different formulas for the functions $\beta_n$. To state these, we denotes elements of $Y$ by $y = \Gamma (y_1, y_2)$ and define a function
\[
x:Y \to Y_o, \quad x(y) :=  \Gamma (y_1,y_2s_2(H_2y_2)^{-1})
\]
We then introduce functions $\Phi : Y \ra [0,\infty)$ and $\varphi_{H_2y_2} : H_1 \times H_2 \ra [0,\infty)$ by
\[
\Phi(y) = \sum_{h_1 \in Z_y \cap H_1} \rho_o(h_1) \psi_W(h_1.y) \qand \varphi_{H_2y_2}(h_1,h_2) = \rho_o(h_1)\widetilde{u}_W(h_2,H_2y_2).
\]
Given a function $\Phi$ on $H_1 \times H_2$ we denote by $\mathrm{Per}_\Delta(\Phi)$ its periodization over $\Delta$.
\end{no}
\begin{lemma}\label{betans} For every $n \in \bN$ the following hold:
\begin{enumerate}[(i)]
\item $\beta_n(y) = \frac 1 n \sum_{k=1}^n \int_{H_1} \rho_{o}^{*(k-1)}(t) \Phi(t.y) \dd m_{H_1}(t) $.
\item $\Phi \in L^1(Y,\tau)$ and $\int_Y \Phi \dd \tau = m_{H_2 \backslash G_2}(\pi_2(W))$.
\item  $\beta_n(y) = \frac 1 n \sum_{k=1}^n \int_{H_1} \rho_{o}^{*(k-1)}(t) \Per_\Delta(\varphi_{H_2y_2})(t.x(y)) \dd m_{H_1}(t)$.
\item For every $y_2 \in W'$, the function 
$\Per_{\Delta}(\varphi_{H_2y_2})$ is bounded and $\tau_o$-Jordan measurable on $Y_o$.
\end{enumerate}
\end{lemma}
\begin{proof} (i) Since $\rho_o$ is symmetric we can write out the convolution powers as 
\begin{align*}
\rho_o^{*k}(h_1^{-1}) 
&= 
\int_{H_1} \rho_o^{*(k-1)}(t) \rho_o(t^{-1}h_1^{-1}) \, dm_{H_1}(t) = 
\int_{H_1} \rho_o^{*(k-1)}(t) \rho_o(h_1 t) \, dm_{H_1}(t) \\[0.2cm]
&=
\int_{H_1} \rho_o^{*(k-1)}(t) \rho_o(h_1 t^{-1}) \, dm_{H_1}(t).
\end{align*}
Using only invariance of $m_{H_1}$ and the definition of $\Psi$ we obtain
\begin{align*}
\sum_{h_1 \in Z_y \cap H_1} \rho_o^{*k}(h_1^{-1}) \psi_W(h_1.y)
&= 
\int_{H_1} \rho_o^{*(k-1)}(t) \Big( \sum_{h_1 \in Z_y \cap H_1} 
\rho_o(h_1 t^{-1}) \psi_W(h_1.y) \Big) \, dm_{H_1}(t) \\[0.2cm]
&= 
\int_{H_1} \rho_o^{*(k-1)}(t) \Big( \sum_{h_1t^{-1} \in Z_{t.y} \cap H_1} 
\rho_o(h_1 t^{-1}) \psi_W(h_1t^{-1}.(t.y)) \Big) \, dm_{H_1}(t) \\[0.2cm]
&=
\int_{H_1} \rho_o^{*(k-1)}(t) \Big( \sum_{h_1 \in Z_{t.y} \cap H_1} 
\rho_o(h_1) \psi_W(h_1.(t.y)) \Big) \, dm_{H_1}(t) \\[0.2cm]
&= 
\int_{H_1} \rho_o^{*(k-1)}(t) \Phi(t.y) \, dm_{H_1}(t),
\end{align*} 
and averaging over $k$ yields the desired formula.

\item (ii) Using the definitions of $\nu$ and $\psi_W$ and the fact that $\tau$ is induced from $\nu$ we can write out
\[
\int_{Y} \Phi \dd\tau = \int_{H_1} \int_Z \rho_o(h_1) \, \psi_W(z) \dd\nu(z) 
\dd m_{H_1}(h_1) = \int_W u_W(w) \dd m_{G_2}(w) = m_{H_2 \backslash G_2}(\pi_2(W)),
\]
where we have used the normalization of $\rho_o$ and Lemma \ref{WFunctions}.(iv).

\item (iii) Let $y = \Gamma(y_1,y_2)$ with $y_1 \in H_1$ and $y_2 \in W$ and define $h_2(y_2) = y_2 s_2(H_2y_2)^{-1} \in H_2$. For all $h_1 \in H_1$ we have the equivalences
\begin{align*}
h_1 \in Z_y \cap H_1 &\iff \Gamma(y_1h_1^{-1},y_2) \in Z \iff \exists \, (\gamma_1,\gamma_2) \in \Gamma: \; (\gamma_1y_1 = h_1) \wedge (\gamma_2 y_2 \in W).
\end{align*}
In fact, $\gamma_1$ is uniquely determined by $h_1$, and by the second property of (L3) also $\gamma_2$ is determined.
Moreover $\gamma_1 = h_1y_1^{-1} \in H_1$, so by (L3) we must have $\gamma_2 \in H_2$, and thus the pair $(\gamma_1,\gamma_2)$ in the equivalence above must belong to $\Delta$. In summary we obtain a bijection
\[Z_y \cap H_1 \to \Delta, \; h_1 \mapsto (\delta_1, \delta_2) \quad \text{such that}
\quad  (\delta_1 y_1 = h_1) \wedge (\delta_2 h_2(y_2) \in H_2 \cap Ws_2(H_2y_2)^{-1}).
\]
For all $t \in H_1$ we thus obtain
\begin{align*}
\rho_o(h_1 t^{-1}) \psi_W(h_1.y)
&= 
\rho_o(\delta_1 y_1 t^{-1}) \psi_W(\Gamma(y_1h_1^{-1},y_2)) \\[0.2cm]
&=
\rho_o(\delta_1 y_1 t^{-1}) \psi_W(\Gamma(\delta_1y_1h_1^{-1},\delta_2y_2)) \\[0.2cm]
&=
\rho_o(\delta_1 y_1 t^{-1}) \psi_W(\Gamma(e,\delta_2y_2)) = 
\rho_o(\delta_1 y_1 t^{-1}) u_W(\delta_2y_2),
\end{align*}
since $\delta_2 y_2 \in W$ and thus $\psi_W(\Gamma(e,\delta_2y_2)) = u_W(\delta_2 y_2)$.  We can therefore write
\begin{align*}
\sum_{h_1 \in Z_y \cap H_1} \rho_o^{*k}(h_1^{-1}) \psi_W(h_1.y)
&=
\int_{H_1} \rho_o^{*{k-1}}(t) \Big( \sum_{(\delta_1,\delta_2) \in \Delta} 
\rho_o(\delta_1 y_1 t^{-1}) u_W(\delta_2 y_2) \Big) \dd m_{H_1}(t) \\[0.2cm]
&=
\int_{H_1} 
\rho_o^{*{k-1}}(t) 
\Per_{\Delta}(\rho_o \otimes \widetilde{u}_W(\cdot,H_2y_2))(\Delta(y_1 t^{-1},y_2s(H_2 y_2)^{-1}) \Big) \dd m_{H_1}(t) \\[0.2cm]
&=
\int_{H_1} \rho_o^{*(k-1)}(t) \Per_{\Delta}(\varphi_{H_2y_2})((t,e).x(y)) \dd m_{H_1}(t),
\end{align*}
\item (iv) It readily follows from Lemma \ref{WFunctions}.(iv) that for every $y_2 \in W'$, the function $\varphi_{H_2 y_2}$
is a bounded $m_{H_1} \otimes m_{H_2}$-Jordan measurable function on $H_1 \times H_2$
with bounded support, and thus $\Per_{\Delta}(\varphi_{H_2y_2})$ is a bounded and $\sigma_o$-Jordan measurable function on $Y_o$.
\end{proof}
\begin{proof}[Proof of Lemma \ref{LimitsExample}] We first prove that $\beta_n(y) \to c(H_1, H_2, \Delta)$ for almost all $y = \Gamma(y_1, y_2) \in Y$. 
By Lemma \ref{WFunctions}.(i) we may assume that $y_2 \in W'$, which by Lemma \ref{betans}.(iv) then implies that
$\Per_{\Delta}(\varphi_{H_2y_2})$ is a bounded and $\sigma_o$-Jordan measurable function on $Y_o$. We may thus apply \eqref{ConvolutionLimit} to deduce from Lemma \ref{betans}.(iii) that
\begin{align*}
\lim_{n \ra \infty} \beta_n(y) = &\lim_{n \ra \infty}
\frac{1}{n} \sum_{k=0}^{n-1}
\int_{H_1} \rho_o^{*k}(t) \Per_{\Delta}(\varphi_{H_2y_2})((t,e).x(y)) \, dm_{H_1}(t) \\[0.2cm]
&= \int_{Y_o} \Per_{\Delta}(\varphi_{H_2y_2}) \, d\sigma_o = c(H_1,H_2,\Delta) \cdot \int_{H_1} \int_{H_2} \varphi_{H_2y_2} \, dm_{H_2} dm_{H_1} \\[0.2cm]
&= c(H_1,H_2,\Delta) \cdot m_{H_1}(\rho_o) \cdot \int_{H_2} \widetilde{u}_W(h_2,H_2y_2) \, dm_{H_2}(h_2) = c(H_1,H_2,\Delta),
\end{align*}
where we have used \eqref{DefCovol}, the normalization of $\rho_o$ and Lemma \ref{WFunctions}.(ii) (which applies since $y_2 \in W'$).

For the proof of $L^1$-convergence we consider the probability measure $p = \rho_o \cdot m_H$ on $H$ and form the shift space $(\Omega, \bP) := (H^{\bN_0}, p^{\bN_0})$. We denote by $S(\omega_n) := (\omega_{n+1})$ the shift and define 
\[
\widetilde{S}: \Omega \times Y \to \Omega \times Y, \; \widetilde{S}(\omega, y) := S(\omega, \omega_0.y) \quad \text{so that} \quad \widetilde{S}^k(\omega, y) = (S^k\omega, \omega_{k-1} \cdots \omega_0. y).
\]
Note that since $\tau$ is $\rho_o$-stationary, the measure $\bP \otimes \tau$ is $\widetilde{S}$-invariant. If we define
\[
\widetilde{\Phi}(\omega, y) := \Phi(y) \qand \widetilde{\beta}_n(\omega, y) := \frac 1 n \sum_{k=0}^{n-1}\widetilde{\Phi}(\widetilde{S}^k(\omega, y)) ,
\]
then $\widetilde{\Phi} \in L^1(\Omega \times Y, \bP \otimes \tau)$ and Lemma \ref{betans}.(i) yields
\[
\int_{\Omega} \widetilde{\beta}_n(\omega, y) \dd \bP(\omega) =  \frac 1 n \sum_{k=0}^{n-1} \int_H \rho^{\ast k}(h) \Phi(h.y) \dd m_H(h) = \beta_n(y).
\]
By the Birkhoff ergodic theorem the sequence $\widetilde{\beta_n}$ converges in $L^1(\Omega \times Y, \bP \otimes \tau)$ and almost surely to some $\widetilde{S}$-invariant limit $\widetilde{\beta}_\infty$. Since $\tau$ need not be ergodic, we cannot say at this point whether this limit is almost everywhere constant, but we will deal with this problem later. If we now define
\[
\beta_\infty(y) := \int \widetilde{\beta}_\infty(\omega, y) \dd \bP(\omega), 
\]
then by Fubini we have $\beta_\infty \in L^1(Y, \tau)$ and moreover
\begin{align*}
\int_Y |\beta_n(y) - \beta_\infty(y)| \dd \tau(y) & = \int_Y \left| \int_\Omega \widetilde{\beta}_n(\omega, y) - \widetilde{\beta}_\infty(\omega, y) \dd \bP(\omega) \right| \dd \tau(y)\\
&\leq \int_Y\int_{\Omega} |\widetilde{\beta}_n - \widetilde{\beta}_\infty| \dd \bP \otimes \tau = \|\widetilde{\beta}_n - \widetilde{\beta}_\infty\|_{L^1(\Omega \times Y, \bP \otimes \tau)} \to 0.
\end{align*}
This shows that the sequence $(\beta_n)$ converges in $L^1$ and almost everywhere, possibly to a non-constant limit. However, since we already know that the almost everywhere pointwise limt equals
$c(H_1, H_2, \Delta)$, hence the $L^1$-limit is given by the same constant.
\end{proof}
This concludes the proof of Theorem \ref{TransverseExplicit} and our discussion of hulls of regular cut-and-project set.

\subsection{Measured subsets in the amenable case}\label{SecDensity}

In order to apply Theorem \ref{HullInt} beyond the case of cut-and-project sets we need to find more examples of measured subsets of approximate lattices. In this subsection we show that in \emph{amenable} groups there is a rich supply of such subsets since every subset of positive density is measured; this is basically a weak version of Furstenberg's correspondence principle \cite[Theorem 1.1]{Fu} for non-discrete groups.
\begin{no}\label{DensityDn}
Throughout this subsection let $G$ be an amenable lcsc group. Since $G$ is $\sigma$-compact. By \cite[Theorem 4 and Corollary 6]{ER}, we can thus find a 
sequence $(D_n)$ of Borel sets in $G$ of positive Haar measure such that
\begin{enumerate}[(i)]
\item each $D_n$ is pre-compact. 
\item for every compact set $C \subset G$, 
\begin{equation}\label{DnSpecialFolner}
\lim_{n \ra \infty} \frac{m_G(CD_n)}{m_G(D_n)} = 1.
\end{equation}
\end{enumerate}
In what follows, we fix such a sequence $(D_n)$ in $G$. We then say that a locally finite subset $P_o \subset G$ has \emph{positive upper density} with respect to $(D_n)$ if
\[
\varlimsup_{n \ra \infty} \frac{|P_o \cap D_n|}{m_G(D_n)} > 0.
\]
\end{no}
\begin{theorem}\label{Densitymeasure}
If $P_o \subset G$ is uniformly discrete and has positive upper density with respect to $(D_n)$, then it is measured.
\end{theorem}

\begin{proof}
Recall that $X^+(P_o)$ denotes the (compact) orbit closure of $P_o$ in $\cC(G)$ so that $X(P_o) \subset X^+(P_o) \subset X(P_o) \cup \{\emptyset\}$. Since $P_o$ is uniformly discrete, we can find an identity neighbourhood in $G$ such that that $P_o P_o^{-1} \cap U = \{e\}$. Note that if $P \in X^+(P_o)$ and $p,q \in P \cap C$, then $pq^{-1} \in P^{-1}P \cap C^2 \subset P_o P_o^{-1} \cap U = \{e\}$, hence $|P \cap C| \leq 1$. We claim that the function
\[
S\chi_C: X^+(P_o) \to \{0,1\}, \quad S \chi_C(P) :=  \sum_{p \in P} \chi_C(p) = |P \cap C|
\]
is upper semicontinuous. Indeed, if $(P_n)$ is a sequence in $X^+(P_o)$ converging to $P \in X^+(P_o)$; we need to show that 
\begin{equation}\label{USC}
\varlimsup_{n \ra \infty} S\chi_C(P_n) \leq S\chi_C(P).
\end{equation}
If the left-hand side is $0$, then this is automatic. Assume otherwise, so that the left-hand side equals $1$. By definition of Chabauty convergence there is then a sub-sequence $(n_k)$ and elements $p_{n_k} \in P_{n_k} \cap C$ for 
all $k$. Since $C$ is compact, we may assume (by passing to a further subsequence) that $p_{n_k}$ converges to some $p \in C$.  By definition of Chabauty convergence we then have $p \in P \cap C$, and thus $S\chi_C(P) = 1$, hence \eqref{USC} holds.

\item Now define $\widetilde{D}_n = CD_n$ and set 
\[
\mu_n = \frac{1}{m_G(\widetilde{D}_n)} \int_{\widetilde{D}_n} \delta_{g.P_o} \, dm_G(g) \in \mathrm{Prob}(X^+(P_o)).
\]
Since $X^+({P_o})$ is compact and metrizable, this sequence admits an accumulation point $\mu$. It follows from \eqref{DnSpecialFolner} that
\[
\lim_{n \ra \infty} \frac{m_G(h\widetilde{D}_n \Delta \widetilde{D}_n)}{m_G(\widetilde{D}_n)} = 0 \quad \textrm{for all $h \in G$},
\]
hence $\mu$ is $G$-invariant. It remains to show that $\mu \neq \delta_{\emptyset}$, since then at least one ergodic component of $\mu$ gives full mass to $X({P_o}) \subset X^+({P_o})$. Note that
\begin{align*}
\mu_n(S\chi_C)
& = \frac{1}{m_G(\widetilde{D}_n)} \int_{\widetilde{D}_n} \chi_C(P_og^{-1}) \dd m_G(g)
= 
\frac{1}{m_G(\widetilde{D}_n)} \sum_{p \in P_o}
\int_{\widetilde{D}_n} \chi_C(pg^{-1}) \dd m_G(g) \\[0.2cm]
&= \frac{1}{m_G(\widetilde{D}_n)} \sum_{p \in P_o}
m_G(Cp \cap \widetilde{D}_n) = \frac{1}{m_G(\widetilde{D}_n)} \sum_{p \in P_o}
m_G(Cp \cap C D_n) \\[0.2cm]
&\geq 
\frac{m_G(D_n)}{m_G(\widetilde{D}_n)} \cdot \frac{1}{m_G(D_n)} \sum_{p \in P_o \cap D_n}
m_G(C) = m_G(C) \cdot \frac{m_G(D_n)}{m_G(\widetilde{D}_n)} \cdot \frac{|P_o \cap D_n|}{m_G(D_n)},
\end{align*}
where we have used the inequality $Cp \subset CD_n$ for all $p \in P \cap D_n$.
Since, by \eqref{DnSpecialFolner},
\[
\lim_{n \ra \infty} \frac{m_G(D_n)}{m_G(\widetilde{D}_n)} = 1,
\]
we conclude that 
\[
\varlimsup_{n \ra \infty} \mu_n(S\chi_C) \geq m_G(C) \cdot \varlimsup_{n \ra \infty} 
\frac{|P_o \cap D_n|}{m_G(D_n)} > 0.
\]
Since $S \chi_C$ is a bounded upper semi-continuous function on $X^+(P_o)$, the Portmanteau theorem now implies
\[
\mu(S \chi_C) \geq \varlimsup_{n \ra \infty} \mu_n(S\chi_C) > 0.
\]
Since $\delta_{\emptyset}(S \chi_C) = S\chi_C(\emptyset) = 0$, we deduce that $\mu \neq \delta_{\emptyset}$.
\end{proof}

\subsection{Thinnings of measured approximate lattices}\label{SecCupThinning}

\begin{no} In non-amenable groups, finding examples of measured subsets of approximate lattices beyond cut-and-project sets is more complicated. This is e.g.\ witnessed by the fact that
even relatively dense subsets of lattices need not be measured, so the existence of invariant measures is no longer just a question of ``density''. Nevertheless, also in the non-amenable case there exist plenty of examples, as can be shown using thinnings.
\end{no}
\begin{construction}\label{ThinningCUP}
Let $\cS := (G_1, G_2, \Gamma)$ be a cut-and-project scheme, $H_1 <G_1$ a closed unimodular subgroup and $W \subset G_2$ be a window. We want to construct an equivariant thinning of the hull system of a corresponding generic cut-and-project set, which we may identify with $(X(\cS), Z(W), \mu_\cS)$ by \eqref{Canonicalmodel}, where $X(\cS) = \Gamma \backslash (G_1 \times G_2)$.

\item Denote by $\Gamma_j$ the projection of $\Gamma$ to $G_j$. Then, by definition of a cut-and-project scheme, $\Gamma$ is the graph of an isomorphism $\iota: \Gamma_1 \to \Gamma_2$, i.e.
\[
\Gamma = \{ (\gamma,\iota(\gamma)) \, : \, \gamma \in \Gamma_1 \} < G_1 \times G_2 
\]
Let $(T,\tau)$ be a mixing probability-measure preserving $\Gamma_1$-space and let $B \subset T$ be a Borel set with $\tau(B)>0$. For example, $T$ could be the Bernoulli shift $(\{0,1\}^{\Gamma_1}, (\frac 1 2 (\delta_0 + \delta_1))^{\otimes \Gamma_1})$ and $B$ could be a cylinder set. We define a $G_1 \times \Gamma_1$-action on $G_1 \times G_2 \times T$ by
\[
(g_1, \gamma).(x_1, x_2, t) = (\gamma x_1 g_1^{-1}, \iota(\gamma) x_2,\gamma.t),
\]
which then induces a $G_1$-action on the quotient $\widehat{X} := \Gamma \backslash (G_1 \times G_2 \times T)$. We then define subsets $Z \subset Y \subset X \subset \widehat{X}$ by
\[
Z :=  \{ \Gamma(e,w,b) \mid w \in W, \enskip b \in B \}, \quad Y := H_1.Z \qand X := G_1.Z,
\]
and an invariant measure $\mu$ on $X$ by
\[
\mu(f) = \int_{X(\cS)} \Big( \int_{T} f(\Gamma(x_1,x_2,t)) \dd\tau(t) \Big) \dd \mu_{\cS}(\Gamma(x_1,x_2)).
\]
\end{construction}
\begin{proposition} Let $Z \subset Y \subset X$ and $\mu$ be as in Construction \ref{ThinningCUP}.
\begin{enumerate}[(i)]
\item $\mu(X) = 1$ and $(X, Z, \mu)$ is a separated transverse $G_1$-system.
\item $(X, Z, \mu)$ is a thinning of $(X(\cS), Z(W), \mu_\cS)$, and hence $(X, Y, \mu)$ is a thinning of $(X(\cS), Y(W), \mu_\cS)$.
\end{enumerate}
\end{proposition}
\begin{proof} We first observe that we have a $G_1$-equivariant Borel map $p: X \to X(\cS)$ given by $p(\Gamma(x_1, x_2, t)) := \Gamma(x_1, x_2)$. Using this map, we now show that $\mu(X) = 1$.

\item Given $x_2 \in G_2$ we define the \emph{infinite} set 
\[
W(x_2) :=  \{ \gamma \in \Gamma_1 \mid \iota(\gamma) \in x_2 W^{-1} \}.
\]
For all $\gamma \in \Gamma_1$ and $x_2 \in G_2$ we then have $W(\iota(\gamma)x_2) = \gamma W(x_2)$, hence the map $\tau_W: X_{\cS} \to [0,1]$, $\Gamma(x_1,x_2) \mapsto \tau(W(x_2).B)$ is well-defined. Since $W(x_2)$ is infinite and the $\Gamma_1$-actions on $(Y, \tau)$ is mixing we actually have $\tau_W \equiv 1$. Note that
\[
p^{-1}(\Gamma(x_1,x_2)) = 
\{ \Gamma(x_1, x_2,\gamma.b) \mid \gamma \in W(x_2), \enskip b \in B \},
\]
and thus unravelling the definition of $\mu$ and using $\tau_W \equiv 1$ we get
\[
\mu(X_1) = \int_{X(\cS)} \tau(W(x_2).B) \, d\mu_{\cS}(\Gamma(g,g')) = \int_{X({\cS})} \tau_W d\mu_{\cS} = 1.
\]
Since $Z \subset p^{-1}(Z(W))$ and $p_*\mu = \mu_{\cS}$ we can now apply
Lemma \ref{GeneralThinning} to conclude.
\end{proof}

\subsection{Higher integrability}\label{HigherIntegrability}

\begin{no} Theorem \ref{HullInt} provides integrability for a large class of induced hull systems. However, in many applications one is interested in stronger integrability properties, in particular, square-integrability, and these do not hold in general in the setting of the theorem (as the case of lattices demonstrates). Stronger conclusions can be made if the compatible pair in question is \emph{uniform}. Indeed, directly from Proposition \ref{CocompactCase} we obtain:
\end{no}
\begin{proposition}\label{L2Unif} Let $(H, \Lambda)$ be a uniform compatible pair and let $P_o \subset \Lambda$ be a measured subset of $G$. Then for every compact $L \subset H\backslash G$ there is $C_L>0$ such that $|\pi(P) \cap L| < C_L$. In particular, $(X(P_o), Y(P_o), \mu)$ is $p$-integrable for all $\mu \in \mathrm{Prob}(X(P_o))^G$ and all $p$.\qed
\end{proposition}
In the non-uniform case we need other methods to prove square-integrability.
\begin{no} Let $G_1, G_2, H_1, H_2, \Gamma, W$ be as in \S \ref{ConvenientCuP}. The cut-and-project system $\cS :=(G_1, G_2, \Gamma)$ then gives rise to a $(G_1, H_1)$-system $(X(\cS), Y(W), \mu_{\cS})$. 
On the other hand, if we
 abbreviate $G := G_1 \times G_2$, $H:= H_1 \times H_2$, $\Gamma_H := \Gamma \cap H$, $X := \Gamma \backslash G$ and $Y := \Gamma_H \backslash H$ then we obtain a (periodic) $(G,H)$-system $(X, Y, \mu_{\cS})$ with $X(\cS) = X$ and $Y(W) \subset Y$. We then have
 \[
  H_1g_1 Y(W)_x \implies \exists w \in W: H(g_1, w) \in Y_x.
 \]
 If $L_1 \subset G_1$ is compact and $\pi_1: G_1 \to H_1 \backslash G_1$ and $\pi: G \to H \backslash G$ denote the projections, we thus obtain
 \[
 \int_{X(\cS)} |Y(W)_x \cap \pi_1(L_1)|^p \dd \mu_{\cS}(x) \leq \int_X |Y_x \cap \pi(L_1 \times W)|^p \dd \mu_{\cS}(x)
 \]
 for all $p \geq 1$. Combining this with Proposition \ref{JoinInt} we obtain the following result.
\end{no}
\begin{proposition}[Transfer principle] If the (periodic) $(G,H)$-transverse system $(X, Y, \mu_{\cS})$ is $p$-integrable for some $p>1$, then so are the (aperiodic) $(G_1, H_1)$-system $(X(\cS), Y(W), \mu_{\cS})$ and all of its thinnings, as well as the corresponding hull systems.\qed
\end{proposition}
\begin{example} From \cite[Theorem 1.3]{Kim} one can deduce square-integrability of periodic systems $(X,Y, \mu)$ arising from arithmetic lattices in products similar to Example \ref{ArithmeticExamples}, and hence for the corresponding cut-and-project systems (equivalently, the induced hull systems of the corresponding generic cut-and-project sets) and their thinnings. \end{example}
Once one has square-integrability for the hull of some regular model set, one can also get square-integrability for hulls of subsets by exploiting the fact that Proposition \ref{JoinInt} does not require the coupling to be invariant:
\begin{proposition}\label{ModelSetJoining} Let $\Lambda$ be a regular cut-and-project set with invariant measure $\mu_{\Lambda}$ and let $P_o \subset \Lambda$ be measured with invariant measure $\mu$. If $(X(\Lambda), Y(\Lambda), \mu_{\Lambda})$ is square-integrable, then so is $(X(P_o), Y(P_o), \mu)$.
\end{proposition}
The proposition is an immediate consequence of unique stationarity of hulls of regular cut-and-project sets and the following general principle:
\begin{lemma}\label{JoiningApp} Let $P_o \subset \Lambda$ be measured subsets of a lcsc group $G$. Assume the following:
\begin{enumerate}[(i)]
\item $X(\Lambda)$ is uniquely $p$-stationary with invariant measure $\mu_{\Lambda}$ for some regular $p \in \mathrm{Prob}(G)$.
\item $(X(\Lambda), Y(\Lambda), \mu_{\Lambda})$ is $q$-integrable for some $q \in [1,\infty]$.
\end{enumerate}
 Then $(X(P_o), Y(P_o), \mu)$ is $q$-integrable for every $\mu \in \mathrm{Prob}(X(P_o))^G$.\end{lemma}
\begin{proof} By Proposition \ref{JoinInt} it suffices to construct a monotone coupling between the systems $(X(P_o), Y(P_o), \mu)$ and $(X(\Lambda), Y(\Lambda), \mu_\Lambda)$. As before we denote by $X^+(P_o)$ and $X^+(\Lambda)$ the orbit closures of $P_o$ and $\Lambda$ in $\cC(G)$. Recall that either $X^+(P_o) = X(P_o)$ or $X^+(P_o) = X(P_o) \sqcup \{\emptyset\}$.

\item Consider now the orbit closure $X^+(P_o, \Lambda)$ of the pair $(P_o, \Lambda)$ in $\cC(G) \times \cC(G)$ and denote the restricted coordinate projections by $\pi_1: X^+(P_o,  \Lambda) \to X^+(P_o)$ and $\pi_2: X^+(P_o,  \Lambda) \to X^+(G)$. We note that $\pi_1$ and $\pi_2$ are surjective. 

\item The subset  $\{(P, P') \in \cC(G) \times \cC(G) \mid P \subset P'\} \subset \cC(G) \times \cC(G) \subset \cC(G) \times \cC(G)$ is closed, $G$-invariant and contains $(P_o, \Lambda)$, hence its orbit closure, i.e.\
\begin{equation}\label{monotonejoining}
P \subset P' \quad \text{for all }(P, P') \in  X^+(P_o, \Lambda)
\end{equation}
Now let $X(P_o, \Lambda) := \{(Q, Q') \in X^+(P_o, \Lambda)) \mid Q \cap H \neq \emptyset\} = \pi_1^{-1}(X(P_o))$. Then \[Y(P_o, \Lambda) := \{(Q, Q') \in X^+(P_o, \Lambda)) \mid Q \cap H \neq \emptyset\}\] is a $(G, H)$-cross section of $X(P_o, \Lambda)$ and $\pi_2(X(P_o, \Lambda)) \subset X(\Lambda)$ by \eqref{monotonejoining}.

\item Since $X^+(P_o, \Lambda)$ is compact, it supports a $p$-stationary probability measure $\widetilde{\mu}$ such that $(\pi_1)_*\widetilde{\mu} = \mu$. The latter implies that $\widetilde{\mu}(X(P_o, \Lambda))= 1$, and hence
$(\pi_2)_*\widetilde{\mu}(X(\Lambda)) = 1$, hence $(\pi_2)_*\widetilde{\mu} = \mu_\Lambda$ by unique stationarity. Thus $\widetilde{\mu}$ is a coupling of $\mu$ and $\mu_{\Lambda}$, which is monotone by \eqref{monotonejoining}.
\end{proof}

\section{Unitary intertwiners}\label{SecZak}

\subsection{An  abstract $L^2$-extension theorem}
If $(X,Y, \mu)$ is a square-integrable transverse $(G,H)$-system and $\xi$ is an eigencharacter of $Y$ with eigenfunction $\psi_\xi$, then it is natural to ask whether some scaled version of the twisted Siegel-Radon transform $S_{\psi_\xi}$ extends continuously to a unitary intertwiner $\smash{\mathrm{Ind}_H^G(\xi) \to L^2(X, \mu)}$. In this section we provide a positive answer to this question in a very specific setting, which is just broad enough to study the aperiodic generalization of the 
classical twisted Zak transform mentioned in the introduction.

\begin{no} In this subsection, $G$ is a group which can be written as an internal semidirect product 
\[
G = H \rtimes L \quad \text{with} \quad H\lhd G \; \text{ amenable}.
\]
We denote by $s: H \backslash G \to G$ the unique section of $\pi$ such that $s(H \backslash G) = L$ and set $\ell(g) := s(Hg)$. Then for all $g \in G$ we have $g = h(g)\ell(g)$ for some $h(g) \in H$, and obtain a bijection \[G \to H \times L, g \mapsto (h(g), \ell(g)).\]

\item Now assume that $(X,Y, \mu)$ is a square-integrable $(G,H)$-transverse system with $H$-transverse measure $\sigma$. It is convenient to identify $H\backslash G$ with $L$ via our section $s$, so that the hitting time sets get identified with subsets
\[
Y_x = \{l \in L \mid l.x \in Y\} \subset L \quad \text{for all } x\in X.
\]
We make these identifications and set  $\Lambda(Y) := \bigcup_{y \in Y}Y_y \subset L$. Note that for all $h \in H$ and $x \in X$,
\begin{equation}
\label{Yhx}
Y_{h.x} = \{ l \in L \, : \, lh.x \in Y \} = \{ l \in L \, : \, (lhl^{-1})l.x \in Y\}
= \{ l \in L \, : \, l.x \in Y \} = Y_x,
\end{equation}
where in the second to last identity we have used that $H\lhd G$ and that $Y$ is $H$-invariant.
\end{no}

\begin{no}
Since $L$ normalizes $H$ we obtain an action of $L$ on $\Hom(H, \bT)$ by
\[
(l.\eta)(h) = \eta(l^{-1}hl), \quad (l \in L, \eta \in \Hom(H,\bT), h \in H).
\]
Given $\xi \in \Hom(H, \bT)$ we denote by $\Stab_L(\xi)$ the stabilizer of $\xi$ under this action. If $\xi$  is an eigencharacter of $Y$ with eigenfunction $\psi_\xi$, then we denote by
\[
S_{\psi_\xi}^{\mathrm{norm}} := \sigma(Y)^{-1/2} \cdot S_{\psi_\xi}: {}^0\ind_H^G(\xi) \to L^2(X(\Lambda),  \mu).
\]
the \emph{normalized $\psi_\xi$-twisted Siegel-Radon transform}.
\end{no}

\begin{theorem}\label{ExtensionAbstract} If $\xi \in \Hom(H, \bT)$ is an eigencharacter of $Y$ with eigenfunction $\psi_\xi$ such that \[\Stab_L(\xi) \cap \Lambda(Y) = \{e_L\},\] then $\smash{S_{\psi_\xi}^{\mathrm{norm}} }$ extends continuously to a unitary intertwiner
\[
S_{\psi_\xi}^{\mathrm{norm}}: \ind_H^G(\xi) \to L^2(X,  \mu).
\]
\end{theorem}
\begin{proof} It is convenient to work with the second model of the induced representation as described in \S \ref{GHConventions}. Working in this model and identifying $H\backslash G$ with $L$ via $\ell$, the formula \eqref{SecondComing} for $S_{\psi_\xi}$ reads as
\[
S_{\psi_\xi}: C_c(L) \to L^2(X, \mu), \quad S_{\psi_\xi}f(x)  = \sum_{l \in Y_x} f(l) \psi_{\xi}(l.x),
\]
and we have to show that
\[
\int_X |S f(x)|^2 \dd\mu(x) = \sigma(Y) \cdot \int_{L} |f(l)|^2 \dd m_L(l).
\]
By definition,
\[
\int_X |Sf(x)|^2 \dd\mu(x) = \int_X \sum_{l_1, l_2 \in Y_x} f(l_1) \overline{f(l_2)} \psi_\xi(l_1.x) \psi_\xi(l_2.x) \dd\mu(x). 
\]
Since $\mu$ is $G$-invariant (and hence $H$-invariant) we see from \eqref{Yhx}
that for all $h \in H$ we have
\begin{align*}
\int_X |Sf(x)|^2 \dd\mu(x) 
&= \int_X \sum_{l_1, l_2 \in Y_{h.x}} f(l_1) \overline{f(l_2)} \psi_\xi(l_1h.x) \overline{\psi_\xi(l_2h.x)} \dd\mu(x) \\[0.2cm]
&= \int_X \sum_{l_1, l_2 \in Y_{x}} f(l_1) \overline{f(l_2)} \psi_\xi(l_1h.x) \overline{\psi_\xi(l_2h.x)} \dd\mu(x) \\[0.2cm]
&= \int_X \sum_{l_1, l_2 \in Y_{x}} f(l_1) \overline{f(l_2)} \psi_\xi(l_1.x) \overline{\psi_\xi(l_2.x)} \xi(l_1 h l_1^{-1}) \overline{\xi(l_2 h l_2^{-1})}\dd\mu(x).
\end{align*}
We now average the right-hand side over a F\o lner sequence $(B_n)$ in $H$: We abbreviate
\[
\eta_{l_1,l_2}(h) := \xi(l_1 h l_1^{-1}) \overline{\xi(l_2 h l_2^{-1})} \qand A_n(l_1,l_2) := \frac{1}{m_H(B_n)} \int_{B_n} \eta_{l_1,l_2}(h) \dd m_H(h).
\]
Note that if $l_1 = l_2$, then $\eta_{l_1,l_2} \equiv 1$ and hence $A_{l_1, l_2} = 1$. We deduce that for all $n \in \bN$,
\begin{align*}
\int_X |Sf(x)|^2 \dd\mu(x) 
&= 
\int_X \sum_{l_1, l_2 \in Y_{x}} f(l_1) \overline{f(l_2)} \psi_\xi(l_1.x) \overline{\psi_\xi(l_2.x)} A_n(l_1,l_2) \dd\mu(x)\\
&= \int_X \sum_{l \in Y_{x}} |f(l)|^2   \dd\mu(x) + \int_X \sum_{l_1 \neq l_2 \in Y_x} f(l_1) \overline{f(l_2)} \psi_\xi(l_1.x) \overline{\psi_\xi(l_2.x)} A_n(l_1,l_2) \dd\mu(x).
\end{align*}
Concerning the first summand we observe that by \eqref{SigmaDefEq},
\[
\int_X \sum_{l \in Y_{x}} |f(l)|^2   \dd\mu(x) = \sigma(Y) \cdot \int_{L} |f(l)|^2 \dd m_L(l).
\]
We are thus left with showing that
\begin{equation}\label{ExtThmToShow}
 \int_X F_n(x) \dd\mu(x) \xrightarrow{n \to \infty} 0, \quad \text{where} \quad F_n(x) := \sum_{l_1 \neq l_2 \in Y_x} f(l_1) \overline{f(l_2)} \psi_\xi(l_1.x) \overline{\psi_\xi(l_2.x)} A_n(l_1,l_2).
\end{equation}
We first observe that if $l_1 \neq l_2$, then $\eta_{l_1,l_2}$ is a non-trivial character on $H$, for otherwise we have $l_1^{-1}.\xi = l_2^{-1}.\xi_2$, or equivalently, $e \neq l_1 l_2^{-1} \in \Stab_L(\xi) \cap Y_x Y_x^{-1} \subset \Stab_L(\xi) \cap \Lambda(Y)$, contradicting our assumption. This implies that if $l_1 \neq l_2$, then $A_n(l_1,l_2) \ra 0$ as $n \ra \infty$. Since $f$ is bounded and has compact support in $L$ and $Y_x$ is locally finite, we deduce that
\begin{equation}
\label{Fnzero}
\lim_{n \ra \infty} F_n(x) = 0 \quad \textrm{for all $x \in X$}.
\end{equation}
Furthermore, since $|A_n(l_1,l_2)| \leq 1$ for all $l_1, l_2 \in Y_x$, we have
\[
|F_n(x)| \leq \sum_{l_1, l_2 \in Y_x} |f(l_1)||f(l_2)| = \Big( \sum_{l \in Y_x} |f(l)| \Big)^2 \leq \|f\|^2_\infty \cdot  |Y_x \cap \supp(f)|^2 \quad \textrm{for all $n \in \bN$}.
\]
Our square-integrability assumption now implies that there exists a constant $C$ depending only on $\mathrm{supp}(f)$ such that
\[
\int_X \sup_n |F_n(x)| \dd\mu(x) \leq C \cdot \|f\|^2_\infty < \infty,
\]
so the convergence in \eqref{Fnzero} is dominated, and we deduce that \eqref{ExtThmToShow} holds.
\end{proof}

\subsection{Construction of eigencharacters}\label{SecEigencharacters}
\begin{no}\label{HAmenInd} 
In order to apply Theorem \ref{ExtensionAbstract} we need a method to construct non-trivial eigencharacters. We will only discuss this in a very specific setting here, following \cite{BHNilpotent}. For this let $(X,Y, \mu)$ be a $(G, H)$-transverse system with $H$-transverse measure $\sigma$ induced by a separated $G$-transverse system $(X,Z, \mu)$ and assume that $H$ is \emph{amenable}. We make the (rather restrictive) assumptions that 
the $H$-return time set $\Delta := \Lambda(Z) \cap H \subset H$ is countable and that for some $\eps \in (0,1)$ the \emph{$\eps$-dual} of $\Delta$ defined as
\[ \Delta^{\eps} := \{\xi \in \widehat{H} \mid \forall\, \delta \in \Delta:\; |\xi(\delta)-1| \leq \eps \}\]
contains a non-trivial character. A notable case where these assumptions are satisfied is when $\Delta$ is an approximate lattice in an abelian group $H$ \cite[Chapter 2]{Meyer1}.
\end{no}
\begin{proposition}\label{EigenMeyer} In the situation of \S \ref{HAmenInd}, every $\xi \in \bigcup_{\eps \in (0,1)}\Delta^\eps$ is an eigencharacter of $Y$ with respect to $\sigma$.
\end{proposition}
For the proof we first construct approximate eigenfunctions:
\begin{lemma}\label{epsEigen} Let $\xi \in \Delta^{\eps}$ for some $\eps \in (0,1)$. Then there exists $\phi_\xi: Y \to \C$ with $|\phi_\xi(y)|= 1$ such that
\[
|\phi_\xi(h^{-1}.y) - \xi(h)\phi_\xi(y)| \leq \eps \quad (h \in H, y \in Y).
\]
\end{lemma}
\begin{proof} The fibers of the map 
\[
H \times Z \ra Y, \enskip (h,z) \mapsto h.z
\]
are countable, hence by \cite[Corollary 18.10]{K} we can find a Borel section $\varphi : Y \ra H \times Z$. Let $\mathrm{pr}_H: H \times Z \to H$ denote the projection and let $s := \pr_H \circ \varphi: Y \to H$. We define $\phi_\xi(y):= \overline{\xi}(s(y))$.  Since $|\xi| \equiv 1$ we have $|\phi_\xi| = 1$

If $y \in Y$ and $\varphi(y) = (h,z)$, then $s(y) = h$ and we have
\[
h.z = y \implies h^{-1}.y \in Z  \implies s(y)^{-1} = h^{-1} \in Z_y.
\]
For all $h \in H$ and $y \in Y$ we then have $\{s(y)^{-1}.y, s(h.y)^{-1}h.y\} \subset Z$. We deduce that
\begin{equation}
\label{slambda}
s(h.y)^{-1}hs(y).[s(y)^{-1}.y] = [s(h.y)^{-1}h.y] \implies \delta(h,y) := s(h.y)^{-1} h s(y) \in \Delta \quad (h \in H, y \in Y).
\end{equation}
We deduce that
\begin{align*}
|\phi_\xi(h^{-1}.y) - \xi(h)\phi_\xi(y)| &= |\overline{\xi}(s(h^{-1}.y))-\overline{\xi}(h^{-1})\overline{\xi}(s(y))| = |1-\delta(h^{-1}, y)| \leq \eps.\qedhere
\end{align*}
\end{proof}
Proposition \ref{EigenMeyer} then follows by an averaging argument as in \cite{BHNilpotent}:
\begin{proof}[Proof of Proposition \ref{EigenMeyer}] Consider the representation of $H$ on $L^p(Y, \sigma)$ given by $\pi(h)f(y) := \overline{\xi}(h) f(h^{-1}y)$. The function $\phi_\xi$ from Lemma \ref{epsEigen} satisfies $|\pi(h)(\phi_\xi)(y)-\phi_\xi(y)| < \eps$ for all $h \in H$, $y \in Y$ and we want to construct a non-zero fixpoint which is bounded. For this we fix a F\o lner sequence $(F_n)$ in $H$ and define a sequence of functions by
\[
\psi_\xi^{(n)}(y) := \frac{1}{m_H(F_n)} \int_{F_n} \pi(h)\phi_\xi(y) \dd m_H(h).
\]
Note that for every $y \in Y$ we have
\[
|\psi_\xi^{(n)}(y) - \phi_\xi(y)| \leq \frac{1}{m_H(F_n)} \int_{F_n} |\pi(h)(\phi_\xi)(y)-\phi_\xi(y)| \dd m_H(h) \leq \eps,
\]
hence $0 < 1-\eps \leq |\psi_\xi^{(n)}(y)| \leq 1+ \eps < 2$. Thus, by a suitable version of the pointwise ergodic theorem some subsequence of $(\psi_\xi^{(n)})_{n \in \bN}$ converges almost every and in $L^1$ to some $\pi(H)$-invariant Borel function.
 $\psi_\xi(y)$. By pointwise almost everywhere convergence we then have $0 < 1-\eps \leq |\psi_\xi(y)| \leq 2$ for almost all $y \in Y$, hence $\psi_\xi(y) \neq 0$ and $\psi_\xi$ is bounded. 
\end{proof}

\subsection{Aperiodic Zak transforms}\label{SecAperiodicZak}
We now apply Theorem \ref{ExtensionAbstract} to construct an aperiodic version of the (twisted) Zak transform, using the construction of the previous subsection. As in the introduction, we focus on the case of the Heisenberg group where the results take a particularly nice form. We emphasize though that Theorem \ref{ExtensionAbstract} applies beyond this case.
\begin{no} 
Concerning the Heisenberg group we use the notation of the introduction, i.e. $U=V= \R^n$, $Z = \R$ and $G = U \times Z \times V$ with group operations given by
\[
(u,t,v)(u',t',v') = (u+u', t+t' + \langle u,v' \rangle - \langle u', v \rangle, v+v') \qand (u,t,v)^{-1} = (-u,-t,-v).
\]
The subgroup $H = U \times Z$ is normal and abelian, and if we define an action of $H$ on $V$ by
\begin{equation}\label{HeisVaction}
(u,t).v = (u, t - 2\langle u, v \rangle), 
\end{equation}
then $G$ is the semidirect product $H \rtimes V$ with respect to this action, hence fits into the framework of Theorem \ref{ExtensionAbstract}.

\item If $\xi \in \widehat{Z}$, then we denote by $1 \otimes \xi$ the character of $H$ given by $(1 \otimes \xi)(u,t) = \xi(t)$. If $\xi \neq \mathbf{1}$, then $\pi_\xi = \mathrm{Ind}_H^G(1\otimes \xi)$ is irreducible by Mackey's little group theorem; it has central character $\xi$ and is known as the \emph{Schr\"odinger representation} of this central character.

\item As before before, given a subset $P_o \subset G$ we denote by $X(P_o)$ its hull and set
\[
Z(P_o) := \{Q \in X(P_o) \mid e \in Q\} \qand Y(P_o) := \{Q \in X(P_o) \mid Q \cap H \neq \emptyset\}.
\]
\end{no}
\begin{theorem}\label{Zakmain} Let $(H, \Lambda)$ be a compatible pair and let $P_o \subset \Lambda$ be measured. Assume that for some $\eps \in (0,1)$ the set $ \Lambda_{Z,H}^{\eps}$ contains a non-trivial character.
Then the following hold for every $\mu_{P_o} \in \mathrm{Prob}(X(P_o))^G$:
\begin{enumerate}[(i)]
\item $(X(P_o), Y(P_o), \mu_{P_o})$ is a square-integrable $(G,H)$-transverse system, hence gives rise to an $H$-transverse measure $\sigma_{P_o}$.
\item For every $\xi \in  \Lambda_{Z,H}^{\eps}$ the character $1 \otimes \xi$ is an eigencharacter of $(Y(P_o), \sigma_{P_o})$, and \[\Stab_V(1\otimes \xi) = \{0\}.\]
\item If $\xi \in  \Lambda_{Z,H}^{\eps}$ and $\psi_\xi$ is an eigenfunction for $\xi$, then the corresponding rescaled $\psi_\xi$-twisted Siegel-Radon transform extends to a unitary intertwiner
\[
\mathrm{Zak}_{\psi_\xi} := \sigma_{P_o}(Y(P_o))^{-1/2} \cdot S_{\psi_{\xi}}: \mathrm{Ind}_H^G(\xi) \to L^2(X(P_o), \mu_{P_o}).
\]
In particular, the Schr\"odinger representation with central character $\xi$ appears discretely in $L^2(X(P_o), \mu_{P_o})$ for every $\xi \in  \Lambda_{Z,H}^{\eps}$.
\end{enumerate}
\end{theorem}
\begin{proof} (i) By the main result of \cite{MachadoNilpotent}, $\Lambda$ is a relatively dense subset of a model set $\Lambda'$, and by enlarging the window we may assume that $\Lambda'$ is actually a \emph{regular} model set. Then the intersection $\Lambda' \cap H$ is still discrete, and since it contains $\Lambda \cap H$ it is of finite covolume, i.e. $(H, \Lambda')$ is a compatible pair. By \cite[Thm. 1.12]{BHAL}, the compatible pair $(H, \Lambda)$ is automatically uniform, hence it follows from Theorem \ref{HullInt} that $(X(\Lambda), Y(\Lambda), \mu_{\Lambda})$ is square-integrable. The claim then follows from Proposition \ref{ModelSetJoining}.

\item (ii) If $\xi \in \Lambda_{Z,H}^{\eps}$, then $1\otimes \xi \in \Lambda_H^\eps$, hence $1\otimes \xi$ is an eigencharacter of $(Y, \sigma)$ by Proposition \ref{EigenMeyer}. It follows from \eqref{HeisVaction} and the fact that the inner product $\langle \cdot, \cdot \rangle$ is non-degenerate that $\Stab_V(1\otimes \xi) = \{0\}$.

\item (iii) This follows from (ii) and Theorem \ref{ExtensionAbstract}.
\end{proof}
\begin{definition} The unitary intertwiner $\mathrm{Zak}_{\psi_\xi} : \mathrm{Ind}_H^G(\xi) \to L^2(X(P_o), \mu_{P_o})$ is called the \emph{aperiodic Zak transform} associated with $\psi_\xi$.
\end{definition}
\begin{no}\label{har}
In order to apply the theorem we need to produce $\Lambda$ and $P_o$ as in the theorem such that the set $\Lambda_{Z, H}^\eps$ is large. For this we are going to use a classical theorem of Meyer \cite[Chapter 2]{Meyer1}, which says that every approximate lattice in an abelian locally compact group is harmonious in the sense that its $\eps$-dual is relatively dense for every $\eps \in (0,1)$.
\end{no}
\begin{corollary} Let $\Lambda_U$, $\Lambda_V$ and $\Lambda_Z$ be approximate lattices in $U$, $V$ and $Z$ with $\langle \Lambda_U, \Lambda_V \rangle \subset \Lambda_Z$ and let $P_o \subset  \Lambda_U \times \Lambda_Z \times \Lambda_V$ be a subset of positive density. Then the following hold for every $\eps \in (0,1)$:
\begin{enumerate}[(i)]
\item The hull $X(P_o)$ admits a $G$-invariant probability measure $\mu_{P_o}$.
\item $(X(P_o), Y(P_o), \mu_{P_o})$ is a square-integrable $(G,H)$-transverse system, hence gives rise to an $H$-transverse measure $\sigma_{P_o}$.
\item If $\xi \in \Lambda_Z^\eps \setminus \{\mathbf{1}\}$, then $1 \otimes \xi$ is an eigencharacter of $(Y(P_o), \sigma_{P_o})$, and the associated aperiodic Zak transform embeds the Schr\"odinger representation $\pi_\xi$ discretely into $L^2(X(P_o), \mu_{P_o})$.
\item The subset $\Lambda_Z^\eps \subset \widehat{Z}$ is relatively dense.
\end{enumerate} 
In particular, a relatively dense set of Schr\"odinger representations embeds into $L^2(X(P_o), \mu_{P_o})$.
\end{corollary}
\begin{proof} The assumption $\langle \Lambda_U, \Lambda_V \rangle \subset \Lambda_Z$ ensures that $\Lambda := \Lambda_U \times \Lambda_Z \times \Lambda_V$ is an approximate lattice. By definition we have $\Lambda \cap H = \Lambda_U \times \Lambda_Z$, which is an approximate lattice in $H$, hence $(H, \Lambda)$ is a compatible pair. Finally, $P_o$ is measured by Theorem \ref{Densitymeasure} and $\Lambda_{Z,H} = \Lambda_Z$. Then the corollary follows from Theorem \ref{Zakmain} and \S\ref{har}.
\end{proof}
This establishes Theorem \ref{ZakIntro} from the introduction.

\end{document}